\documentclass{article}
\usepackage[left=3cm, right=3cm, top=2cm]{geometry}
\usepackage[export]{adjustbox}
\usepackage{amsmath}
\usepackage{amsfonts}
\usepackage{graphicx, subfigure}
\usepackage[pdftex]{hyperref}
\usepackage[dutch,english]{babel}
\usepackage{amsmath,amssymb}
\usepackage{amsthm}
\usepackage[usenames,dvipsnames]{color}
\usepackage{tikz}
\usepackage{accents}
\usepackage{xcolor}
\usepackage{hyperref}
\usepackage{mathtools}
\usepackage{float}
\usepackage{multicol}
\usepackage{comment}
\usepackage[T1]{fontenc}
\usepackage{ stmaryrd }
\usepackage[font={small,it}]{caption}
\usetikzlibrary{decorations.pathmorphing}

\newcommand{\la}{\lambda}

\DeclareMathOperator{\nil}{Nil}
\DeclareMathOperator{\im}{Im}

\DeclareMathOperator{\Hom}{Hom}

\DeclareMathOperator{\End}{End}

\DeclareMathOperator{\lin}{Lin}
\DeclareMathOperator{\tr}{Tr}
\DeclareMathOperator{\btr}{btr}

\DeclareMathOperator{\Id}{Id}

\definecolor{bloesh}{RGB}{160, 194, 185}
\definecolor{roo}{RGB}{241,23,70}

\newcommand{\bl}[1]{{\color{blue}#1}}

\newcommand{\geel}[1]{{\color{yellow}#1}}
\newcommand{\pa}[1]{{\color{purple}#1}}
\newcommand{\gr}[1]{{\color{green}#1}}

\newcommand{\roo}[1]{{\color{roo}#1}}
\newcommand{\gre}[1]{{\color{gray}#1}}

\newcommand{\ora}[1]{{\color{orange}#1}}

\def\tr{\mathop{\rm tr}\nolimits}

\newfont{\gothic}{eufm10}

\theoremstyle{plain}
\newtheorem{thr}{Theorem}[section]

\newtheorem{lem}[thr]{Lemma}
\newtheorem{cor}[thr]{Corollary}
\newtheorem{prop}[thr]{Proposition}
\theoremstyle{definition}
\newtheorem{defi}[thr]{Definition}
\newtheorem{ex}[thr]{Example}
\theoremstyle{remark}
\newtheorem{remk}[thr]{Remark}
\theoremstyle{remark}

\newcommand{\field}[1]{\mathbb{#1}}
\newcommand{\R}{\field{R}}
\newcommand{\N}{\field{N}}
\newcommand{\C}{\field{C}}
\newcommand{\Z}{\field{Z}}

\definecolor{bloesh}{RGB}{160, 194, 185}
\definecolor{roo}{RGB}{241,23,70}

\title{Circulant Type Formulas for the Eigenvalues of Linear Network Maps}
\author{Eddie Nijholt\footnote{\mbox{Department of Mathematics, University of Illinois, \href{mailto:eddie.nijholt@gmail.com}{eddie.nijholt@gmail.com} }  }, Lee DeVille\footnote{\mbox{Department of Mathematics, University of Illinois, \href{mailto:rdeville@illinois.edu}{rdeville@illinois.edu} }} }
\date{\today}

\begin{document}

\maketitle

\begin{abstract} \label{abstract}
Given an admissible map $\gamma_f$ for a homogeneous network $\mathcal{N}$, it is known that the Jacobian $D{\gamma_f}(x)$ around a fully synchronous point $x = (x_0, \dots, x_0)$ is again an admissible map for $\mathcal{N}$. Motivated by this, we study the spectra of linear admissible maps for homogeneous networks. In particular, we define so-called network multipliers. These are (relatively small) matrices that depend linearly on the coefficients of the response function, and whose eigenvalues together make up the spectrum of the corresponding admissible map. More precisely, given a network $\mathcal{N}$, we define a finite set of network multipliers $(\Lambda^l)_{l=1}^k$ and a class of networks $\mathfrak{C}$ containing $\mathcal{N}$. This class is furthermore closed under taking quotient networks, subnetworks and disjoint unions. We then show that the eigenvalues of an admissible map for any network in $\mathfrak{C}$ are given by those of (a subset of) the network multipliers, with fixed multiplicities $(m_l)_{l=1}^k$ and independently of the given (finite dimensional) phase space of a node. The coefficients of all the network multipliers of $\mathfrak{C}$ are furthermore linearly independent, which implies that one may find the multiplicities $(m_l)_{l=1}^k$ by simply expressing the trace of an admissible map as a linear combination of the traces of the multipliers. In particular, we will give examples of networks where the network multipliers need not be constructed, but can be determined by looking at small networks in $\mathfrak{C}$. We also show that network multipliers are multiplicative with respect to composition of linear maps.
\end{abstract}

%%%%%%%%%%%%%%%%%%%%%%%%%%%%%%%%%%%Introduction%%%%%%%%%%%%%%%%%%%%%%%%%%%%%%%%%%%

\section{Introduction} \label{Introduction}
Networks play an important role in many of the sciences. They are used to understand the brain \cite{ermentrout, sporns}, the spread of a disease \cite{Tiago} or the activity on social media \cite{twit}, among many other examples. Unsurprisingly, network dynamical systems have garnered a great deal of attention in recent years, see for example \cite{simplereal1, simplereal2, pikovsky, torok, field1, field2}. Even though the networks involved in for example the brain, the internet or systems biology are typically of a huge size, networks of just a few cells have been found to exhibit surprisingly complex dynamical behaviour as well, \cite{Leite, anto4, cen, jeroen, dias, elmhirst, krupa, synbreak2, curious, leite2}. What is more, relatively small networks have been applied successfully to describe a wide range of phenomena, such as animal locomotion \cite{loco}, binocular rivalry \cite{Binocular} and homeostasis \cite{Homeostasis}. Smaller networks may also be used to understand larger ones through their role as motifs. A good example of this is the class of so-called feedforward networks. They are known to function as amplifiers or filters \cite{claire, RinkSanders2}, and have been found to retain this function, at least in a theoretical setting, when `pasted to' other networks \cite{proj}. \\
In this article we present a technique that greatly reduces the cost of finding the eigenvalues of a linear network map. This technique also works for families of admissible maps, and we will give examples of networks for which it gives the eigenvalues as linear functions of the coefficients of the admissible map. Moreover, it will be clear that when this technique does not yield the eigenvalues explicitly, then the algebra of admissible maps contains a matrix algebra. In particular, this implies that no linear expressions of the eigenvalues exist. In that case, our technique reduces the problem of finding the spectrum of an admissible map to the same problem for a set of smaller matrices without additional structure. More precisely, this is under the assumption that the admissible maps of the network indeed form an algebra. It will be clear that this may always be assumed after adding additional arrows to the network. \\
The results in this article are based on the construction of the so-called fundamental network and on techniques from representation theory. To keep the article as accessible as possible however, we have postponed the more technical proofs to the last sections.

\subsection{Circulant Matrices}
The aim of this article is to generalise results for so-called \textit{(block) circulant matrices}. In technical terms, a circulant matrix is an $n \times n$ (block) matrix $\mathcal{A} = (A_{i,j})$ such that $A_{i,j} = A_{k,l}$ whenever $i-j = k-l$ modulo $n$. More visually, this means that every row is obtained from the one above it by cyclically shifting all entries one place to the right (meaning that the entry that was on the far right is now placed on the far left). In particular, for increasing $n$ these matrices look like
\begin{equation} \nonumber
\begin{pmatrix}
A 
\end{pmatrix} , \,
\begin{pmatrix}
A & B \\
B & A \\
\end{pmatrix} , \,
\begin{pmatrix}
A & B & C \\
C & A & B \\
B & C & A \\
\end{pmatrix} , \,
\begin{pmatrix}
A & B & C & D\\
D & A & B & C \\
C & D & A &B \\
B & C & D & A  \\
\end{pmatrix} \, \dots \, ,
\end{equation}
where the $A, B, C$ etc. are elements in $\R$ or $\C$, or more generally $m \times m$ complex matrices. It is well known how to relate the eigenvalues of an $n \times n$ circulant matrix to those of its blocks, or rather to those of certain complex linear combinations of its blocks. To this end, let $\omega$ denote $e^{2\pi i/n}$, so that $\omega^n = 1$. Let $\mathcal{A}$ furthermore be the $n \times n$ circulant matrix given by
\begin{equation}\label{circform}
\mathcal{A} = \begin{pmatrix*}[l]
A_n & A_{n-1} & \dots & A_{2} & A_1\\
A_1 & A_n & \dots & A_{3} & A_{2}\\
 &  \ddots &\ddots &\ddots & \\
A_{n-2} & A_{n-3} & \dots & A_{n} & A_{n-1}\\
A_{n-1} & A_{n-2} & \dots & A_{1} & A_{n}\\
\end{pmatrix*}  \, ,
\end{equation}
with $A_1, \dots, A_n$ complex $m \times m$ matrices. We will interpret the index $k$ of $A_k$ as an element in $\Z/n\Z$, so that we may write $A_{i,j} = A_{i-j}$. Lastly, for $k \in \{1, \dots, n\}$ we define a linear map $\Lambda^k$ from the space of $n \times n$ circulant matrices with $m \times m$ blocks to the space of $m \times m$ matrices. Denoting a circulant matrix $\mathcal{A}$ as in \eqref{circform}, $\Lambda^k$ is given by
\begin{align}\label{lambdcirc}
\Lambda^k(\mathcal{A}) = \sum_{j=1}^n \omega^{jk}A_j\,.
\end{align}
Again, the index $k$ in $\Lambda^k$ may safely be taken from $\Z/n\Z$, as does the summation index $j$ in \eqref{lambdcirc}. Using this notation, we wish to generalise the following result to linear network maps.
\begin{prop}\label{main0}
A complex number $\la \in \C$ is an eigenvalue of the $n \times n$ circulant matrix $\mathcal{A}$, if and only if it is an eigenvalue of one of the $m \times m$ matrices $\Lambda^k(\mathcal{A})$.
\end{prop}
\noindent Proposition \ref{main0} is a straightforward generalisation of results in \cite{davis1994circulant} and \cite{circulant}. As an example, we have that the eigenvalues of a matrix of the form
\begin{equation} \nonumber
\begin{pmatrix}
A & B \\
B & A \\
\end{pmatrix} 
\end{equation}
are given by those of $A+B$ together with those of $A-B$. Likewise, the eigenvalues of
\begin{equation}
\begin{pmatrix}
A & B & C \\
C & A & B \\
B & C & A \\
\end{pmatrix} , 
\end{equation}
are given by those of $A+B+C$, $A+e^{2\pi i/3}B + e^{4\pi i/3}C$ and $A + e^{4\pi i/3}B + e^{2\pi i/3}C$. Note that Proposition \ref{main0} just gives the eigenvalues of $\mathcal{A}$ in the case of $m = 1$. However, for $m > 1$ this result still yields a significant reduction in the computational cost of calculating or estimating the eigenvalues of $\mathcal{A}$, or of deducing certain properties of them. Moreover, as the linear maps $\Lambda^k$ do not depend on $\mathcal{A}$ themselves (but rather have $\mathcal{A}$ as their input), the result of Proposition \ref{main0} can be applied to families of circulant matrices as well. Likewise, the size $m$ of the blocks of $\mathcal{A}$ plays no essential role, as it does not appear in the formal expression \eqref{lambdcirc} for $\Lambda^k$.

\subsection{Circulant Matrices as Network Maps}
Next, we would like to point out that circulant matrices may be interpreted as linear maps respecting a certain network structure. For example, consider the $4$-cell ring network given by the left side of Figure \ref{fig1}. Every node in this network may be interpreted as some quantity evolving in time, whereas the arrows depict how these quantities influence each other. More precisely, each of the four arrow types (the red ones, the black self-loops, the green double-headed ones and the blue dashed ones) represents a particular type of interaction. For this statement to make sense, we need that the states of the individual nodes are somehow comparable. Hence, to each node $p \in \{1, \dots, 4\}$ we associate a variable $X_p$ taking values in some same phase space $V$. We moreover pick one \textit{response function} $f: V^4 \rightarrow V$ such that every slot of $f$ denotes a single type of input (and so a single arrow colour). This defines an \textit{admissible map} or, given that $V$ is a linear space, an \textit{admissible vector field} $\gamma_f$ of the network on the left of Figure \ref{fig1}. It is given by
\begin{equation}\label{fullformcirc4}
\arraycolsep=1.4pt\def\arraystretch{1}
\gamma_f(X_1, \dots, X_4) = \begin{array}{llcllr}
f(X_1,\roo{X_2}, \gr{X_{3}}, \bl{X_4}) \\
f(X_2,\roo{X_3}, \gr{X_{4}}, \bl{X_1}) \\
f(X_3,\roo{X_4}, \gr{X_{1}}, \bl{X_2}) \\
f(X_4,\roo{X_1}, \gr{X_{2}}, \bl{X_3}) \\
\end{array}\, .
\end{equation}
Note, for example, that there is a red arrow from node $2$ to node $1$ in the left network of Figure \ref{fig1}, so that node $1$ receives input of the `red-type'  from node $2$. Therefore, equation \eqref{fullformcirc4} shows an $X_2$ (depicted in red) in the second slot of $f$ in the first component of $\gamma_f$. \\
Now let us assume that the response function $f$ is linear. This would for example be the case if we looked at the Jacobian of $\gamma_f$ around a point $X = (X_1, \dots, X_4) \in V^4$ satisfying $X_1 = \dots = X_4$. Writing 
\begin{equation}
f(X,Y, Z,W) = AX+\roo{B}Y+\gr{C}Z+\bl{D}W\, ,
\end{equation}
equation \eqref{fullformcirc4} becomes
\begin{equation}\label{fullformcirc4lin}
\arraycolsep=1.4pt\def\arraystretch{1}
\gamma_f(X_1, \dots, X_4) = \begin{bmatrix}
A & \roo{B} & \gr{C} & \bl{D}\\
\bl{D} & A & \roo{B} & \gr{C} \\
\gr{C} & \bl{D} & A &\roo{B} \\
\roo{B} & \gr{C} & \bl{D} & A  \\
\end{bmatrix} \begin{bmatrix}
X_1\\
X_2\\
X_3\\
X_4\\
\end{bmatrix}\, .
\end{equation}
Hence, we exactly find all $4 \times 4$ circulant matrices as the linear admissible maps of the left network of Figure \ref{fig1}. \\
\begin{figure}
\centering
\begin{tikzpicture}
	\node[circle,draw=black, fill=bloesh, fill opacity = 1, inner sep=1pt, minimum size=12pt] (n) at (3.95,-1.05) {n};
	\node[circle,draw=black, fill=bloesh, fill opacity = 1, inner sep=1pt, minimum size=12pt] (1) at (5,0) {1};
	\node[circle,draw=black, fill=bloesh, fill opacity = 1, inner sep=1pt, minimum size=12pt] (2) at (6.5,0) {2};
	\node[circle,draw=black, fill=bloesh, fill opacity = 1, inner sep=1pt, minimum size=12pt] (3) at (7.55,-1.05) {3};
	\node[circle,draw=black, fill=bloesh, fill opacity = 1, inner sep=0pt, minimum size=12pt, ] (ip) at (5,-3) {\tiny i+1};
	\node[circle,draw=black, fill=bloesh, fill opacity = 1, inner sep=1pt, minimum size=12pt] (i) at (6.5,-3) {i};
	\node[] (t) at (7.13,-1.82) {.};
	\node[] (t1) at (7.02,-2.02) {.};
	\node[] (t2) at (6.91,-2.22) {.};
	\node[] (t3) at (4.37,-1.82) {.};
	\node[] (t4) at (4.48,-2.02) {.};
	\node[] (t5) at (4.59,-2.22) {.};
	\draw [->, >=stealth, thick, roo, shorten <=2pt, shorten >=2pt] (1) to [bend right = 20] (n);
	\draw [->,  >=stealth, thick, roo, shorten <=2pt, shorten >=2pt] (2) to [bend right = 20] (1);
	\draw [->,  >=stealth, thick, roo, shorten <=2pt, shorten >=2pt] (3) to [bend right = 20] (2);
	\draw [->,  >=stealth, thick, roo, shorten <=2pt, shorten >=2pt] (ip) to [bend right = 20] (i);
	\node[circle,draw=black, fill=bloesh, fill opacity = 1, inner sep=1pt, minimum size=12pt] (1v) at (0,-0.5) {1};
	\node[circle,draw=black, fill=bloesh, fill opacity = 1, inner sep=1pt, minimum size=12pt] (2v) at (2,-0.5) {2};
	\node[circle,draw=black, fill=bloesh, fill opacity = 1, inner sep=1pt, minimum size=12pt] (3v) at (2,-2.5) {3};
	\node[circle,draw=black, fill=bloesh, fill opacity = 1, inner sep=1pt, minimum size=12pt] (4v) at (0,-2.5) {4};
	\draw [->,  >=stealth, thick, roo, shorten <=2pt, shorten >=2pt] (2v) to [bend right = 30] (1v);
	\draw [->,  >=stealth, thick, roo, shorten <=2pt, shorten >=2pt] (3v) to [bend right = 30] (2v);
	\draw [->,  >=stealth, thick, roo, shorten <=2pt, shorten >=2pt] (4v) to [bend right = 30] (3v);
	\draw [->,  >=stealth, thick, roo, shorten <=2pt, shorten >=2pt] (1v) to [bend right = 30] (4v);
	\draw [->,  >=stealth, dashed, thick, blue, shorten <=2pt, shorten >=2pt] (1v) to [bend left = 10] (2v);
	\draw [->,  >=stealth, dashed, thick, blue, shorten <=2pt, shorten >=2pt] (2v) to [bend left = 10] (3v);
	\draw [->,  >=stealth, dashed, thick, blue, shorten <=2pt, shorten >=2pt] (3v) to [bend left = 10] (4v);
	\draw [->,  >=stealth, dashed, thick, blue, shorten <=2pt, shorten >=2pt] (4v) to [bend left = 10] (1v);
	\draw [->>,  >=stealth, thick, green, shorten <=2pt, shorten >=2pt] (3v) to [bend left = 20] (1v);
	\draw [->>,  >=stealth, thick, green, shorten <=2pt, shorten >=2pt] (1v) to [bend left = 20] (3v);
	\draw [->>,  >=stealth, thick, green, shorten <=2pt, shorten >=2pt] (2v) to [bend left = 20] (4v);
	\draw [->>,  >=stealth, thick, green, shorten <=2pt, shorten >=2pt] (4v) to [bend left = 20] (2v);
	\draw [->,  >=stealth, thick, black, loop above, shorten <=2pt, shorten >=2pt] (2v) to [out=25,in=65,looseness=10]( 2v);
	\draw [->,  >=stealth, thick, black, loop above, shorten <=2pt, shorten >=2pt] (1v) to [out=115,in=155,looseness=10]( 1v);
	\draw [->,  >=stealth, thick, black, loop above, shorten <=2pt, shorten >=2pt] (3v) to [out=295,in=335,looseness=10]( 3v);
	\draw [->,  >=stealth, thick, black, loop above, shorten <=2pt, shorten >=2pt] (4v) to [out=205,in=245,looseness=10]( 4v);
\end{tikzpicture}
\caption{Left: a ring coupled cell network with $4$ nodes and with all arrows shown. \\ Right: a ring coupled cell network with $n$ nodes. Shown here is only one type of arrow in red. The other types are obtained by following the given arrows two times, three times, and so forth, up to $n$ times (which corresponds to all self-loops). }
\label{fig1}
\end{figure}
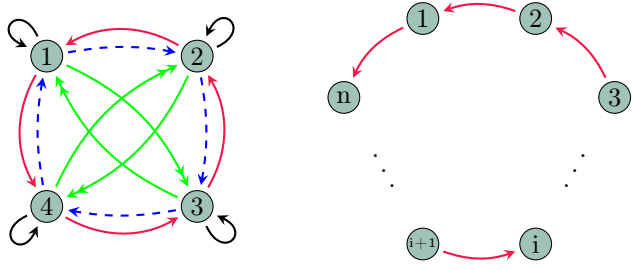
The right side of Figure \ref{fig1} shows a general ring coupled cell network. This graph consists of $n$ nodes with $n$ types of arrows, where there are furthermore $n$ arrows of each type. The first type consists of self-loops for every node, and is not shown on the right side of Figure \ref{fig1}. The second type is shown as the red arrows. The next type consists of all arrows one would obtain by following the red arrows forward twice, i.e. by concatenating two consecutive red arrows. The type after that is obtained by following the red arrows three times, and so forth. Note that following the red arrows $n$ times corresponds to all self-loops, and that following the red arrows $n+1$ times just gives back the red arrows. It can again be seen that the linear admissible maps for this network are exactly the $n \times n$ circulant matrices. \\
A natural question to ask is if we can generalise Proposition \ref{main0} from circulant matrices to linear admissible maps for any given network structure. The answer is affirmative for many examples of networks, and it holds true for all (finite, homogeneous) networks if we allow the $\Lambda_k$ to be matrices of several times the dimension of the phase space of a single cell (but still in general smaller than the total size of the network). \\
More precisely, we will prove Theorem \ref{main1} below. This result is on \textit{homogeneous networks with asymmetric input}. That is, networks in which every cell is targeted by exactly one arrow of every type. In addition, we require that the network is \textit{complete}. This means that for every ordered pair of arrow types $(c,d)$, there exists an arrow type $e$ such that the following holds: Tracing an arrow of type $c$ back from its target to its source and then following an arrow of type $d$ back to its source always amounts to directly following an arrow of type $e$ back. Moreover, we will always assume that a complete network has an arrow type consisting of all self loops. The networks of Figure \ref{fig1} are both examples of complete homogenous networks with asymmetric input (if one adds the arrow types not shown on the right). Every homogeneous network with asymmetric input can be made complete by adding more arrow types, see Section \ref{Homogeneous Networks and the Fundamental Network} for more details.

\begin{thr}\label{main1}
Let $\mathcal{N}$ be a complete homogeneous network with asymmetric input and with $t$ types of arrow. There exists a class of complete homogeneous networks with asymmetric input, denoted by $\mathfrak{C}_{\mathcal{N}}$, and formal linear maps
\begin{equation}
\Lambda^l_{i,j}(x_1, \dots, x_t) = \sum_{s=1}^t a_{i,j}^{l,s} \cdot x_s\, ,
\end{equation}
with $l \in \{1, \dots, k\}$ for some $k \geq 1$, $i,j \in \{1, \dots, n_l\}$ for some $n_l \geq 1$ and with $a_{i,j}^{l,s} \in \C$, such that the following holds:
\begin{enumerate}
\item $\mathfrak{C}_{\mathcal{N}}$ contains $\mathcal{N}$ and is closed under taking subnetworks, quotient networks and disjoint unions. Moreover, every network in $\mathfrak{C}_{\mathcal{N}}$ has the same arrow types as $\mathcal{N}$, so that a single response function can be used to describe admissible maps for all networks in $\mathfrak{C}_{\mathcal{N}}$ simultaneously. 
\item The linear maps $\Lambda^l_{i,j}$, for $l \in \{1, \dots, k\}$ and $i,j \in \{1, \dots, n_l\}$, are all linearly independent. In particular, we have $n_1^2 + n_2^2 + \dots + n_k^2 \leq t$.
\item $n_1 = 1$ and $\Lambda^1_{1,1}(x_1, \dots, x_t) = x_1 + x_2 + \dots + x_t$.
\item Let $\mathcal{M}$ be a network in $\mathfrak{C}_{\mathcal{N}}$. If every node corresponds to a variable in a finite dimensional linear phase space $V$, and we have a linear response function $f: V^t \rightarrow V$ given by
\begin{equation}
f(X_1, \dots, X_t) = \sum_{s=1}^t C_s X_s\, ,
\end{equation}
then the eigenvalues of the corresponding admissible linear map $\gamma^{\mathcal{M}}_f$ are given exactly by those of the $n^l \times n^l$ block matrices
\begin{equation}\label{repr1}
\Lambda^l(C_1, \dots, C_t) := \begin{pmatrix*}[l]
\ddots & & \reflectbox{$\ddots$} \\
 &  \Lambda^l_{i,j}(C_1, \dots, C_t) & \\
\reflectbox{$\ddots$}  & & \ddots \\
\end{pmatrix*}  \, 
\end{equation}
together, for all $l$ in some subset of $\{1, \dots, k\}$.\\
More precisely, for a linear map $X$ let $M_\la(X)$ denote the algebraic multiplicity of an eigenvalue $\la$. (where $M_\la(X):= 0$  if $\la$ is not an eigenvalue of $X$). There exist non-negative integers $m^{\mathcal{M}}_1 \dots, m^{\mathcal{M}}_k$, independent of $V$ or $f$, such that 
\begin{equation}
M_\la(\gamma^{\mathcal{M}}_f) = \sum_{l=1}^k m^{\mathcal{M}}_l M_\la(\Lambda^l(C_1, \dots, C_t))\, .
\end{equation}
In other words, the eigenvalues of $\gamma^{\mathcal{M}}_f$ are given by $m^{\mathcal{M}}_1$ times those of $\Lambda^1(C_1, \dots, C_t)$, together with $m^{\mathcal{M}}_2$ times those of $\Lambda^2(C_1, \dots, C_t)$ up to $m^{\mathcal{M}}_k$ times those of $\Lambda^k(C_1, \dots, C_t)$. These multiplicities $m^{\mathcal{M}}_l$ may be obtained by looking at the trace of $\gamma^{\mathcal{M}}_f$, using the fact that the maps $\Lambda^l_{i,j}$ are linearly independent. The linear maps $\Lambda^l_{i,j}$ may be obtained by investigating the representation theory of the so-called \textit{fundamental network}, which will be made more precise later on. However, in many cases the $ \Lambda^l$ can be obtained directly by investigating relatively small networks in $\mathfrak{C}_{\mathcal{N}}$.
\item For any network $\mathcal{M} \in \mathfrak{C}_{\mathcal{N}}$ with $M>0$ nodes we have
\begin{equation}\label{sumhuh1}
\sum_{l=1}^k m^{\mathcal{M}}_l n_l = M\, .
\end{equation}
and $m^{\mathcal{M}}_1 \geq 1$. Moreover, if $\mathcal{P}$ is a quotient network of $\mathcal{M}$, then we have $m^{\mathcal{P}}_l \leq m^{\mathcal{M}}_l$ for all $l \in \{1, \dots, k\}$.
\end{enumerate}
\end{thr} 
\begin{defi}
The class of networks $\mathfrak{C}_{\mathcal{N}}$ of Theorem \ref{main1} will be called the class of \textit{constructible networks} of ${\mathcal{N}}$. This will be precisely defined in Section \ref{Input Networks and Constructible Networks}, where it is also shown that $\mathfrak{C}_{\mathcal{N}}$ contains the so-called \textit{fundamental network} of $\mathcal{N}$. See Section \ref{Homogeneous Networks and the Fundamental Network} or \cite{fibr,  CCN,  RinkSanders3} for more on the fundamental network. We call the formal $n_l \times n_l$ block matrices $\Lambda^l$ the \textit{network multipliers} of $\mathfrak{C}_{\mathcal{N}}$. $\hfill \triangle$
\end{defi}
\noindent Before we address the specifics, we illustrate Theorem \ref{main1} with an example.
\begin{ex}
Consider the complete homogeneous network with asymmetric input $\mathcal{N}$ depicted on the left of Figure \ref{fig2x}. Writing a response function $f: V^6 \rightarrow V$ as 
\begin{equation}
f(X_1, \dots, X_6) = AX_1 +  \roo{B}X_2 +  \gr{C}X_3  + \bl{D}X_4+ \ora{E}X_5  + \pa{F} X_6\, ,
\end{equation}
we get the admissible map
\begin{equation}\label{gammafex142we2r}
\arraycolsep=1.4pt\def\arraystretch{1}
{\gamma}^{\mathcal{N}}_f = \begin{bmatrix}
A +  \roo{B}&   \gr{C}  &\bl{D}+ \ora{E}  &\pa{F}  \\
\roo{B} & A+  \gr{C} &\ora{E} & \bl{D} + \pa{F}   \\
\roo{B} &  \gr{C}  & A + \ora{E}  &\bl{D} + \pa{F}   \\
\roo{B} &  \gr{C}  &\ora{E}  &  A + \bl{D}  +\pa{F} 
\end{bmatrix} \, .
\end{equation}
The right side of Figure \ref{fig2x} shows a quotient network of $\mathcal{N}$, obtained by identifying cells $1$ and $2$, and at the same time cells $3$ and $4$. We will call this quotient network $\mathcal{M}$, and its corresponding admissible maps are given by
\begin{equation}\label{gammafex142we2rr}
\arraycolsep=1.4pt\def\arraystretch{1}
{\gamma}^{\mathcal{M}}_f = \begin{bmatrix}
A +  \roo{B} +  \gr{C}  &\bl{D}+ \ora{E}  + \pa{F}  \\
\roo{B} +  \gr{C}  & A + \ora{E}  + \bl{D} + \pa{F}   
\end{bmatrix} \, .
\end{equation}
By Theorem \ref{main1}, point 1, $\mathcal{M}$ belongs to the class of constructible networks $\mathfrak{C}_{\mathcal{N}}$. Therefore, we can get information about the network multipliers of $\mathfrak{C}_{\mathcal{N}}$ by looking at ${\gamma}^{\mathcal{M}}_f$. More precisely, equation \eqref{sumhuh1} in point 5 tells us that
\begin{equation}\label{sumhuh12}
\sum_{l=1}^k m^{\mathcal{M}}_l n_l = 2\, .
\end{equation}
Moreover, by point 3 we have $n_1 = 1$ and point 5 tells us that $m^{\mathcal{M}}_1 \geq 1$. Hence, we conclude that either $m^{\mathcal{M}}_1 = 2$, or $m^{\mathcal{M}}_1 = 1$ with  $m^{\mathcal{M}}_2 = 1$ and $n_2 = 1$ (after numbering the network multipliers accordingly). We conclude by point 4 that either $\tr({\gamma}^{\mathcal{M}}_f) = 2\Lambda^1$ or $\tr({\gamma}^{\mathcal{M}}_f) = \Lambda^1 + \Lambda^2$. (Here we assume $V = \C$, or we may interpret $\tr(\bullet)$ as the sum of the diagonal blocks.) A direct calculation shows that
\begin{equation}
\tr({\gamma}^{\mathcal{M}}_f) - \Lambda^1 = (2A +  \roo{B} +  \gr{C}  + \ora{E}  + \bl{D} + \pa{F} ) - (A +  \roo{B} +  \gr{C}  + \ora{E}  + \bl{D} + \pa{F}) = A\,.
\end{equation}
Therefore, we find $\Lambda^2(A, \dots,  \pa{F} )= A$. In particular, we have already found two network multipliers: $\Lambda^1(A, \dots,  \pa{F} ) = A +  \roo{B} +  \gr{C}  + \ora{E}  + \bl{D} + \pa{F}$ and $\Lambda^2(A, \dots,  \pa{F} )= A$. If we now look at our original network $\mathcal{N}$, we see that
\begin{align}\label{sumhuh3}
\tr({\gamma}^{\mathcal{N}}_f) &= 4A +  \roo{B} +  \gr{C}  + \ora{E}  + \bl{D} + \pa{F} = (A +  \roo{B} +  \gr{C}  + \ora{E}  + \bl{D} + \pa{F} ) + 3A\\
 &= \Lambda^1(A, \dots,  \pa{F} ) + 3\Lambda^2(A, \dots,  \pa{F} ) \nonumber \,.
\end{align}
Because the coefficients of all the network multipliers are linearly independent by point 2, we conclude that equation \eqref{sumhuh3} is the unique expression of $\tr({\gamma}^{\mathcal{N}}_f)$ as the sum of (the traces of) the network multipliers of $\mathfrak{C}_{\mathcal{N}}$. We conclude that $m^{\mathcal{N}}_1 = 1$ and $m^{\mathcal{N}}_2 = 3$. In particular, for any choice of matrices $A$ to $\pa{F}$ of any size $m$, the eigenvalues of the $4m \times 4m$ matrix \eqref{gammafex142we2r} are given by one time those of $A +  \roo{B} +  \gr{C}  + \ora{E}  + \bl{D} + \pa{F}$ together with three times those of $A$ (counting algebraic multiplicity).
$\hfill \triangle$
\end{ex}
\begin{figure}
\centering
\begin{tikzpicture}
	\node[circle,draw=black, fill=bloesh, fill opacity = 1, inner sep=1pt, minimum size=12pt] (1) at (7,-0.75) {1};
	\node[circle,draw=black, fill=bloesh, fill opacity = 1, inner sep=1pt, minimum size=12pt] (2) at (7,-3.75) {2};
	\draw [->,  >=stealth, thick, roo, loop above, shorten <=2pt, shorten >=2pt] (1) to [out=0,in=45,looseness=10](1);
	\draw [->,  >=stealth, thick, roo, shorten <=2pt, shorten >=2pt] (1) to [bend left = 20] (2);
	\draw [->,  >=stealth, dashed, thick, loop above, blue, shorten <=2pt, shorten >=2pt] (2) to  [out=-15,in=25,looseness=10](2);
	\draw [->,  >=stealth, dashed, thick, blue, shorten <=2pt, shorten >=2pt] (2) to [bend right = 0] (1);
	\draw [->>,  >=stealth, thick, green, shorten <=2pt, shorten >=2pt] (1) to [bend left = -20] (2);
	\draw [->>,  >=stealth, thick, green, loop above, shorten  <=2pt, shorten >=2pt] (1) to  [out=70, in=110,looseness=10](1);
	\draw [->,  >=stealth, thick, black, loop above, shorten <=2pt, shorten >=2pt] (2) to [out=145,in=185,looseness=10]( 2);
	\draw [->,  >=stealth, thick, black, loop above, shorten <=2pt, shorten >=2pt] (1) to [out=135,in=175,looseness=10]( 1);
	\draw [->>,  >=stealth, dashed, thick, loop above, orange, shorten <=2pt, shorten >=2pt] (2) to  [out=275,in=315,looseness=10](2);
	\draw [->>,  >=stealth, dashed, thick, orange, shorten <=2pt, shorten >=2pt] (2) to [bend right = -40] (1);
	\draw [->,  >=stealth, dash dot, thick, loop above, purple, shorten <=2pt, shorten >=2pt] (2) to  [out=215,in=255,looseness=10](2);
	\draw [->,  >=stealth, dash dot, thick, purple, shorten <=2pt, shorten >=2pt] (2) to [bend right = 40] (1);
	\node[circle,draw=black, fill=bloesh, fill opacity = 1, inner sep=1pt, minimum size=12pt] (1v) at (0,-0.75) {1};
	\node[circle,draw=black, fill=bloesh, fill opacity = 1, inner sep=1pt, minimum size=12pt] (2v) at (3,-0.75) {2};
	\node[circle,draw=black, fill=bloesh, fill opacity = 1, inner sep=1pt, minimum size=12pt] (3v) at (3,-3.75) {3};
	\node[circle,draw=black, fill=bloesh, fill opacity = 1, inner sep=1pt, minimum size=12pt] (4v) at (0,-3.75) {4};
	\draw [->,  >=stealth, thick, roo, loop above, shorten <=2pt, shorten >=2pt] (1v) to [out=55,in=95,looseness=10]( 1v);
	\draw [->,  >=stealth, thick, roo, shorten <=2pt, shorten >=2pt] (1v) to [bend left = 20] (2v);
	\draw [->,  >=stealth, thick, roo, shorten <=2pt, shorten >=2pt] (1v) to [bend right = 30] (3v);
	\draw [->,  >=stealth, thick, roo, shorten <=2pt, shorten >=2pt] (1v) to [bend right = 30] (4v);
	\draw [->,  >=stealth, dashed, thick, blue, shorten <=2pt, shorten >=2pt] (4v) to [bend left = 20] (2v);
	\draw [->,  >=stealth, dashed, thick, blue, shorten <=2pt, shorten >=2pt] (4v) to [bend right = 10] (3v);
	\draw [->,  >=stealth, dashed, thick, blue, shorten <=2pt, loop above, shorten >=2pt] (4v) to [out=137,in=177,looseness=10] (4v);
	\draw [->,  >=stealth, dashed, thick, blue, shorten <=2pt, shorten >=2pt] (3v) to [bend right = 10] (1v);
	\draw [->>,  >=stealth, thick, green, shorten <=2pt, shorten >=2pt] (2v) to [bend left = 20] (3v);
	\draw [->>,  >=stealth, thick, green, shorten <=2pt, shorten >=2pt] (2v) to [bend left = 10] (1v);
	\draw [->>,  >=stealth, thick, green, shorten <=2pt, loop above, shorten >=2pt] (2v) to  [out=-35,in=5,looseness=10] (2v);
	\draw [->>,  >=stealth, thick, green, shorten <=2pt, shorten >=2pt] (2v) to [bend left = 20] (4v);
	\draw [->,  >=stealth, thick, black, loop above, shorten <=2pt, shorten >=2pt] (2v) to [out=25,in=65,looseness=10]( 2v);
	\draw [->,  >=stealth, thick, black, loop above, shorten <=2pt, shorten >=2pt] (1v) to [out=115,in=155,looseness=10]( 1v);
	\draw [->,  >=stealth, thick, black, loop above, shorten <=2pt, shorten >=2pt] (3v) to [out=295,in=335,looseness=10]( 3v);
	\draw [->,  >=stealth, thick, black, loop above, shorten <=2pt, shorten >=2pt] (4v) to [out=200,in=240,looseness=10]( 4v);
	\draw [->>,  >=stealth, dashed, thick, orange, shorten <=2pt, shorten >=2pt] (3v) to [bend left = 0] (2v);
	\draw [->>,  >=stealth, dashed, thick, orange, shorten <=2pt, shorten >=2pt] (3v) to [bend left =10] (1v);
	\draw [->>,  >=stealth, dashed, thick, orange, shorten <=2pt, loop above, shorten >=2pt] (3v) to  [out= -5,in=35,looseness=10] (3v);
	\draw [->>,  >=stealth, dashed, thick, orange, shorten <=2pt, shorten >=2pt] (3v) to [bend left = 30] (4v);
	\draw [->,  >=stealth, dash dot, thick, purple, shorten <=2pt, shorten >=2pt] (4v) to [bend left = 0] (2v);
	\draw [->,  >=stealth, dash dot, thick, purple, shorten <=2pt, shorten >=2pt] (4v) to [bend left =10] (1v);
	\draw [->,  >=stealth, dash dot, thick, purple, shorten <=2pt, loop above, shorten >=2pt] (4v) to  [out= -100,in=-65,looseness=10] (4v);
	\draw [->,  >=stealth, dash dot, thick, purple, shorten <=2pt, shorten >=2pt] (4v) to [bend right = 50] (3v);
		\draw [->,  >=stealth, decorate, thick, decoration={snake,amplitude= 0.4mm,segment length=4mm,post length=1mm}, purple]  (3.65,-2.5) to node[above] {\begin{tabular}{l} $1,2 \mapsto 1$ \\ $3,4 \mapsto 2$  \end{tabular}}  (6,-2.5);
\end{tikzpicture}
\caption{Left: a complete homogeneous network with asymmetric input. \\ Right: a quotient network of the network on the left, obtained by identifying cells $1$ and $2$, and cells $3$ and $4$. }
\label{fig2x}
\end{figure}
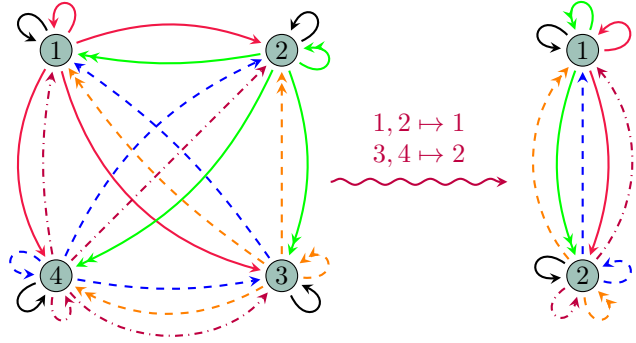

%%%%%%%%%%%%%%%%%%%%%%%%Homogeneous Networks and the Fundamental Network%%%%%%%%%%%%%%%%%%%%%%%%%%%%%

\section{Homogeneous Networks and the Fundamental Network} \label{Homogeneous Networks and the Fundamental Network}
In this section we briefly introduce the definition of a homogeneous coupled cell network with asymmetric input, as well as the fundamental network formalism introduced by Nijholt, Rink and Sanders \cite{fibr,  CCN,  RinkSanders3}. We begin by defining a general coupled cell network, largely following for example \cite{field, golstew, pivato}.
\begin{defi}\label{ccnsgen}
A \textit{coupled cell network} is a directed graph $A \stackrel{\mathclap{\small\mbox{${s,t}$}}}{\rightrightarrows} N$ with finite set of cells or nodes $N$ and finite set of arrows $A$. Moreover, there is an equivalence relation $\sim_N$ on $N$ (usually called \textit{colour}) and an equivalence relation $\sim_A$ on $A$ (usually called \textit{arrow type} or also \textit{colour}), such that the following compatibility conditions are satisfied.
\begin{enumerate}
\item If $a,b \in A$ are such that $a \sim_A b$, then we have $s(a) \sim_N s(b)$ and $t(a) \sim_N t(b)$. In other words, arrows of the same type have sources of the same colour and targets of the same colour.
\item if  $p,q \in N$ are such that $p \sim_N q$, then there exists a bijection $\psi_{p.q}$ from $\mathcal{I}_p := \{a \in A\, \mid t(a) = p\}$ to $\mathcal{I}_q := \{a \in A\, \mid t(a) = q\}$ satisfying $\psi_{p.q}(a) \sim_A a$ for all $a \in \mathcal{I}_p$. In other words, nodes of the same colour are targeted by the same number of arrows for every type.
\end{enumerate}
Note that we allow for self loops, as well as multiple arrows between nodes (even of the same arrow type). $\hfill \triangle$
\end{defi}
\noindent In this article, we will mostly focus on the case where every node is of the same colour and where every node is targeted by precisely one arrow of every type. 
\begin{defi}
A coupled cell network is called \textit{homogeneous} if every node has the same colour. We say that a network has \textit{asymmetric input} if every node receives input from at most one arrow of a given type. In other words, if for every node $p \in N$ we have $a,b \in \mathcal{I}_p := \{a \in A\, \mid t(a) = p\}, a \sim_A b \implies a = b$. $\hfill \triangle$
\end{defi}
\noindent As mentioned before, we will mostly focus on homogeneous coupled cell networks with asymmetric input. The reason for choosing homogeneous networks is mostly convenience; a lot of the same ideas may be applied to non-homogeneous networks, but this becomes notationally far heavier. Our results will also apply when we drop the asymmetric condition, as every such network may be interpreted as a network with asymmetric input (by `pretending' some arrows are not of the same type). Of course, there are in general multiple ways of doing this, and it is not yet clear which way would somehow be best. As of yet, it is unclear how to (canonically) generalise the fundamental network construction below (see Definition \ref{fundamet}) to networks with symmetric input.\\

\noindent Given a homogeneous network with asymmetric input, we may associate to every arrow type an \textit{input map} from the set of nodes $N$ to itself. This works as follows. Since the network has asymmetric input, every node is targeted by at most one arrow of a given type. Moreover, as the network is homogeneous, and by condition 2 of Definition \ref{ccnsgen}, every node is targeted by exactly one arrow of a given type. The input map we associate to a given arrow type is then obtained by following this unique arrow back to its source node. In other words, suppose we consider the arrows of type $d$, and suppose we obtain the input map $\sigma: N \rightarrow N$ in this manner. Then for a node $p$, $\sigma(p)$ equals $s(a)$ for $a$ the unique arrow of type $d$ such that $t(a) = p$. More intuitively, this means that node $p$ `feels' node $\sigma(p)$ through the interaction of type $d$. Following this formalism, we will henceforth denote a homogeneous coupled cell network with asymmetric input by $\mathcal{N} = (N, \mathcal{T})$. Here $N$ is a finite set of nodes, and $\mathcal{T}$ is a set of input maps from $N$ to itself. Note that none of the maps in $\mathcal{T}$ has to be invertible. There is a similar description for coupled cell networks with asymmetric input but with multiple colours of nodes. In that case, every element of $\mathcal{T}$ denotes a map from the set of all nodes of some colour to the set of nodes of some (possibly different) colour.\\

\noindent Given a homogeneous network with asymmetric input $\mathcal{N} =(N, \mathcal{T})$, we may define an \textit{admissable map} for this network. To this end, we fix a space $V$ as the phase space of every individual node. In principle, $V$ can be any set with any structure, but for our ends it will be a finite dimensional vector space. To every node $p \in N$ we then assign a variable $x_p \in V$, so that we get the total phase space
\begin{equation}\label{totfase}
V^N:= \bigoplus_{p \in N} V \, .
\end{equation}
We may also on occasion write $V^{\mathcal{N}}$ for $V^N$, when the set of nodes is not made explicit. Note that $V^N$ is equal to $V^{\#N}:= V \times V \times \dots \times V$ ($\#N$ times) as a vector space. However, we choose the notation of equation \eqref{totfase} so that we may unambiguously write $x = (x_p)_{p \in N}$ for an element $x \in V^N$. Likewise, we define the \textit{input space} of any node by
\begin{equation}\label{infase}
V^I:= \bigoplus_{\sigma \in \mathcal{T}} V \, .
\end{equation}
If we now fix a  \textit{response function} $f: V^{I} \rightarrow V$, then we get the admissible map $\gamma_f: V^N \rightarrow V^N$ given by
\begin{equation}\label{gammaf}
	\begin{array}{ccl}
		(\gamma_f(x))_1 &= &f(x_{\sigma_1(1)}, x_{\sigma_2(1)}, \dots, x_{\sigma_m(1)})  \\
		(\gamma_f(x))_2 &= &f(x_{\sigma_1(2)}, x_{\sigma_2(2)}, \dots, x_{\sigma_m(2)})   \\
		   &\vdots  & \\
		 (\gamma_f(x))_n &= &f(x_{\sigma_1(n)}, x_{\sigma_2(n)}, \dots, x_{\sigma_m(n)}) 
	\end{array} \, .
\end{equation}
Here we have set $N = \{1, \dots, n\}$ and $\mathcal{T} = \{\sigma_1, \dots, \sigma_m\}$ for convenience. We can find a shorter way of writing equation \eqref{gammaf} by defining linear maps $\pi_p: V^N \rightarrow V^I$. These are given by $(\pi_p(x))_{\sigma} = x_{\sigma(p)}$ for every $p \in N$, $x = (x_q)_{q \in N} \in V^N$ and $\sigma \in \mathcal{T}$. Equation \eqref{gammaf} then simply becomes
\begin{equation}
(\gamma_f)_p = f \circ \pi_p \, ,
\end{equation}
for all $p \in N$.\\
The response function $f$ can again be required to respect any structure on $V$, but is usually $C^\infty$ or $C^k$ for some $k > 0$ (when $V$ is a vector space, or more generally a manifold). In this article we will only focus on linear response functions though, in which case we may write
\begin{equation}
f(w) = \sum_{\sigma \in \mathcal{T}} C_{\sigma} w_{\sigma} \,
\end{equation}
for $w = (w_{\sigma})_{\sigma \in \mathcal{T}} \in V^I$ and with $C_{\sigma} : V \rightarrow V$ a linear map for every $\sigma \in \mathcal{T}$.

\begin{figure}
\centering
\begin{tikzpicture}
	\node[circle,draw=black, fill=bloesh, fill opacity = 1, inner sep=1pt, minimum size=12pt] (1) at (1.5,2.625) {1};
	\node[circle,draw=black, fill=bloesh, fill opacity = 1, inner sep=1pt, minimum size=12pt] (2) at (0,0) {2};
	\node[circle,draw=black, fill=bloesh, fill opacity = 1, inner sep=1pt, minimum size=12pt] (3) at (3,0) {3};
	\draw [->,  >=stealth, thick, black, loop above, shorten <=2pt, shorten >=2pt] (1) to [out=25,in=65,looseness=10]( 1);
	\draw [->,  >=stealth, thick, black, loop above, shorten <=2pt, shorten >=2pt] (2) to [out=115,in=155,looseness=10]( 2);
	\draw [->,  >=stealth, thick, black, loop above, shorten <=2pt, shorten >=2pt] (3) to [out=295,in=335,looseness=10]( 3);
	\draw [->,  >=stealth, thick, roo, loop above, shorten <=2pt, shorten >=2pt] (1) to [out=85,in=125,looseness=10]( 1);
	\draw [->,  >=stealth, thick, roo, shorten <=2pt, shorten >=2pt] (3) to [bend right = 20] (2);
	\draw [->,  >=stealth, thick, roo, shorten <=2pt, shorten >=2pt] (2) to [bend right = 5] (3);	
	\draw [->>,  >=stealth, thick, green, loop above, shorten <=2pt, shorten >=2pt] (1) to [out=145,in=185,looseness=10]( 1);
	\draw [->>,  >=stealth, thick, green, shorten <=2pt, shorten >=2pt] (1) to [bend right = 20] (2);
	\draw [->>,  >=stealth, thick, green, shorten <=2pt, shorten >=2pt] (1) to [bend left = 20] (3);
	\draw [->,  >=stealth, dashed, thick, blue, loop above, shorten <=2pt, shorten >=2pt] (3) to [out=355,in=395,looseness=10]( 3);
	\draw [->,  >=stealth, dashed, thick, blue, shorten <=2pt, shorten >=2pt] (3) to [bend left = 10] (1);
	\draw [->,  >=stealth, dashed, thick, blue, shorten <=2pt, shorten >=2pt] (3) to [bend left = 25] (2);	
%	\draw [->>,  >=stealth, dashed, thick, orange, shorten <=2pt, shorten >=2pt] (2) to [bend right = 45] (3);
%	\draw [->>,  >=stealth, dashed, thick, orange, shorten <=2pt, shorten >=2pt] (2) to [bend right = 15] (1);		
%	\draw [->>,  >=stealth, dashed, thick, orange, loop above, shorten <=2pt, shorten >=2pt] (2) to [out=185,in=225,looseness=10]( 2);
\end{tikzpicture}
\caption{An example of a homogeneous coupled cell network with asymmetric input.}
\label{fig2}
\end{figure}
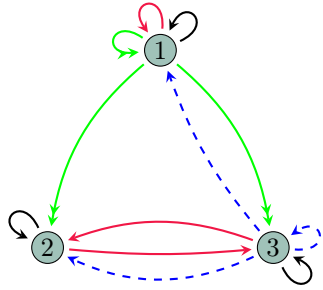

\begin{ex}\label{exam1}
Figure \ref{fig2} denotes a homogenous network with asymmetric input. Its nodes are given by $N = \{1, 2,3\}$ and its input maps are the identity on $N$ (corresponding to the black self-loops), the map $[1,3,2]$ (the red arrows), $[1,1,1]$ (the green double-headed ones) and $[3,3,3]$ (the blue dashed ones). Here and in the rest of the article we use the notation $[p_1, p_2, \dots, p_{n}]$ with $p_1, \dots, p_n \in N = \{1, \dots, n\}$ to denote the map from $N$ to itself that sends a node $i$ to a node $p_i$. A general admissible map for this network has the form
\begin{equation}\label{gammafex1}
	\begin{array}{cr}
		(\gamma_f(x))_1 &= f(x_{1}, \roo{x_1}, \gr{x_1}, \bl{x_3})  \\
		(\gamma_f(x))_2 &= f(x_{2}, \roo{x_3}, \gr{x_1}, \bl{x_3})  \\
		(\gamma_f(x))_3 &= f(x_{3}, \roo{x_2}, \gr{x_1}, \bl{x_3}) 
	\end{array} \, ,
\end{equation}
for $x = (x_1, x_2, x_3) \in V^3$ and a response function $f: V^4 \rightarrow V$. If $f$ is assumed linear, then we may write it as 
\begin{equation}
f(w) =Aw_1 + \roo{B}w_2 + \gr{C}w_3 + \bl{D}w_4\, ,
\end{equation}
for $w = (w_1, w_2, w_3, w_4) \in V^4$ and with $A, \roo{B}, \gr{C}$ and $\bl{D}$ linear maps from $V$ to $V$. In that case, equation \eqref{gammafex1} becomes
\begin{equation}\label{gammafex11}
\arraycolsep=1.4pt\def\arraystretch{1}
\gamma_f(x) = \begin{bmatrix}
A + \roo{B} + \gr{C}& 0 & \bl{D}\\
\gr{C} & A & \roo{B} + \bl{D} \\
\gr{C} &\roo{B} & A+ \bl{D} 
\end{bmatrix} \begin{bmatrix}
x_1\\
x_2\\
x_3\\
\end{bmatrix}\, .
\end{equation}
$\hfill \triangle$
\end{ex}
\noindent Next, we need the notion of an \textit{input network}. Simply put, the input network of a node $p$ in a network $\mathcal{N} =(N, \mathcal{T})$ consists of all nodes in $N$ that directly or indirectly influence $p$ (including $p$ itself). In other words, the nodes of the input network of $p$ are given by
\begin{equation}
N_p := \{q \in N \, \mid \, \exists~ \sigma_1, \dots, \sigma_k \in \mathcal{T} \text{ s.t. } q = \sigma_1 \circ \dots \circ \sigma_k(p) \} \cup \{p\}\, .
\end{equation}
The input maps are given by the restriction $\sigma|_{N_p}$ of $\sigma$ to $N_p$, for all $\sigma \in \mathcal{T}$. Note that the set $N_p$ is, by definition, mapped to itself by all elements of $\mathcal{T}$. We also point out that some of the maps $\sigma|_{N_p}$ might coincide for different $\sigma \in \mathcal{T}$. However, when this happens we will not identify them. The reason for this is that it allows us to naturally construct an admissible map  $\gamma_f^p$ for the input network $\mathcal{N}_p := (N_p, (\sigma|_{N_p})_{\sigma \in \mathcal{T}})$, given an admissible map $\gamma_f$ for $\mathcal{N}$. Formally, $\gamma_f^p$ is given by
\begin{equation}
(\gamma_f^p)_q := f \circ \pi_q^p
\end{equation}
for $q \in N_p$ and with $\pi_q^p: V^{N_p} \rightarrow V^I$ given by $(\pi_q^p(x))_{\sigma} = x_{\sigma|_{N_p}(q)} = x_{\sigma(q)}$. Here, $V^{N_p}$ denotes the phase space of the network $\mathcal{N}_p$, that is
\begin{equation}\label{totfase234}
V^{N_p}:= \bigoplus_{q \in N_p} V \, .
\end{equation}
Informally, one should just think of $\gamma_f^p$ as the expression \eqref{gammaf} for $\gamma_f$, but with the rows corresponding to nodes outside of $N_p$ `erased'. This intuition is formalised by the observation that there exists a linear surjective map $\psi_p: V^N \rightarrow V^{N_p}$ that conjugates $\gamma_f$ and $\gamma^p_f$. This map is given simply by $(\psi_p(x))_q = x_q$ for all $q \in N_p$. One verifies that indeed
\begin{equation}
\gamma^p_f \circ \psi_p = \psi_p \circ \gamma_f \, ,
\end{equation}
for all response functions $f$. One can likewise show that $\psi_p$ is indeed surjective, with kernel given by
\begin{equation}
\ker(\psi_p) = \{ v \in V^N \, \mid \, v_q = 0 \, \forall \, q \in N_p \subset N \} \, .
\end{equation}

\begin{ex}\label{exam2}
We revisit the network of Example \ref{exam1}. It may be seen from Figure \ref{fig2} that the input network of any node is equal to the whole network. Next, we consider the network that has the same nodes as that of Figure \ref{fig2}, but with all but the blue (dashed) arrows removed. In that case the input network of node $1$ consists of nodes $1$ and $3$, the input network of node $2$ contains nodes $2$ and $3$ and the input network of node $3$ contains only node $3$. $\hfill \triangle$
\end{ex}
\noindent As we will see below, the main reason for introducing input networks is the definition of the class of constructible networks $\mathfrak{C}_{\mathcal{N}}$. This class will also contain the so-called fundamental network. We will see that the input networks of the original network may be realised as quotients of the fundamental network. This observation is then generalised in the definition of $\mathfrak{C}_{\mathcal{N}}$. We proceed our analyses by showing how networks are related to certain algebraic structures.

\begin{defi}\label{monoidt1} A \textit{monoid} may be seen as a generalisation of a group, where one drops the condition that every element has to have an inverse. More precisely, a monoid is a set $\Sigma$, together with a unit $e \in \Sigma$ and a multiplication $\circ: \Sigma \times \Sigma \rightarrow \Sigma$. These have to satisfy $e \circ \sigma =  \sigma \circ e = \sigma$ for all $\sigma \in \Sigma$ (justifying the term unit) and $(\sigma \circ \tau) \circ \kappa = \sigma \circ (\tau \circ \kappa)$ for all $\sigma, \tau, \kappa \in \Sigma$ (associativity). Equivalently, one might think of a monoid as a semigroup with a unit. An example of a monoid is given by the set of all (not necessarily invertible) maps from a set of nodes $N$ to itself. Here the identity map is the unit, and multiplication is given by composition of maps. We will refer to this monoid as the full transformation monoid on $N$. In general, the set of input maps $\mathcal{T}$ of a homogeneous network $\mathcal{N}  = (N, \mathcal{T})$ may not form a monoid. However, we may always construct a monoid $\Sigma$ as the smallest sub-monoid of the full transformation monoid on $N$ that contains $\mathcal{T}$. Explicitly, $\Sigma$ is given by the identity map and all finite compositions of maps in $\mathcal{T}$. $\hfill \triangle$ 
\end{defi}
\noindent Using the construction of $\Sigma$ out of $\mathcal{T}$, we may now define a new network out of $\mathcal{N}$.

\begin{defi}\label{fundamet}
The \textit{fundamental network} of $\mathcal{N} = (N, \mathcal{T})$ is a homogeneous network with asymmetric input, with nodes indexed by the monoid $\Sigma$ and input maps indexed by $\mathcal{T}$. More precisely, the element $\tau \in \mathcal{T}$ defines a map from the set of nodes $\Sigma$ to itself, by mapping $\sigma \in \Sigma$ to $\tau \circ \sigma \in \Sigma$. With slight abuse of notation, we may therefore write $\mathcal{F} = (\Sigma, \mathcal{T})$ for the fundamental network of $\mathcal{N}$. $\hfill \triangle$
\end{defi}

\noindent As the input maps of a network $\mathcal{N} = (N, \mathcal{T})$ may be identified with those of its fundamental network $\mathcal{F} = (\Sigma, \mathcal{T})$, any response function $f: V^I \rightarrow V$ can be used to define an admissible map for both networks. More precisely, given a vector space $V$ we define the total phase space of the fundamental network by
\begin{equation}\label{infase2}
V^{\Sigma}:= \bigoplus_{\sigma \in \Sigma} V \, .
\end{equation}
An element of $V^{\Sigma}$ is then given by $(X_{\sigma})_{\sigma \in \Sigma}$, where $X_{\sigma} \in V$ denotes the state of the fundamental network node indexed by $\sigma \in \Sigma$. Given a response function $f: V^I \rightarrow V$ (with $V^I$ as in equation \eqref{infase}), we get an admissible network map for $\mathcal{F}$. We will denote this map by $\Gamma_f: V^{\Sigma} \rightarrow V^{\Sigma}$, to distinguish it from $\gamma_f$, and it is given by
\begin{equation}\label{gammaf2}
	\begin{array}{rcl}
		(\Gamma_f(X))_{\sigma_1} &= &f(X_{\sigma_1 \circ \sigma_1}, X_{\sigma_2 \circ \sigma_1}, \dots, X_{\sigma_m \circ \sigma_1})  \\
		(\Gamma_f(X))_{\sigma_2} &= &f(X_{\sigma_1 \circ \sigma_2}, X_{\sigma_2 \circ \sigma_2}, \dots, X_{\sigma_m \circ \sigma_2})  \\
		& \vdots & \\
		 (\Gamma_f(X))_{\sigma_M} &= &f(X_{\sigma_1 \circ \sigma_M}, X_{\sigma_2 \circ \sigma_M}, \dots, X_{\sigma_m \circ \sigma_M}) 
	\end{array} \, .
\end{equation}
Here we have written $\Sigma = \{\sigma_1, \dots, \sigma_M\}$ and $\mathcal{T} = \{\sigma_1, \dots, \sigma_m\}$. Note that an expression of the form $X_{\sigma_i \circ \sigma_j}$ for $\sigma_i \in \mathcal{T}$ and $\sigma_j \in \Sigma$ makes sense, as $\Sigma$ is closed under multiplication and we have $\mathcal{T} \subset \Sigma$.\\

\noindent The fundamental network is introduced in \cite{CCN} and its main purpose is twofold. First of all, every input network of $\mathcal{N} = (N, \mathcal{T})$ may be realised as a quotient of the fundamental network $\mathcal{F} = (\Sigma, \mathcal{T})$. This means that the map $\gamma_f^p$ may be obtained by restricting $\Gamma_f$ to some appropriate invariant space $S_p \subset V^{\Sigma}$. This will be the content of Lemma \ref{includedyeah}. Secondly, $\Gamma_f$ can be shown to have many symmetries. In fact, if we slightly adapt the networks $\mathcal{N}$ and $\mathcal{F}$, then the class of admissible maps $\Gamma_f$ can be seen as exactly the class of maps with certain symmetries. This will be explained in Lemma \ref{symmetryyeah}.  \\

\begin{defi}\label{synchr}
Given a network $\mathcal{N} = (N, \mathcal{T})$ and a node $p \in N$, we define the \textit{synchrony space} $S_p \subset V^{\Sigma}$ given by
\begin{equation}
S_p := \{ X \in V^{\Sigma}\,\mid \, X_{\sigma} = X_{\tau} \, \forall \, \sigma, \tau \in \Sigma \text{ s.t. } \sigma(p) = \tau(p)\} \, .
\end{equation}
One can show that this space is invariant for every $\mathcal{F}$-admissible map. That is, we have $\Gamma_f(S_p) \subset S_p$ for every response function $f: V^{I} \rightarrow V$. One says that the synchrony space $S_p$ is \textit{robust}.  \\
More generally, the term synchrony space may denote any subspace of $V^N$ or $V^{\Sigma}$ that is given by the equality of some of the components of a vector $(x_p)_{p \in N}$ or $(X_{\sigma})_{\sigma \in \Sigma}$. One also speaks of a polydiagonal space.  In particular, a general synchrony space of $V^N$ may be denoted by 
\begin{equation}\label{synch1s}
S_{\bowtie} := \{ x \in V^{\Sigma}\,\mid \, x_{p_1} = x_{p_2} \, \forall \, p_1, p_2 \in N \text{ s.t. } p_1 \bowtie p_2\} \, ,
\end{equation}
for some equivalence relation $\bowtie$ on the nodes $N$. Likewise, a general synchrony space of $V^{\Sigma}$ may be denoted by 
\begin{equation}\label{synch2s}
S_{\bowtie} := \{ X \in V^{\Sigma}\,\mid \, X_{\sigma} = X_{\tau} \, \forall \, \sigma, \tau \in \Sigma \text{ s.t. } \sigma \bowtie \tau\} \, ,
\end{equation}
for some equivalence relation $\bowtie$ on $\Sigma$. (Note that the fundamental network is again a network, but with nodes indexed by $\Sigma$. Hence, equation \eqref{synch2s} is just a special case of equation \eqref{synch1s}.) It can then be shown that $S_{\bowtie}$ is robust (i.e. satisfies $\gamma_f(S_{\bowtie}) \subset S_{\bowtie}$ for all $\gamma_f$, with $\gamma_f = \Gamma_f$ for the fundamental network) if and only if $\bowtie$ is a  so-called \textit{balanced} relation. What this means is that we have  $p_1 \bowtie p_2 \implies \sigma(p_1) \bowtie \sigma(p_2)$ for all $p_1, p_2 \in N$ and $\sigma \in \mathcal{T}$ (again with $N = \Sigma$ for the fundamental network). See for example \cite{curious} or \cite{CCN} for more on balanced relations and synchrony spaces. Note that for any node $p \in N$, the equivalence relation $\bowtie_p$ on $\Sigma$ given by $\sigma \bowtie_p \tau \Leftrightarrow \sigma(p) = \tau(p)$ is balanced, and that we have $S_p = S_{\bowtie_p}$.  One important consequence is that robustness of a synchrony space (defined by some equivalence relation $\bowtie$) does not depend on the space $V$, but only on $\bowtie$ itself. $\hfill \triangle$
\end{defi}

\noindent The following lemma relates input networks to the fundamental network.

\begin{lem}\label{includedyeah}
Given a network $\mathcal{N} = (N, \mathcal{T})$ and a node $p \in N$, denote by $\theta_p: V^{N_p} \rightarrow V^{\Sigma}$ the linear map given by $(\theta_p(x))_{\sigma} = x_{\sigma(p)}$ for $\sigma \in \Sigma$. Then $\theta_p$ is an injective map with image $S_p$. Moreover, we have
\begin{equation}
 \theta_p \circ \gamma_f^p = \Gamma_f \circ  \theta_p \, ,
\end{equation}
for every response function $f: V^I \rightarrow V$. As such, we may identify $\gamma^p_f$ with $\Gamma_f$ restricted to $S_p$.
\end{lem}
\noindent A proof of Lemma \ref{includedyeah} can be found in for example \cite{cen}. In general, restricting an admissible network map to a robust synchrony space yields an admissible map for the corresponding \textit{quotient network}. That is, for the network whose nodes are the classes of the corresponding balanced partition, and where the arrows are induced by those of the original network. More precisely, the input maps are given by $\sigma([p]) = [\sigma(p)]$ for $\sigma \in \mathcal{T}$ and with $[p]$ the class of a node $p \in N$. This is well-defined by the definition of a balanced partition. Such a relation between networks may also be understood through so-called \textit{graph fibrations}, see \cite{deville} and \cite{fibr}. \\
Next, we introduce a network property that will allow us to describe admissible maps in terms of symmetry.

\begin{defi}
A network $\mathcal{N} = (N, \mathcal{T})$ is called \textit{complete} if $\mathcal{T}$ forms a monoid of maps, with the unit given by the identity function on $N$ and multiplication given by composition of maps. In other words, $\mathcal{N}$ is complete precisely when we have $\mathcal{T} = \Sigma$. In general, we call $\overline{\mathcal{N}} := (N, \Sigma)$ the \textit{completion} of $\mathcal{N} = (N, \mathcal{T})$.
\end{defi}
\begin{remk}\label{onesbigger231}
The class of admissible maps for a network $\mathcal{N} = (N, \mathcal{T})$ can always be seen as a subset of the class of admissible maps for its completion $\overline{\mathcal{N}} = (N, \Sigma)$. This follows from the fact that $\mathcal{T}$ is a subset of  $\Sigma$. More precisely, given a fixed space $V$, the input spaces of $\mathcal{N}$ and $\overline{\mathcal{N}}$ are given by
\begin{equation}
V_{\mathcal{N}}^I = \bigoplus_{\sigma \in \mathcal{T}} V \text{ and } V_{\overline{\mathcal{N}}}^I = \bigoplus_{\sigma \in \Sigma} V  \, ,
\end{equation}
respectively. Let us write $\mathcal{T} = \{\sigma_1, \dots, \sigma_m\}$ and $\Sigma = \{\sigma_1, \dots, \sigma_M\}$ for some $M \geq m$. Now suppose we are given an admissible map $\gamma_f^{\mathcal{N} }: V^N \rightarrow V^N$ for $\mathcal{N}$ corresponding to a response function $f: V_{\mathcal{N}}^I \rightarrow V$. Out of $f$ we may define the response function $g: V_{\overline{\mathcal{N}}}^I \rightarrow V$, given simply by $g(x_{\sigma_1}, \dots, x_{\sigma_M}) := f(x_{\sigma_1}, \dots, x_{\sigma_m})$ for all $(x_{\sigma_i})_{i=1}^M \in V_{\overline{\mathcal{N}}}^I$. It follows that $\gamma_f^{\mathcal{N} } = \gamma_g^{\overline{\mathcal{N}} }$, where the latter is an admissible map for $\overline{\mathcal{N}}$. $\hfill \triangle$
\end{remk}
\noindent Note that for any complete network $\mathcal{N}$, the input space $V^I$ is equal to the total phase space of the fundamental network $V^{\Sigma}$. Indeed, both are given by
\begin{equation}\label{infase32458}
V^I = V^{\Sigma} = \bigoplus_{\sigma \in \Sigma} V \, .
\end{equation}
It also follows from the definitions that a network is complete if and only if its fundamental network is complete. Finally, we will need the linear maps $A_{\sigma}: V^{\Sigma} \rightarrow V^{\Sigma}$ for $\sigma \in \Sigma$. These are given by $(A_{\sigma}(X))_{\tau} = X_{\tau \circ \sigma}$ for all $\tau \in \Sigma$. A key lemma in the proof of Theorem \ref{main1} is the following.

\begin{lem}\label{symmetryyeah}
Given a network $\mathcal{N} = (N, \mathcal{T})$ with fundamental network $\mathcal{F} = (\Sigma, \mathcal{T})$, the maps $(A_{\sigma})_{\sigma \in \Sigma}$ form a representation of the monoid $\Sigma$. That is, we have $A_{e} = \Id_{V^{\Sigma}}$ for $e \in \Sigma$ the identity map on $N$, and
\begin{equation}
A_{\sigma} \circ A_{\tau} = A_{\sigma \circ \tau} \, ,
\end{equation}
for all $\sigma, \tau \in \Sigma$. Moreover, every map $\Gamma_f$ is equivariant under this representation. In other words, we have
\begin{equation}
A_{\sigma} \circ \Gamma_f = \Gamma_f \circ A_{\sigma} \,  ,
\end{equation}
for every response function $f: V^I \rightarrow V$ and for all $\sigma \in \Sigma$. Lastly, if $\mathcal{N}$ (and hence $\mathcal{F}$) is complete, then the maps $\Gamma_f$ are exactly all maps with this symmetry. In other words, if $F: V^{\Sigma} \rightarrow V^{\Sigma}$ satisfies 
\begin{equation}
A_{\sigma} \circ F = F \circ A_{\sigma} \,  
\end{equation}
for all $\sigma \in \Sigma$, then we may write $F = \Gamma_f$ for some unique response function $f$. This function $f$ is furthermore linear if and only $F$ is, and may in fact be described by $f = F_e$ for $e$ the unit of $\Sigma$.
\end{lem}
\noindent Lemma \ref{symmetryyeah} is proven in for example \cite{cen}. It can also be shown that when $\mathcal{N}$ (and hence $\mathcal{F}$) is complete, the admissible maps for the fundamental network may be described by: 
\begin{equation}
(\Gamma_f(X))_{\sigma} = (f \circ A_{\sigma})(X)\, ,
\end{equation}
for all $\sigma \in \Sigma$ and $X \in V^{\Sigma}$.

\begin{figure}
\centering
\begin{tikzpicture}
	\node[circle,draw=black, fill=bloesh, fill opacity = 1, inner sep=1pt, minimum size=12pt] (1) at (1.5,2.625) {1};
	\node[circle,draw=black, fill=bloesh, fill opacity = 1, inner sep=1pt, minimum size=12pt] (2) at (0,0) {2};
	\node[circle,draw=black, fill=bloesh, fill opacity = 1, inner sep=1pt, minimum size=12pt] (3) at (3,0) {3};
	\draw [->,  >=stealth, thick, black, loop above, shorten <=2pt, shorten >=2pt] (1) to [out=25,in=65,looseness=10]( 1);
	\draw [->,  >=stealth, thick, black, loop above, shorten <=2pt, shorten >=2pt] (2) to [out=115,in=155,looseness=10]( 2);
	\draw [->,  >=stealth, thick, black, loop above, shorten <=2pt, shorten >=2pt] (3) to [out=295,in=335,looseness=10]( 3);
	\draw [->,  >=stealth, thick, roo, loop above, shorten <=2pt, shorten >=2pt] (1) to [out=85,in=125,looseness=10]( 1);
	\draw [->,  >=stealth, thick, roo, shorten <=2pt, shorten >=2pt] (3) to [bend right = 20] (2);
	\draw [->,  >=stealth, thick, roo, shorten <=2pt, shorten >=2pt] (2) to [bend right = 5] (3);	
	\draw [->>,  >=stealth, thick, green, loop above, shorten <=2pt, shorten >=2pt] (1) to [out=145,in=185,looseness=10]( 1);
	\draw [->>,  >=stealth, thick, green, shorten <=2pt, shorten >=2pt] (1) to [bend right = 20] (2);
	\draw [->>,  >=stealth, thick, green, shorten <=2pt, shorten >=2pt] (1) to [bend left = 20] (3);
	\draw [->,  >=stealth, dashed, thick, blue, loop above, shorten <=2pt, shorten >=2pt] (3) to [out=355,in=395,looseness=10]( 3);
	\draw [->,  >=stealth, dashed, thick, blue, shorten <=2pt, shorten >=2pt] (3) to [bend left = 10] (1);
	\draw [->,  >=stealth, dashed, thick, blue, shorten <=2pt, shorten >=2pt] (3) to [bend left = 25] (2);	
	\draw [->>,  >=stealth, dashed, thick, orange, shorten <=2pt, shorten >=2pt] (2) to [bend right = 45] (3);
	\draw [->>,  >=stealth, dashed, thick, orange, shorten <=2pt, shorten >=2pt] (2) to [bend right = 15] (1);		
	\draw [->>,  >=stealth, dashed, thick, orange, loop above, shorten <=2pt, shorten >=2pt] (2) to [out=185,in=225,looseness=10]( 2);
	\node[circle,draw=black, fill=bloesh, fill opacity = 1, inner sep=1pt, minimum size=12pt] (f1) at (8,3.5) {1};
	\node[circle,draw=black, fill=bloesh, fill opacity = 1, inner sep=1pt, minimum size=12pt] (f2) at (6.1,2.11) {2};
	\node[circle,draw=black, fill=bloesh, fill opacity = 1, inner sep=1pt, minimum size=12pt] (f3) at (6.82, -0.11) {3};
	\node[circle,draw=black, fill=bloesh, fill opacity = 1, inner sep=1pt, minimum size=12pt] (f4) at (9.18, -0.11) {4};
	\node[circle,draw=black, fill=bloesh, fill opacity = 1, inner sep=1pt, minimum size=12pt] (f5) at (9.90, 2.11) {5};
	\draw [->,  >=stealth, thick, black, loop above, shorten <=2pt, shorten >=2pt] (f1) to [out=90,in=135,looseness=10]( f1);
	\draw [->,  >=stealth, thick, black, loop above, shorten <=2pt, shorten >=2pt] (f2) to [out=162,in=202,looseness=10]( f2);
	\draw [->,  >=stealth, thick, black, loop above, shorten <=2pt, shorten >=2pt] (f3) to [out=234,in=274,looseness=10]( f3);
	\draw [->,  >=stealth, thick, black, loop above, shorten <=2pt, shorten >=2pt] (f4) to [out=306,in=346,looseness=10]( f4);
	\draw [->,  >=stealth, thick, black, loop above, shorten <=2pt, shorten >=2pt] (f5) to [out=328,in=368,looseness=10]( f5);
	\draw [->,  >=stealth, thick, roo, shorten <=2pt, shorten >=2pt] (f1) to [bend right = 20] (f2);
	\draw [->,  >=stealth, thick, roo, shorten <=2pt, shorten >=2pt] (f2) to [bend right = 25] (f1);
	\draw [->,  >=stealth, thick, roo, loop above, shorten <=2pt, shorten >=2pt] (f3) to [out=294,in=334,looseness=10]( f3);
	\draw [->,  >=stealth, thick, roo, shorten <=2pt, shorten >=2pt] (f4) to [bend right = 25] (f5);
	\draw [->,  >=stealth, thick, roo, shorten <=2pt, shorten >=2pt] (f5) to [bend right = 20] (f4);
	\draw [->>,  >=stealth, thick, green, loop above, shorten <=2pt, shorten >=2pt] (f3) to [out=175,in=215,looseness=10]( f3);
	\draw [->>,  >=stealth, thick, green, shorten <=2pt, shorten >=2pt] (f3) to [bend right = 13] (f1);
	\draw [->>,  >=stealth, thick, green, shorten <=2pt, shorten >=2pt] (f3) to [bend left = 20] (f2);
	\draw [->>,  >=stealth, thick, green, shorten <=2pt, shorten >=2pt] (f3) to [bend left = 20] (f4);
	\draw [->>,  >=stealth, thick, green, shorten <=2pt, shorten >=2pt] (f3) to [bend left = 10] (f5);
	\draw [->,  >=stealth, dashed, thick, blue, loop above, shorten <=2pt, shorten >=2pt] (f4) to [out=246,in=286 ,looseness=10](f4);
	\draw [->,  >=stealth, dashed, thick, blue, shorten <=2pt, shorten >=2pt] (f4) to [bend left = 0] (f1);
	\draw [->,  >=stealth, dashed, thick, blue, shorten <=2pt, shorten >=2pt] (f4) to [bend left = 0] (f2);	
	\draw [->,  >=stealth, dashed, thick, blue, shorten <=2pt, shorten >=2pt] (f4) to [bend left = 0] (f3);	
	\draw [->,  >=stealth, dashed, thick, blue, shorten <=2pt, shorten >=2pt] (f4) to [bend left = 0] (f5);
	\draw [->>,  >=stealth, dashed, thick, orange, shorten <=2pt, shorten >=2pt] (f5) to [bend right = 20] (f1);	
	\draw [->>,  >=stealth, dashed, thick, orange, shorten <=2pt, shorten >=2pt] (f5) to [bend right = 30] (f2);		
	\draw [->>,  >=stealth, dashed, thick, orange, shorten <=2pt, shorten >=2pt] (f5) to [bend right = 50] (f3);	
	\draw [->>,  >=stealth, dashed, thick, orange, shorten <=2pt, shorten >=2pt] (f5) to [bend left = 50] (f4);
	\draw [->>,  >=stealth, dashed, thick, orange, loop above, shorten <=2pt, shorten >=2pt] (f5) to [out=408,in=448,looseness=10](f5);
	\draw [->,  >=stealth, decorate, thick, decoration={snake,amplitude= 0.4mm,segment length=4mm,post length=1mm}, purple]  (5.5,1.5) to node[above] {\begin{tabular}{l} $1 \mapsfrom 1,2,3$ \\ $3 \mapsfrom 4$ \\ $2 \mapsfrom 5$ \end{tabular}}  (3.5,1.5);
\end{tikzpicture}
\caption{Left: the network obtained by completing the network of Figure \ref{fig2}. \\ Right: the fundamental network of the network on the left. The original network on the left may be seen as a quotient of the fundamental network by sending the nodes $1$, $2$ and $3$ to $1$,  $4$ to $3$ and $5$ to $2$.}
\label{fig3}
\end{figure}
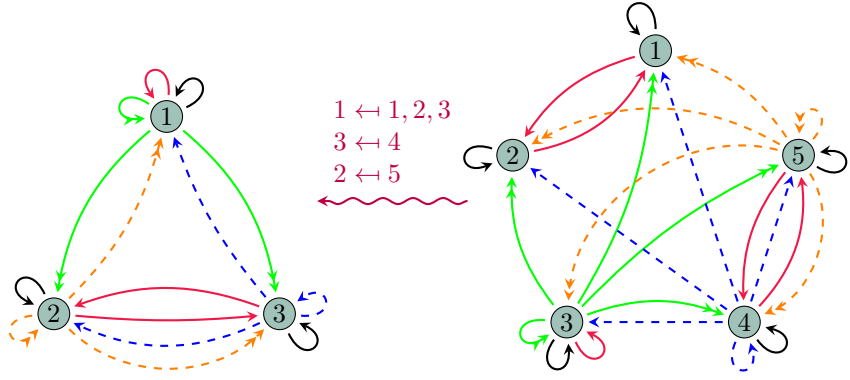
%decorate,decoration=snake ], $1 \mapsfrom 1,2,3 $

\begin{ex}\label{exam3}
We return to our running example of Figure \ref{fig2}. Here we have 
\begin{equation}
\mathcal{T} = \{[1,2,3], [1,3,2], [1,1,1], [3,3,3]\} \, .
\end{equation}
As it holds that $[1,3,2]\circ[3,3,3] = [2,2,2]$, the monoid $\Sigma$ will also contain this latter map. In fact, one can show that no other maps have to be added to make $\mathcal{T}$ into a monoid, and it follows that 
\begin{equation}
\Sigma= \{[1,2,3], [1,3,2], [1,1,1], [3,3,3], [2,2,2]\} \, .
\end{equation}
The left side of Figure \ref{fig3} shows the completion of the network of Figure \ref{fig2}. Note that only the orange (dashed, double headed) arrows are added, corresponding to the input map $[2,2,2]$. Moreover, the identity $[1,3,2]\circ[3,3,3] = [2,2,2]$ is reflected in the left network of Figure \ref{fig3} by the fact that following a blue arrow backwards and then a red arrow backwards amounts to following an orange arrow backwards. In the same way, concatenating two colours of arrows always yields an existing colour of arrow, which reflects the fact that $\Sigma$ is closed under multiplication.\\
An admissible map for the network on the left of Figure \ref{fig3} is given by
\begin{equation}\label{gammafex2}
	\gamma_f(x) = \begin{array}{cr}
		f(x_{1}, \roo{x_1}, \gr{x_1}, \bl{x_3}, \ora{x_2})  \\
		f(x_{2}, \roo{x_3}, \gr{x_1}, \bl{x_3}, \ora{x_2})  \\
		f(x_{3}, \roo{x_2}, \gr{x_1}, \bl{x_3}, \ora{x_2}) 
	\end{array} \, ,
\end{equation}
for $x = (x_1, x_2, x_3) \in V^3$ for some phase space $V$ and a response function $f: V^5 \rightarrow V$. If $f$ is linear, then we may write it as 
\begin{equation}\label{lini01}
f(w) =Aw_1 + \roo{B}w_2 + \gr{C}w_3 + \bl{D}w_4 + \ora{E}w_5\, ,
\end{equation}
for $w = (w_1, w_2, w_3, w_4, w_5) \in V^5$ and with $A, \roo{B}, \gr{C}, \bl{D}$ and  \ora{E} linear maps from $V$ to $V$. Equation \eqref{gammafex2} then becomes
\begin{equation}\label{gammafex12}
\arraycolsep=1.4pt\def\arraystretch{1}
\gamma_f(x) = \begin{bmatrix}
A + \roo{B} + \gr{C}&  \ora{E} & \bl{D}\\
\gr{C} & A+ \ora{E} & \roo{B} + \bl{D} \\
\gr{C} &\roo{B} +  \ora{E}& A+ \bl{D} 
\end{bmatrix} \begin{bmatrix}
x_1\\
x_2\\
x_3\\
\end{bmatrix}\, .
\end{equation}
The right side of Figure \ref{fig3} shows the fundamental network of the network on the left. The admissible maps for this network on the right are given by
\begin{equation}\label{gammafex3}
	\Gamma_f(X) = \begin{array}{cr}
		f(X_{1}, \roo{X_2}, \gr{X_3}, \bl{X_4}, \ora{X_5})  \\
		f(X_{2}, \roo{X_1}, \gr{X_3}, \bl{X_4}, \ora{X_5})  \\
		f(X_{3}, \roo{X_3}, \gr{X_3}, \bl{X_4}, \ora{X_5})  \\
		f(X_{4}, \roo{X_5}, \gr{X_3}, \bl{X_4}, \ora{X_5})  \\
		f(X_{5}, \roo{X_4}, \gr{X_3}, \bl{X_4}, \ora{X_5})  
	\end{array} \, ,
\end{equation}
for $X = (X_1, X_2, X_3, X_4, X_5) \in V^5$. Figure \ref{fig3} also shows that the original network may be seen as a quotient network of the fundamental network. (Recall that the network on the left is the input network of any of its nodes.) More precisely, the network on the left of Figure \ref{fig3} can be obtained from the one on the right by identifying the nodes $1$, $2$ and $3$ to a single node, and by renaming the nodes of the resulting three cell network. This is reflected in equations \eqref{gammafex2} and \eqref{gammafex3} by the fact that the synchrony space $\{X_1 = X_2 = X_3\} \subset V^5$ is fixed by any map of the form \eqref{gammafex3}, and that the restriction to this space yields equation \eqref{gammafex2} (after renaming the nodes). Moreover, it can be shown that equation \eqref{gammafex3} describes exactly all maps that commute with the linear maps
\begin{align}
&A_1: X \mapsto X  & \\ \nonumber 
&A_2: X \mapsto (X_2, X_1, X_3, X_4, X_5)  &A_3: X \mapsto (X_3, X_3, X_3, X_4, X_5) \\ \nonumber
&A_4: X \mapsto (X_4, X_5, X_3, X_4, X_5)  &A_5: X \mapsto (X_5, X_4, X_3, X_4, X_5) \nonumber
\end{align}
for $X =  (X_1, X_2, X_3, X_4, X_5) \in V^5$. If $f$ is linear and given by equation \eqref{lini01}, then equation \eqref{gammafex3} becomes
\begin{equation}\label{gammafex14}
\arraycolsep=1.4pt\def\arraystretch{1}
\Gamma_f(X) = \begin{bmatrix}
A & \roo{B} & \gr{C} &  \bl{D} & \ora{E} \\
 \roo{B} & A& \gr{C} &  \bl{D} & \ora{E} \\
0 & 0 & A + \roo{B} + \gr{C} &  \bl{D} & \ora{E} \\
0 & 0 & \gr{C} &  A + \bl{D} & \roo{B} + \ora{E} \\
0 & 0 & \gr{C} & \roo{B} + \bl{D} & A + \ora{E} 
\end{bmatrix} \begin{bmatrix}
X_1\\
X_2\\
X_3\\
X_4\\
X_5\\
\end{bmatrix}\, .
\end{equation}
$\hfill \triangle$
\end{ex}

%%%%%%%%%%%%%%%%%%%%%%Input Networks and Constructible Networks%%%%%%%%%%%%%%%%%%%%%%%%%%%%%

\section{Constructible Networks} \label{Input Networks and Constructible Networks}
In the previous section we have seen that every homogeneous network with asymmetric input $\mathcal{N}$ has a fundamental network $\mathcal{F}$. Moreover, the input networks of the original network $\mathcal{N}$ may all be realised as quotients of $\mathcal{F}$. We will generalise this relation between $\mathcal{N}$ and $\mathcal{F}$ in the definition of the class of constructible networks.

\begin{defi}[Constructible Networks]\label{ConstructibleNetworks}
Let $\mathcal{F} = (\Sigma, \Sigma)$ be a fixed complete fundamental network. The class of \textit{constructible networks} of $\mathcal{F}$ consists of all homogeneous coupled cell networks with asymmetric input for which the input network of each node is isomorphic to a quotient of $\mathcal{F}$. We denote the class of constructible networks of $\mathcal{F}$ by $\mathfrak{C}_{\mathcal{F}}$, or sometimes $\mathfrak{C}_{\Sigma}$. If $\mathcal{F}$ is the fundamental network of $\mathcal{N}$, then we will also write $\mathfrak{C}_{\mathcal{N}} := \mathfrak{C}_{\mathcal{F}}$. This should not cause any confusion, as the fundamental network of a fundamental network is isomorphic to itself, see e.g. \cite{CCN}. $\hfill \triangle$
\end{defi}

\begin{remk}\label{remkconstrntwrks}
Terms such as \textit{quotient network} and \textit{isomorphic} should be understood in the right category of coupled cell networks, namely one in which we consider graph fibrations between networks. In our setting, a graph fibration between two networks is a map of directed graphs such that arrow-type is preserved. This means that nodes are sent to nodes and arrows to arrows, in such a way that both arrow-type and source and target maps are respected.  To make sure we can speak of a type-preserving map between arrows, we impose that every network in $\mathfrak{C}_{\mathcal{F}}$ has to have exactly $\#\Sigma$ types of arrow, which are furthermore indexed by $\Sigma$. As our networks are all asymmetric, such a graph fibration is completely determined by its restriction to the nodes. In fact, a map $F$ between the nodes of two networks $\mathcal{N}_1$ and $\mathcal{N}_2$ (both with input maps indexed by $\Sigma$), determines a graph fibration if and only if $F \circ \sigma_1 = \sigma_2 \circ F$ for all $\sigma \in \Sigma$. (Here $\sigma_i$ denotes the input map of $\mathcal{N}_i$ indexed by $\sigma$, for $i \in \{1,2\}$.) Note that we do not assume that all input maps corresponding to the different arrow-types are distinct. A useful consequence of the definition of $\mathfrak{C}_{\mathcal{F}}$ in this way is that a single response function $f: V^{\Sigma} \rightarrow V$ allows us to construct admissible maps for all the networks in $\mathfrak{C}_{\mathcal{F}}$ simultaneously  in an unambiguous way. We also note that we call a network a quotient of another one if there is a surjective graph fibration from the latter network to the former. See \cite{deville} or \cite{fibr} for more on graph fibrations. $\hfill \triangle$ 
\end{remk}

%%%%%

\begin{figure}
\centering
\begin{tikzpicture}
	\node[circle,draw=black, fill=bloesh, fill opacity = 1, inner sep=1pt, minimum size=12pt] (0) at (0,0) {0};
	\node[circle,draw=black, fill=bloesh, fill opacity = 1, inner sep=1pt, minimum size=12pt] (1) at (0,2) {1};
	\node[circle,draw=black, fill=bloesh, fill opacity = 1, inner sep=1pt, minimum size=12pt] (c0) at (5,0) {0};
	\node[circle,draw=black, fill=bloesh, fill opacity = 1, inner sep=1pt, minimum size=12pt] (c1) at (4,2) {1};
	\node[circle,draw=black, fill=bloesh, fill opacity = 1, inner sep=1pt, minimum size=12pt] (c2) at (5,2) {2};
	\node[circle, inner sep=1pt, minimum size=12pt] (l) at (6,2) {$\dots$};
	\node[circle,draw=black, fill=bloesh, fill opacity = 1, inner sep=1pt, minimum size=12pt] (cn) at (7,2) {n};
	\draw [->,  >=stealth, thick, black, loop above, shorten <=2pt, shorten >=2pt] (0) to [out=-20,in=20, looseness=10]( 0);
	\draw [->,  >=stealth, thick, black, loop above, shorten <=2pt, shorten >=2pt] (1) to [out=-20,in=20 ,looseness=10]( 1);
	\draw [->,  >=stealth, thick, black, loop above, shorten <=2pt, shorten >=2pt] (c0) to [out=-20,in=20,looseness=10]( c0);
	\draw [->,  >=stealth, thick, black, loop above, shorten <=2pt, shorten >=2pt] (c1) to [out=70,in=110, looseness=10]( c1);
	\draw [->,  >=stealth, thick, black, loop above, shorten <=2pt, shorten >=2pt] (c2) to [out=70,in=110, looseness=10]( c2);	
	\draw [->,  >=stealth, thick, black, loop above, shorten <=2pt, shorten >=2pt] (cn) to [out=70,in=110, looseness=10]( cn);
	\draw [->,  >=stealth, dashed, thick, roo, shorten <=2pt, shorten >=2pt] (0) to [bend right = 20] (1);	
	\draw [->,  >=stealth, dashed, thick, roo, shorten <=2pt, shorten >=2pt] (c0) to [bend right = -20] (c1);	
	\draw [->,  >=stealth, dashed, thick, roo, shorten <=2pt, shorten >=2pt] (c0) to [bend right = -20] (c2);	
	\draw [->,  >=stealth, dashed, thick, roo, shorten <=2pt, shorten >=2pt] (c0) to [bend right = -20] (cn);	
	\draw [->,  >=stealth, dashed, thick, roo, loop above, shorten <=2pt, shorten >=2pt] (0) to [out=160,in=200, looseness=10](0);	
	\draw [->,  >=stealth, dashed, thick, roo, loop above, shorten <=2pt, shorten >=2pt] (c0) to [out=160,in=200, looseness=10](c0);	
\end{tikzpicture}
\caption{Left: an example of a fundamental network $\mathcal{F}$. \\ Right: A family of networks in $\mathfrak{C}_{\mathcal{F}}$ for $n \in \N\cup\{0\}.$}
\label{fig4}
\end{figure}
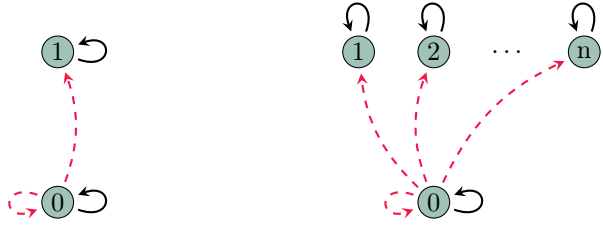

\begin{ex}\label{exam4}
Figure \ref{fig4} shows a fundamental network $\mathcal{F}$ on the left and an infinite family of (finite) networks belonging to $\mathfrak{C}_{\mathcal{F}}$ on the right. As a matter of fact, it is not hard to see that the elements of $\mathfrak{C}_{\mathcal{F}}$ are exactly all networks with connected components given by the networks on the right (possibly for $n=0$). $\hfill \triangle$
\end{ex}
\noindent The following lemma makes it easier to identify elements of $\mathfrak{C}_{\mathcal{F}}$.
\begin{lem}\label{less2check}
A network $\mathcal{N}$ belongs to $\mathfrak{C}_{\mathcal{F}}$ if and only if $\mathcal{N}$ can be covered by input networks that are quotients of ${\mathcal{F}}$.
\end{lem}

\begin{proof}
As the input networks of any network $\mathcal{N}$ clearly cover $\mathcal{N}$, it follows that any network in $\mathfrak{C}_{\mathcal{F}}$ is covered by input networks that are quotients of ${\mathcal{F}}$. Conversely, suppose $\mathcal{N}$ may be covered by input networks $(\mathcal{N}_p)_{p \in U}$ for some set of nodes $U$ of $\mathcal{N}$. Suppose furthermore that every $\mathcal{N}_p$ for $p \in U$ is isomorphic to some quotient of ${\mathcal{F}}$. We pick a node $q$ of ${\mathcal{N}}$, together with a node $p \in U$ such that $q \in \mathcal{N}_p$. As $\mathcal{N}_p$ may be realised as a quotient of ${\mathcal{F}}$, there exists a surjective graph fibration $F: \mathcal{F} \rightarrow \mathcal{N}_p$. Let $\tau \in \Sigma$ moreover be a node of ${\mathcal{F}}$ such that $F(\tau) = q$. We claim that the map $G: \mathcal{F} \rightarrow \mathcal{N}_p$ given by $G(\sigma) = F(\sigma \circ \tau)$ for all $\sigma \in \Sigma$ defines a graph fibration with image equal to $\mathcal{N}_q$. First of all, we see that 
\begin{equation}
\kappa G(\sigma) = \kappa F(\sigma \circ \tau) =  F(\kappa \circ ( \sigma \circ \tau)) = F((\kappa \circ  \sigma) \circ \tau) = G(\kappa \circ  \sigma)\, ,
\end{equation}
for all $\sigma, \kappa \in \Sigma$. Here we simply use $\kappa$ to denote the corresponding input map in both networks. This shows that $G$ is indeed a graph fibration. Secondly, it is clear that the image of $G$ is indeed a subnetwork of $\mathcal{N}_p$, and we note that $\mathcal{N}_r \subset \mathcal{N}_p$ for any node $r \in \mathcal{N}_p$. Lastly, we see that 
\begin{equation}
\im(G) = F(\Sigma\tau) = \Sigma F(\tau) = \Sigma(q) = \mathcal{N}_q \, .
\end{equation}
This shows that the input network of any node in $\mathcal{N}$ may be viewed as a quotient-network of $\mathcal{F}$, so that indeed $\mathcal{N} \in \mathfrak{C}_{\mathcal{F}}$. This proves the lemma.
\end{proof}

\begin{figure}
\centering
\begin{tikzpicture}
	\node[circle,draw=black, fill=bloesh, fill opacity = 1, inner sep=1pt, minimum size=12pt] (1) at (2.2,1.5) {1};
	\node[circle,draw=black, fill=bloesh, fill opacity = 1, inner sep=1pt, minimum size=12pt] (2) at (0.8,1.5) {2};
	\node[circle,draw=black, fill=bloesh, fill opacity = 1, inner sep=1pt, minimum size=12pt] (3) at (0,0) {3};
	\node[circle,draw=black, fill=bloesh, fill opacity = 1, inner sep=1pt, minimum size=12pt] (4) at (3,0) {4};
	\node[circle,draw=black, fill=bloesh, fill opacity = 1, inner sep=1pt, minimum size=12pt] (5) at (1.5,-1) {5};
	\draw [->,  >=stealth, thick, black, loop above, shorten <=2pt, shorten >=2pt] (1) to [out=70,in=110,looseness=10]( 1);
	\draw [->,  >=stealth, thick, black, loop above, shorten <=2pt, shorten >=2pt] (2) to [out=70,in=110,looseness=10]( 2);
	\draw [->,  >=stealth, thick, black, loop above, shorten <=2pt, shorten >=2pt] (3) to [out=175,in=215,looseness=10]( 3);
	\draw [->,  >=stealth, thick, black, loop above, shorten <=2pt, shorten >=2pt] (4) to [out=295,in=335,looseness=10]( 4);
	\draw [->,  >=stealth, thick, black, loop above, shorten <=2pt, shorten >=2pt] (5) to [out=275,in=315,looseness=10]( 5);
	\draw [->,  >=stealth, thick, roo, loop above, shorten <=2pt, shorten >=2pt] (3) to [out=115,in=155,looseness=10]( 3);
	\draw [->,  >=stealth, thick, roo, loop above, shorten <=2pt, shorten >=2pt] (4) to [out=55,in=95,looseness=10]( 4);
	\draw [->,  >=stealth, thick, roo, loop above, shorten <=2pt, shorten >=2pt] (5) to [out=215,in=255,looseness=10]( 5);
	\draw [->,  >=stealth, thick, roo, shorten <=2pt, shorten >=2pt] (1) to [bend right = 25] (2);
	\draw [->,  >=stealth, thick, roo, shorten <=2pt, shorten >=2pt] (2) to [bend right = 0] (1);	
	\draw [->>,  >=stealth, thick, green, loop above, shorten <=2pt, shorten >=2pt] (3) to [out=235,in=275,looseness=10]( 3);
	\draw [->>,  >=stealth, thick, green, shorten <=2pt, shorten >=2pt] (3) to [bend right = -10] (1);
	\draw [->>,  >=stealth, thick, green, shorten <=2pt, shorten >=2pt] (3) to [bend left = 20] (2);
	\draw [->>,  >=stealth, thick, green, shorten <=2pt, shorten >=2pt] (3) to [bend left = -20] (4);
	\draw [->>,  >=stealth, thick, green, shorten <=2pt, shorten >=2pt] (3) to [bend left = -20] (5);
	\draw [->,  >=stealth, dashed, thick, blue, loop above, shorten <=2pt, shorten >=2pt] (4) to [out=355,in=395,looseness=10]( 4);
	\draw [->,  >=stealth, dashed, thick, blue, shorten <=2pt, shorten >=2pt] (4) to [bend left = 0] (1);
	\draw [->,  >=stealth, dashed, thick, blue, shorten <=2pt, shorten >=2pt] (4) to [bend left = 5] (2);
	\draw [->,  >=stealth, dashed, thick, blue, shorten <=2pt, shorten >=2pt] (4) to [bend left = 0] (3);
	\draw [->,  >=stealth, dashed, thick, blue, shorten <=2pt, shorten >=2pt] (4) to [bend left = 45] (5);
	\draw [->>,  >=stealth, dashed, thick, orange, loop above, shorten <=2pt, shorten >=2pt] (4) to [out=235,in=275,looseness=10]( 4);		
	\draw [->>,  >=stealth, dashed, thick, orange, shorten <=2pt, shorten >=2pt] (4) to [bend right = 45] (1);
	\draw [->>,  >=stealth, dashed, thick, orange, shorten <=2pt, shorten >=2pt] (4) to [bend right = 15] (2);		
	\draw [->>,  >=stealth, dashed, thick, orange, shorten <=2pt, shorten >=2pt] (4) to [bend right = 20] (3);
	\draw [->>,  >=stealth, dashed, thick, orange, shorten <=2pt, shorten >=2pt] (4) to [bend right = -5] (5);
	\node[circle,draw=black, fill=bloesh, fill opacity = 1, inner sep=1pt, minimum size=12pt] (f1) at (8,2.5) {1};
	\node[circle,draw=black, fill=bloesh, fill opacity = 1, inner sep=1pt, minimum size=12pt] (f2) at (6.1,1.11) {2};
	\node[circle,draw=black, fill=bloesh, fill opacity = 1, inner sep=1pt, minimum size=12pt] (f3) at (6.82, -1.11) {3};
	\node[circle,draw=black, fill=bloesh, fill opacity = 1, inner sep=1pt, minimum size=12pt] (f4) at (9.18, -1.11) {4};
	\node[circle,draw=black, fill=bloesh, fill opacity = 1, inner sep=1pt, minimum size=12pt] (f5) at (9.90, 1.11) {5};
	\draw [->,  >=stealth, thick, black, loop above, shorten <=2pt, shorten >=2pt] (f1) to [out=90,in=135,looseness=10]( f1);
	\draw [->,  >=stealth, thick, black, loop above, shorten <=2pt, shorten >=2pt] (f2) to [out=162,in=202,looseness=10]( f2);
	\draw [->,  >=stealth, thick, black, loop above, shorten <=2pt, shorten >=2pt] (f3) to [out=234,in=274,looseness=10]( f3);
	\draw [->,  >=stealth, thick, black, loop above, shorten <=2pt, shorten >=2pt] (f4) to [out=306,in=346,looseness=10]( f4);
	\draw [->,  >=stealth, thick, black, loop above, shorten <=2pt, shorten >=2pt] (f5) to [out=328,in=368,looseness=10]( f5);
	\draw [->,  >=stealth, thick, roo, shorten <=2pt, shorten >=2pt] (f1) to [bend right = 20] (f2);
	\draw [->,  >=stealth, thick, roo, shorten <=2pt, shorten >=2pt] (f2) to [bend right = 25] (f1);
	\draw [->,  >=stealth, thick, roo, loop above, shorten <=2pt, shorten >=2pt] (f3) to [out=294,in=334,looseness=10]( f3);
	\draw [->,  >=stealth, thick, roo, shorten <=2pt, shorten >=2pt] (f4) to [bend right = 25] (f5);
	\draw [->,  >=stealth, thick, roo, shorten <=2pt, shorten >=2pt] (f5) to [bend right = 20] (f4);
	\draw [->>,  >=stealth, thick, green, loop above, shorten <=2pt, shorten >=2pt] (f3) to [out=175,in=215,looseness=10]( f3);
	\draw [->>,  >=stealth, thick, green, shorten <=2pt, shorten >=2pt] (f3) to [bend right = 13] (f1);
	\draw [->>,  >=stealth, thick, green, shorten <=2pt, shorten >=2pt] (f3) to [bend left = 20] (f2);
	\draw [->>,  >=stealth, thick, green, shorten <=2pt, shorten >=2pt] (f3) to [bend left = 20] (f4);
	\draw [->>,  >=stealth, thick, green, shorten <=2pt, shorten >=2pt] (f3) to [bend left = 10] (f5);
	\draw [->,  >=stealth, dashed, thick, blue, loop above, shorten <=2pt, shorten >=2pt] (f4) to [out=246,in=286 ,looseness=10](f4);
	\draw [->,  >=stealth, dashed, thick, blue, shorten <=2pt, shorten >=2pt] (f4) to [bend left = 0] (f1);
	\draw [->,  >=stealth, dashed, thick, blue, shorten <=2pt, shorten >=2pt] (f4) to [bend left = 0] (f2);	
	\draw [->,  >=stealth, dashed, thick, blue, shorten <=2pt, shorten >=2pt] (f4) to [bend left = 0] (f3);	
	\draw [->,  >=stealth, dashed, thick, blue, shorten <=2pt, shorten >=2pt] (f4) to [bend left = 0] (f5);
	\draw [->>,  >=stealth, dashed, thick, orange, shorten <=2pt, shorten >=2pt] (f5) to [bend right = 20] (f1);	
	\draw [->>,  >=stealth, dashed, thick, orange, shorten <=2pt, shorten >=2pt] (f5) to [bend right = 30] (f2);		
	\draw [->>,  >=stealth, dashed, thick, orange, shorten <=2pt, shorten >=2pt] (f5) to [bend right = 50] (f3);	
	\draw [->>,  >=stealth, dashed, thick, orange, shorten <=2pt, shorten >=2pt] (f5) to [bend left = 50] (f4);
	\draw [->>,  >=stealth, dashed, thick, orange, loop above, shorten <=2pt, shorten >=2pt] (f5) to [out=408,in=448,looseness=10](f5);
	\draw [->,  >=stealth, decorate, thick, decoration={snake,amplitude= 0.4mm,segment length=4mm,post length=1mm}, purple]  (5.5,0) to node[above] {\begin{tabular}{l} $1 \mapsfrom 1$\\ $2 \mapsfrom 2$ \\ $3 \mapsfrom 3$ \\ $4 \mapsfrom 4,5$  \end{tabular}} node[below]{{\begin{tabular}{l} $5 \mapsfrom 1,2$ \\ $3 \mapsfrom 3$ \\ $4 \mapsfrom 4,5$  \end{tabular}}} (4,0);
\end{tikzpicture}
\caption{Left: an example of a network in $\mathfrak{C}_{\mathcal{F}}$, for $\mathcal{F}$ the fundamental network on the right. \\
Right: a fundamental network equal to the one depicted on the right of Figure \ref{fig3}. The input network of node $1$ in the network on the left may be realised as a quotient of the network on the right by sending  node $1$ to $1$, $2$ to $2$, $3$ to $3$ and $4$ and $5$ to $4$. The input network of node $5$ in the network on the left may be seen as a quotient of the network on the right by sending nodes $1$ and $2$ to $5$, $3$ to $3$ and $4$ and $5$ to $4$.}
\label{fig5}
\end{figure}
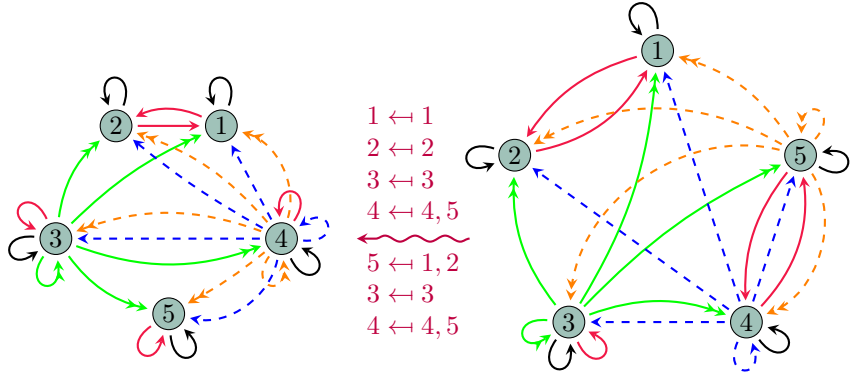

\begin{ex}\label{exam5}
Let $\mathcal{F}$ denote the fundamental network on the right of Figure \ref{fig3}. The left side of Figure \ref{fig5} shows an example of a network that is contained in $\mathfrak{C}_{\mathcal{F}}$. As the network on the left of  Figure \ref{fig5} can be covered by the input networks of nodes $1$ and $5$, it suffices by Lemma \ref{less2check} to check that these two input networks are quotients of $\mathcal{F}$. Figure \ref{fig5} also describes the relevant graph fibrations that show that this is indeed the case (for nodes $1$ and $5$ above and below the curvy arrow, respectively). Note that in the network on the left, the blue and orange arrow types (that is, the two dashed types) define the same input map $[4,4,4,4,4]$. We also point out that the network on the left cannot be realised as a quotient of the fundamental network on the right. If it could, then the networks would have to be isomorphic (both have five nodes), but then all arrow types would define different input maps. Alternatively, it can be seen that a fundamental network is always equal to the input network of the unit in $\Sigma$, whereas the network on the left of Figure \ref{fig5} does not equal the input network of any single node. $\hfill \triangle$
\end{ex}

\noindent Next, we show that $\mathfrak{C}_{\mathcal{F}}$ is closed under many natural operations.

\begin{prop}\label{cloosed}
If $\mathcal{F}$ is the fundamental network of a network $\mathcal{N}$, then we have $\mathcal{N} \in \mathfrak{C}_{\mathcal{F}}$. In particular, it follows that $\mathcal{F} \in \mathfrak{C}_{\mathcal{F}}$. Moreover, for any two constructible networks $\mathcal{M}, \mathcal{N} \in \mathfrak{C}_{\mathcal{F}}$ it holds that $\mathfrak{C}_{\mathcal{F}}$ contains the disjoint union of the two networks $\mathcal{M}\sqcup \mathcal{N}$, any subnetwork of $\mathcal{N}$ and any quotient network of $\mathcal{N}$.
\end{prop}

\begin{proof}
If $\mathcal{F}$ is the fundamental network of $\mathcal{N}$, then it follows from Lemma \ref{includedyeah} that $\mathcal{N} \in \mathfrak{C}_{\mathcal{F}}$. More precisely, for a node $p$ of $\mathcal{N}$ we have a surjective graph fibration from $\mathcal{F}$ onto $\mathcal{N}_p$ given by $\sigma \mapsto \sigma(p)$ for all $\sigma \in \Sigma$. As $\mathcal{F}$ is its own fundamental network, we also find  $\mathcal{F} \in \mathfrak{C}_{\mathcal{F}}$.  It follows directly from the definition of $\mathfrak{C}_{\mathcal{F}}$ that
$\mathcal{M}, \mathcal{N} \in \mathfrak{C}_{\mathcal{F}}$ implies $\mathcal{M} \sqcup \mathcal{N} \in \mathfrak{C}_{\mathcal{F}}$. Likewise, any subnetwork of $\mathcal{N} \in  \mathfrak{C}_{\mathcal{F}}$ is contained in $\mathfrak{C}_{\mathcal{F}}$, as every input network of a node in a subnetwork of $\mathcal{N}$ is equal to the input network of that same node in $\mathcal{N}$. It remains to show that any quotient of a constructible network $\mathcal{N}$ is again constructible. To this end, suppose $F$ is a surjective graph fibration from the constructible network $\mathcal{N}$ to a network $\mathcal{M}$. We pick a node $p \in \mathcal{M}$, so that we have to show that the corresponding input network $\mathcal{M}_p$ is a quotient of $\mathcal{F}$. To this end, let $q$ be a node of $\mathcal{N}$ such that $F(q) = p$. We claim that $F$ restricts to a surjective graph fibration from the input network $\mathcal{N}_q$ of $q$ in $\mathcal{N}$ to $\mathcal{M}_p$. To see this, note that any node in $\mathcal{N}_q$ is of the form $\sigma(q)$ for some $\sigma \in \Sigma$. It follows that $F(\sigma(q)) = \sigma F(q) = \sigma(p) \in \mathcal{M}_p$, so that $F$ indeed maps $\mathcal{N}_q$ to $\mathcal{M}_p$. (As usual, we write $\sigma$ for the corresponding input map in both networks). To show surjectivity, we note that any element of $\mathcal{M}_p$ is of the form $\sigma(p) = \sigma(F(q)) = F(\sigma(q))$ for some $\sigma \in \Sigma$. As it clearly holds that $\sigma(q) \in \mathcal{N}_q$, we conclude that $F|_{\mathcal{N}_q} : \mathcal{N}_q \rightarrow \mathcal{M}_p$ is indeed a surjective graph fibration. By definition, there exists a surjective graph fibration $G$ from $\mathcal{F}$ onto $\mathcal{N}_q$, and we conclude that there is a surjective graph fibration $F|_{\mathcal{N}_q} \circ G: \mathcal{F} \rightarrow \mathcal{M}_p$. As this holds for any node $p$ of $\mathcal{M}$, we see that indeed $M \in  \mathfrak{C}_{\mathcal{F}}$. This concludes the proof.
\end{proof}

\noindent Even though some elements of $\Sigma$ might define the same input map in a constructible network $\mathcal{N} \in \mathfrak{C}_{\Sigma}$ (see Example \ref{exam5}), we may still say that $\mathcal{N}$ is complete, given that $\mathcal{F}$ is. To make this statement precise, let us denote by $\sigma_{\mathcal{N}}$ the input map on the nodes of $\mathcal{N} \in \mathfrak{C}_{\Sigma}$ corresponding to the monoid element $\sigma \in \Sigma$. We have the following statement.

 \begin{lem}\label{stillcloded1}
If $\sigma$, $\tau$ and $\kappa$ are elements of the monoid $\Sigma$ such that $\sigma \circ \tau = \kappa$, then for any constructible network $\mathcal{N} \in \mathfrak{C}_{\Sigma}$ we have $\sigma_{\mathcal{N}} \circ \tau_{\mathcal{N}} = \kappa_{\mathcal{N}}$. Moreover, for $e$ the unit of $\Sigma$ it holds that $e_{\mathcal{N}}$ is the identity map on the nodes of $\mathcal{N}$.
 \end{lem}
 
 \begin{proof}
 First of all, by definition of the fundamental network $\mathcal{F} := (\Sigma, \Sigma)$ we have $\sigma_{\mathcal{F}} \circ \tau_{\mathcal{F}} = \kappa_{\mathcal{F}}$ and $e_{\mathcal{F}} = \Id_{\mathcal{F}}$. Next, suppose we have a surjective graph fibration $F$ from $\mathcal{F}$ to a network $\mathcal{M}$. Let $p$ be any node of $\mathcal{M}$ and suppose $q$ is a node of $\mathcal{F}$ such that $F(q) = p$. It follows that
 \begin{align}\label{zoveel234}
 (\sigma_{\mathcal{M}} \circ \tau_{\mathcal{M}})(p) &=  (\sigma_{\mathcal{M}} \circ \tau_{\mathcal{M}})(F(q)) = F((\sigma_{\mathcal{F}} \circ \tau_{\mathcal{F}})(q)) = F(\kappa_{\mathcal{F}}(q)) \\ \nonumber
 &= \kappa_{\mathcal{M}}(F(q)) =  \kappa_{\mathcal{M}}(p) \, .
 \end{align}
 Likewise, we find
 \begin{equation}\label{zoveel235}
 e_{\mathcal{M}}(p) =  e_{\mathcal{M}}(F(q)) = F(e_{\mathcal{F}}(q)) = F(q) = p\, .
 \end{equation}
 As equations \eqref{zoveel234} and \eqref{zoveel235} hold for any node $p$ in $\mathcal{M}$, we conclude that $\sigma_{\mathcal{M}} \circ \tau_{\mathcal{M}} = \kappa_{\mathcal{M}}$ and $e_{\mathcal{M}} = \Id_{\mathcal{M}}$. In particular, for any input network $\mathcal{N}_p$ of a node $p$ in a constructible network $\mathcal{N}$, we find $\sigma_{\mathcal{N}_p} \circ \tau_{\mathcal{N}_p} = \kappa_{\mathcal{N}_p}$ and $e_{\mathcal{N}_p} = \Id_{\mathcal{N}_p}$. Lastly, as we have $\iota_{\mathcal{N}_p} = (\iota_{\mathcal{N}})\mid_{\mathcal{N}_p}$ for any $\iota \in \Sigma$, we conclude that $\sigma_{\mathcal{N}} \circ \tau_{\mathcal{N}} = \kappa_{\mathcal{N}}$ and that $e_{\mathcal{N}} = \Id_{\mathcal{N}}$. This completes the proof.
\end{proof}

\noindent Let $\mathcal{N} \in \mathfrak{C}_{\Sigma}$ be a constructible network for the monoid $\Sigma$. It follows from Lemma \ref{stillcloded1} that we may construct a sub-monoid $\Sigma_{\mathcal{N}}$ of $\Sigma$ whose elements are exactly the input maps $\sigma_{\mathcal{N}}$ for $\sigma \in \Sigma$. More precisely, we have a surjective morphism of monoids from $\Sigma$ to $\Sigma_{\mathcal{N}}$ given by $\sigma \mapsto \sigma_{\mathcal{N}}$. The consequences of such a morphism are explored in \cite{proj}, but for now it is enough to realise that the class of (linear) admissible maps for a network does not change if we discard double input maps. As a result, we see that the admissible maps for $\mathcal{N}$ are the same as those for some complete network (with corresponding monoid of maps $\Sigma_{\mathcal{N}}$). Moreover, it is shown in \cite{CCN} that the (linear) admissible maps for a complete network are closed under composition. We conclude that the same holds for the admissible maps of $\mathcal{N} \in \mathfrak{C}_{\Sigma}$. In other words, for any two linear response functions $f,g: V^{\Sigma} \rightarrow V$ there exists a (possibly non-unique) linear response function $h: V^{\Sigma} \rightarrow V$ such that $\gamma^{\mathcal{N}}_f \circ \gamma^{\mathcal{N}}_g = \gamma^{\mathcal{N}}_h$, where $\gamma^{\mathcal{N}}_{\bullet}$ denotes an admissible map for $\mathcal{N}$.  As we will need this result later on, we summarise it in a corollary.

\begin{cor}\label{identityplusmulti}
Let $\mathcal{N} \in \mathfrak{C}_{\Sigma}$ be a constructible network for the monoid $\Sigma$, and denote by $\gamma^{\mathcal{N}}_{f}$ an admissible map for $\mathcal{N}$ corresponding to the response function $f$. For any two linear response functions $f,g: V^{\Sigma} \rightarrow V$, there exists a linear response function $h: V^{\Sigma} \rightarrow V$ such that $\gamma^{\mathcal{N}}_f \circ \gamma^{\mathcal{N}}_g = \gamma^{\mathcal{N}}_h$. 
\end{cor}

\begin{remk}[Motifs]\label{motifs}
The definition of $\mathfrak{C}_{\mathcal{F}}$ may be interpreted as describing all networks that can be made by gluing together networks from a certain finite set of \textit{motifs}. More precisely, these motifs are exactly all quotient networks of ${\mathcal{F}}$, and gluing means identifying subnetworks.   $\hfill \triangle$ 
\end{remk}

%%%%%%%%%%%%%%%%%%%%%%Examples%%%%%%%%%%%%%%%%%%%%%%%%%%%%%%%%%

\section{Examples} \label{Examples}
Now that we have explained all the concepts in Theorem \ref{main1}, we may illustrate it further with some examples.
\begin{ex}\label{exam6}
We return to our running example, given by the left hand side of Figure \ref{fig3}. Example \ref{exam3} describes the linear admissible maps for this network and its fundamental network, given by the right side of Figure \ref{fig3}. They are given by
\begin{equation}\label{gammafex1222}
\arraycolsep=1.4pt\def\arraystretch{1}
\gamma_f = \begin{bmatrix}
A + \roo{B} + \gr{C}&  \ora{E} & \bl{D}\\
\gr{C} & A+ \ora{E} & \roo{B} + \bl{D} \\
\gr{C} &\roo{B} +  \ora{E}& A+ \bl{D} 
\end{bmatrix} \, 
\end{equation}
and
\begin{equation}\label{gammafex14212}
\arraycolsep=1.4pt\def\arraystretch{1}
\Gamma_f = \begin{bmatrix}
A & \roo{B} & \gr{C} &  \bl{D} & \ora{E} \\
 \roo{B} & A& \gr{C} &  \bl{D} & \ora{E} \\
0 & 0 & A + \roo{B} + \gr{C} &  \bl{D} & \ora{E} \\
0 & 0 & \gr{C} &  A + \bl{D} & \roo{B} + \ora{E} \\
0 & 0 & \gr{C} & \roo{B} + \bl{D} & A + \ora{E} 
\end{bmatrix} \, .
\end{equation}
Here we have written a general linear response function $f$ as 
\begin{equation}\label{lini0122}
f(w) =Aw_1 + \roo{B}w_2 + \gr{C}w_3 + \bl{D}w_4 + \ora{E}w_5\, ,
\end{equation}
for $w = (w_1, w_2, w_3, w_4, w_5) \in V^5$ and with $A, \roo{B}, \gr{C}, \bl{D}$ and  \ora{E} linear maps from $V$ to $V$. Restricting $\gamma_f$ to the robust synchrony space $S = \{x_2 = x_3\}$ gives rise to an admissible map for the corresponding quotient network. It is given by 
\begin{equation}\label{gammafex12222}
\arraycolsep=1.4pt\def\arraystretch{1}
\gamma^{\mathcal{S}}_f := \gamma_f|_S = \begin{bmatrix}
A + \roo{B} + \gr{C}&  \ora{E} + \bl{D}\\
\gr{C} & A+ \ora{E} + \roo{B} + \bl{D} \\
\end{bmatrix} \, \, .
\end{equation}
As always, we have the network multiplier $\Lambda^1(A, \dots, \ora{E}) = A + \dots + \ora{E}$. Moreover, we see that 
\begin{equation}
\tr(\gamma^{\mathcal{S}}_f) = 2A + 2\roo{B} + \gr{C} + \bl{D} + \ora{E} =  (A + \roo{B} + \gr{C} + \bl{D} + \ora{E}) + (A+\roo{B})\, ,
\end{equation}
in the case of $V=\C$. Hence, we find a second network multiplier given by $\Lambda^2(A, \dots, \ora{E}) = A + \roo{B}$. Likewise, we have 
\begin{equation}
\tr(\gamma_f) = 3A + \roo{B} + \gr{C} + \bl{D} + \ora{E} =  (A + \roo{B} + \gr{C} + \bl{D} + \ora{E}) + (A+\roo{B}) +  (A-\roo{B})\,.
\end{equation}
As the first two network multipliers contributed to the spectrum of $\gamma^{\mathcal{S}}_f$, they also contribute to the spectrum of $\gamma_f$ (see Theorem \ref{main1}, point 5). Therefore, we find that a third network multiplier is given by $\Lambda^3(A, \dots, \ora{E}) = A - \roo{B}$. Setting $m := \dim(V)$, we conclude that the eigenvalues of the $3m \times 3m$ matrix \eqref{gammafex1222} are given by those of $\Lambda^1(A, \dots, \ora{E})$ (one time), those of $\Lambda^2(A, \dots, \ora{E})$ (one time) and those of $\Lambda^3(A, \dots, \ora{E})$ (one time). Note that we may have also gotten this result by observing that we have a sequence of invariant spaces $\{x_1 = x_2 = x_3\} \subset \{x_2 = x_3\} \subset V^3$. However, we may now use the network multipliers to obtain information on the spectrum of admissible maps for other networks in $\mathfrak{C}_{\mathcal{N}}$. For example, as we may write
\begin{align}
\tr(\Gamma_f) &= 5A + \roo{B} + \gr{C} + \bl{D} + \ora{E} \\ \nonumber
&=  (A + \roo{B} + \gr{C} + \bl{D} + \ora{E}) + 2(A+\roo{B}) + 2(A-\roo{B}) \\ \nonumber
&= 1\cdot \Lambda^1(A, \dots, \ora{E}) + 2\cdot\Lambda^2(A, \dots, \ora{E})+2\cdot\Lambda^3(A, \dots, \ora{E}) \, ,
\end{align}
we can immediately conclude that the eigenvalues of the $5m \times 5m$ matrix \eqref{gammafex14212} are given by those of $\Lambda^1(A, \dots, \ora{E})$ (one time), those of $\Lambda^2(A, \dots, \ora{E})$ (two times) and those of $\Lambda^3(A, \dots, \ora{E})$ (two times). \hfill $\triangle$
\end{ex}
\noindent It should be clear at this point that the trace of a linear admissible map plays a key role in our analysis. The following proposition allows us to directly obtain this trace from the corresponding network graph. 
\begin{prop}\label{traceisfixedloops}
Given a (not necessarily complete) homogeneous coupled cell network $\mathcal{N} = (N, \mathcal{T})$ and a linear response function $f: \C^I \rightarrow \C$ with coefficients $(c_{\sigma})_{\sigma \in \mathcal{T}}$, the trace of the corresponding admissible map $\gamma_f$ is given by
\begin{align}
\tr(\gamma_f) 	&= \sum_{\sigma \in \mathcal{T}} c_{\sigma} \#\{ p \in N\, \mid \, \sigma(p) = p\}  \\ \nonumber
			&= \sum_{\sigma \in \mathcal{T}} c_{\sigma} [\text{``\hspace{0.03cm}number of self loops of colour $\sigma$''}] \, .
\end{align}
\end{prop}

\begin{proof}
Given a node $p \in N$, denote by $\delta_p \in \C^{N}$ the element given by $(\delta_p)_q = \delta_{p,q}$ for all $q \in N$. It follows that
\begin{align}
\tr(\gamma_f)	&= \sum_{p \in N} (\gamma_f(\delta_p))_p =  \sum_{p \in N} (f \circ \pi_p )(\delta_p)  =  \sum_{p \in N}\sum_{\sigma \in \mathcal{T}} c_{\sigma}(\pi_p (\delta_p))_{\sigma} \\ \nonumber
			&=  \sum_{p \in N}\sum_{\sigma \in \mathcal{T}} c_{\sigma}(\delta_p)_{\sigma(p)} =  \sum_{\sigma \in \mathcal{T}} c_{\sigma}\sum_{p \in N}(\delta_{p,\sigma(p)}) \\ \nonumber
			&=  \sum_{\sigma \in \mathcal{T}} c_{\sigma}\#\{ p \in N\, \mid \, \sigma(p) = p\} \, .
\end{align}
This proves the proposition.
\end{proof}

\begin{ex}\label{exam7}
\noindent We revisit the class of constructible networks $\mathfrak{C}_{\Sigma}$ from Example \ref{exam6}. Recall that the left side of Figure \ref{fig5} shows a network in $\mathfrak{C}_{\Sigma}$, which we will call ${\mathcal{M}}$. As before, we may try to express the trace of an admissible map $\gamma^{\mathcal{M}}_f$ for ${\mathcal{M}}$ as a combination of the network multipliers we found in Example \ref{exam6}. If we manage, then we can describe the spectrum of $\gamma^{\mathcal{M}}_f$ in terms of the spectrum of the network multipliers as before. (If we do not manage, then it proves that we have not yet found all the network multipliers of $\mathfrak{C}_{\Sigma}$.) Using the result of Proposition \ref{traceisfixedloops}, we see from Figure \ref{fig5} that 
\begin{align}
\tr(\gamma^{\mathcal{M}}_f) &= 5A + 3\roo{B} + \gr{C} + \bl{D} + \ora{E}  \\ \nonumber
&= (A + \roo{B} + \gr{C} + \bl{D} + \ora{E} ) + 3(A + \roo{B}) + (A - \roo{B}) \\ \nonumber
&= 1\cdot \Lambda^1(A, \dots ,\ora{E})  +3\cdot \Lambda^2(A, \dots ,\ora{E})+1\cdot\Lambda^3(A, \dots ,\ora{E})\, .
\end{align}
Therefore, the eigenvalues of $\gamma^{\mathcal{M}}_f$ are given exactly by those of the block combinations $A + \roo{B} + \gr{C} +  \bl{D} + \ora{E}$ (one time), $A + \roo{B}$ (three times) and $A - \roo{B}$ (one time). Note that we did not even have to write down the linear admissible map $\gamma^{\mathcal{M}}_f$ to get to this result. \\
One can see that $\gamma^{\mathcal{M}}_f$ is given explicitly by
\begin{equation}\label{gammafex142212}
\arraycolsep=1.4pt\def\arraystretch{1}
\gamma^{\mathcal{M}}_f = \begin{bmatrix}
A & \roo{B} & \gr{C} &  \bl{D} + \ora{E} & 0 \\
 \roo{B} & A& \gr{C} &  \bl{D} + \ora{E} & 0 \\
0 & 0 & A + \roo{B} + \gr{C} &  \bl{D} + \ora{E} & 0 \\
0 & 0 & \gr{C} &  A + \roo{B} + \bl{D}  + \ora{E} & 0 \\
0 & 0 & \gr{C} &  \ora{E} + \bl{D} &A + \roo{B}  
\end{bmatrix} \, .
\end{equation}
Of course, we may set $\pa{F}: = \bl{D} + \ora{E}$, so that we have proven that the eigenvalues of a block matrix of the form
\begin{equation}\label{gammafex1422}
\arraycolsep=1.4pt\def\arraystretch{1}
\tilde{\gamma}^{\mathcal{M}}_f = \begin{bmatrix}
A & \roo{B} & \gr{C} &  \pa{F} & 0 \\
 \roo{B} & A& \gr{C} &  \pa{F} & 0 \\
0 & 0 & A + \roo{B} + \gr{C} &  \pa{F} & 0 \\
0 & 0 & \gr{C} &  A + \roo{B} + \pa{F} & 0 \\
0 & 0 & \gr{C} & \pa{F} & A + \roo{B}  
\end{bmatrix} \, ,
\end{equation}
are given by those of the block combinations $A + \roo{B} + \gr{C} +  \pa{F}$ (one time), $A + \roo{B}$ (three times) and $A - \roo{B}$ (one time). \hfill $\triangle$
\end{ex}
\noindent Next, we state two more properties of network multipliers that may prove useful.
\begin{thr}
In addition to the properties listed in Theorem \ref{main1}, the network multipliers of $\mathfrak{C}_{\mathcal{N}}$ satisfy the following properties.
\begin{itemize}\label{main2}
\item[6] Let $\mathcal{F}$ denote the fundamental network of $\mathcal{N}$. As mentioned before, we have $\mathcal{F} \in \mathfrak{C}_{\mathcal{N}}$. Moreover, it holds that $m_l^{\mathcal{F}} \not= 0$ for all $l \in \{1, \dots, k\}$. In other words, every network multiplier of $\mathfrak{C}_{\mathcal{N}}$ contributes to the spectrum of $\Gamma_f$.
\item[7] Given linear response functions $f, g: V^{\Sigma} \rightarrow V$, there exists a unique linear response function $h$ such that $\Gamma_f \circ \Gamma_g = \Gamma_h$. It then also holds that $\gamma^{\mathcal{M}}_f \circ \gamma^{\mathcal{M}}_g = \gamma^{\mathcal{M}}_h$, with $\gamma^{\mathcal{M}}_{\bullet}$ an admissible map for any network $\mathcal{M} \in \mathfrak{C}_{\mathcal{N}}$. Moreover, if we denote by $\mathcal{C}_f$ the coefficients of $f$ (with similar notation for $g$ and $h$), then we have $\Lambda^l(\mathcal{C}_f)\circ \Lambda^l(\mathcal{C}_g) = \Lambda^l(\mathcal{C}_h)$ for all $l \in \{1, \dots, k\}$.
\end{itemize}
\end{thr} 
\noindent Point 6 allows us to verify that we have found all the network multipliers of $\mathfrak{C}_{\mathcal{N}}$. In particular, it follows that in Example \ref{exam6} we have found all of them. Point 7 may be a powerful tool in further analysing linear admissible maps. This property is what justifies the name `network multiplier'.

\begin{ex}\label{exam8}
We return to the setup of examples \ref{exam6} and \ref{exam7}. Let us denote the coefficients of a response function $f: V^{\Sigma} \rightarrow V$ by $(A,\roo{B},\gr{C},\bl{D},\ora{E})$ and those of $g: V^{\Sigma} \rightarrow V$ by $(A',\roo{B'},\gr{C'}, \bl{D'},\ora{E'})$. We write  $(A'',\roo{B''},\gr{C''},\bl{D''},\ora{E''})$ for the coefficients of $h: V^{\Sigma} \rightarrow V$, defined by the equation  $\Gamma_f\circ\Gamma_g = \Gamma_h$. One can show that the coefficients of $h$ are given explicitly by
\begin{align}
&A'' = AA' + \roo{B}\roo{B'} \, \quad \roo{B''} = A\roo{B'}  + \roo{B}A' \\ \nonumber
&\gr{C''}  = (A + \roo{B} + \gr{C} +  \bl{D} + \ora{E})\gr{C'} + \gr{C}(A' + \roo{B'}) \\ \nonumber
&\bl{D''}  = (A + \roo{B} + \gr{C} +  \bl{D} + \ora{E})\bl{D'} + \bl{D}A' + \ora{E}\roo{B'} \\ \nonumber
&\ora{E''}  = (A + \roo{B} + \gr{C} +  \bl{D} + \ora{E})\ora{E'} + \bl{D}\roo{B'}  + \ora{E}A' \, .
\end{align}
Using these expressions for $A''$ up to $\ora{E''}$, one verifies that indeed 
\begin{equation}
\Lambda^i(A, \dots ,\ora{E})\Lambda^i(A', \dots ,\ora{E'}) = \Lambda^i(A'', \dots ,\ora{E''})\, ,
\end{equation}
for all $i \in \{1,2,3\}$. \hfill $\triangle$
\end{ex}

\begin{ex}\label{exam9}
We consider the network $\mathcal{N} = (N, \Sigma)$ with set of nodes $N=\{1,2,3\}$ and monoid $\Sigma = \{[1,2,3], [2,2,3], [1,1,1], [2,2,2], [3,3,3]\}$. The corresponding linear admissible maps are given by 
\begin{equation}\label{gammafex14262}
\arraycolsep=1.4pt\def\arraystretch{1}
{\gamma}_f = \begin{bmatrix}
A + \gr{C} & \roo{B} + \bl{D} & \ora{E} \\
 \gr{C} &A +  \roo{B}+ \bl{D} & \ora{E} \\
\gr{C} & \bl{D} & A + \roo{B} + \ora{E} 
\end{bmatrix} \, .
\end{equation}
Again, we may use the existence of a  robust synchrony space to obtain some of the network multipliers. It can be seen that the synchrony space $S = \{x_1 = x_2\}$ is indeed robust. Restricting $\gamma_f$ to $S$ then yields the corresponding quotient network
\begin{equation}\label{gammafex14262w}
\arraycolsep=1.4pt\def\arraystretch{1}
{\gamma}^{\mathcal{S}}_f:= \gamma_f|_S = \begin{bmatrix}
A + \roo{B} + \gr{C}+ \bl{D}   & \ora{E} \\
\gr{C} + \bl{D} & A + \roo{B} + \ora{E} 
\end{bmatrix} \, .
\end{equation}
As a quotient of a constructible network is again constructible, we know that the trace of ${\gamma}^{\mathcal{S}}_f$ may be uniquely written as the sum of (the trace of) two network multipliers for $\mathfrak{C}_{\Sigma}$. One of these network multipliers is given by $\Lambda^1(A, \dots ,\ora{E}) = A + \roo{B} + \gr{C} + \bl{D}  + \ora{E}$, and by looking at the trace of ${\gamma}^{\mathcal{S}}_f$ we see that another one is given by 
\begin{align}
\Lambda^2(A, \dots ,\ora{E}) &= (A + \roo{B} + \gr{C}+ \bl{D})+ (A + \roo{B} + \ora{E} )  \\ \nonumber
&- (A + \roo{B} + \gr{C} + \bl{D}  + \ora{E}) \\ \nonumber
&= A + \roo{B}\, .
\end{align}
Now let us return to equation \eqref{gammafex14262}. As before, we may write

\begin{equation}
\tr(\gamma_f) = \Lambda^1(A, \dots ,\ora{E}) + \Lambda^2(A, \dots ,\ora{E}) + \Lambda^i(A, \dots ,\ora{E})\, ,
\end{equation}
for some $i \in \{1, \dots, k\}$ (not excluding $i=1$ or $i=2$). We find a new network multiplier $\Lambda^3$, given by
\begin{equation}
\Lambda^3(A, \dots ,\ora{E}) = A\, .
\end{equation}
Therefore, we conclude that the eigenvalues of a block matrix of the form \eqref{gammafex14262w} are given by those of $A$, $A + \roo{B}$ and $A + \roo{B} + \gr{C} + \bl{D}  + \ora{E}$. Next, we want to verify that we have found all the network multipliers. We may determine the trace of an admissible map $\Gamma_f$ for the fundamental network of $\mathcal{N}$, by simply counting solutions $\tau \in \Sigma$ to the equation $\sigma \tau = \tau$ for each $\sigma \in \Sigma$. See Proposition \ref{traceisfixedloops}, (where we may replace $p$ by $\tau$). We find
\begin{align}
\tr(\Gamma_f) &= 5A + 3\roo{B} + \gr{C} + \bl{D}  + \ora{E} \\ \nonumber
&= \Lambda^1(A, \dots ,\ora{E}) + 2\cdot \Lambda^2(A, \dots ,\ora{E}) +  2\cdot \Lambda^3(A, \dots ,\ora{E}) \, .
\end{align}
We conclude that the eigenvalues of $\Gamma_f$ are given by those of $\Lambda^1$ (one time), $\Lambda^2$ (two times) and $\Lambda^3$ (two times). Moreover, we see that $\Lambda^1$, $\Lambda^2$ and $\Lambda^3$ are all the network multipliers of $\mathfrak{C}_{\Sigma}$. By the relatively easy form of the $\Lambda^i$, we  may now describe the eigenvalues of an admissible map for any network $\mathcal{M} \in \mathfrak{C}_{\Sigma}$ by just looking at its graph. More precisely, these eigenvalues are given by those of the $\Lambda^i$, ($m^{\mathcal{M}}_i$ times) for $i \in \{1,2,3\}$, where we may give an exact description of the $m^{\mathcal{M}}_i$. Namely, we have that $m^{\mathcal{M}}_1$ equals the number of green self-loops in the graph of $\mathcal{M}$, that $d^{\mathcal{M}}_2$ equals the number of red self loops minus the number of green self-loops, and that $d^{\mathcal{M}}_3$ is equal to the total number of cells in $\mathcal{M}$ minus $d^{\mathcal{M}}_1+d^{\mathcal{M}}_2$. We will use this example to study larger networks in Example \ref{exam11}. $\hfill \triangle$
\end{ex}

\begin{ex}\label{exam10}
Given a node $p$ in a network $\mathcal{N} = (N, \Sigma)$, let us denote by $\theta_p := [p,p, \dots p]$ the input map that sends every node to $p$. Suppose furthermore that $\tau$ is an input map satisfying $\tau^M = \theta_q$ for some $M \in \N$ and $q \in N$, though let us assume that $\tau \not= \theta_q$. We let $\Sigma$ be the smallest monoid of maps from $N$ to $N$ containing $\tau$ and all maps $\theta_p$ for $p \in N$, and denote by $\mathcal{F} = (\Sigma, \Sigma)$ the corresponding (complete) fundamental network. \\
It can be seen that $\Sigma$ is given explicitly by $\Id_N$ together with all powers of $\tau$ and all maps $\theta_p$. Because of this, the only invertible element of $\Sigma$ is $\Id_N$. We claim that a robust synchrony space of $\C^{\Sigma}$ is therefore given by 
\begin{equation}
S:= \{X_{\kappa} = X_{\tau} \, \mid \, \forall \kappa, \tau \not= \Id_N\}\, .
\end{equation}
Indeed, for $\Id_N$ there is nothing to check, and multiplication by a non-invertible input map from the left produces only non-invertible input maps. Let us denote the two nodes of the corresponding quotient network of $S$ by $[\Id_N]$ and $[\tau]$ (representing the equivalence class containing $\Id_N$ and $\tau$, respectively). Given a linear response function $f: \C^{\Sigma} \rightarrow \C$ we will write $\gamma_f$, $\Gamma_f$ and $\gamma^{\mathcal{S}}_f$ for the corresponding admissible maps for the original network (with set of nodes $N$), the fundamental network and the two-cell network corresponding to the synchrony space $S$, respectively. Let us also write $c = (c_{\sigma})_{\sigma \in \Sigma} \in \C^{\Sigma}$ for the coefficients of $f$. Using Proposition \ref{traceisfixedloops}, we may find the trace of $\gamma_f^{\mathcal{S}}$ by counting solutions to the equation $\sigma([\kappa]) = [\kappa]$ for $[\kappa] \in \{[\Id_p], [\tau] \}$ and for each $\sigma \in \Sigma$. The equation $\Id([\kappa]) = [\kappa]$ holds true for every $[\kappa] \in \{[\Id_N], [\tau] \}$, whereas $\sigma([\kappa]) = [\kappa]$ only holds for $[\kappa] = [\tau]$ if $\sigma$ is non-invertible (as composition from the left by a non-invertible input map always yields a non-invertible input map). We conclude that
\begin{equation}\label{trac384765}
\tr(\gamma_f^{\mathcal{S}}) = c_{\Id} + \sum_{\sigma \in \Sigma} c_{\sigma}\, .
\end{equation}
We furthermore know that 
\begin{equation}
\tr(\gamma_f^{\mathcal{S}}) = \Lambda^1(c) + \Lambda^i(c)\,
\end{equation}
for some $i \in \{1, \dots, k\}$ (not excluding $i=1$), and with 
\begin{equation}
\Lambda^1(c) = \sum_{\sigma \in \Sigma} c_{\sigma}\, .
\end{equation}
Hence, we conclude from equation \eqref{trac384765} that $\Lambda^2(c) = c_{\Id}$. Therefore, we have already found two network multipliers of $\mathfrak{C}_{\mathcal{F}}$. Returning to the original network, we see that for every $K \in \N$ we have
\begin{equation}
\tau^K(q) = \tau^{K}(\theta_q(q)) = \tau^{K}(\tau^M(q)) = \tau^{M}(\tau^K(q)) = \theta_q(\tau^K(q)) = q\,.
\end{equation}
Therefore, there exists at least one node $p \in N$ such that $\tau^K(p) = p$. Moreover, if $p \in N$ satisfies $\tau^K(p) = p$, then
\begin{equation}
p = \tau^K(p) = \tau^{2K}(p) = \dots = \tau^{MK}(p) = \theta_q(\tau^{M(K-1)}(p)) = q \, .
\end{equation}
We conclude that for every $K \in \N$ there is exactly one solution $p \in N$ to the equation $\tau^K(p) = p$. Likewise, for a given $r \in N$, the equation $\theta_r(p) = p$ has the unique solution $p=r$. Finally, $\Id_N(p) = p$ of course has $\#N$ solutions. From this we see that 
\begin{equation}\label{trac3q84765}
\tr(\gamma_f) = (\#N-1)c_{\Id} + \sum_{\sigma \in \Sigma} c_{\sigma}  = (\#N-1)\Lambda^2(c) + \Lambda^1(c)\, .
\end{equation}
We conclude that the eigenvalues of $\gamma_f$ are given exactly by $\Lambda^1(c)$ (one time) and $\Lambda^2(c)$ ($(\#N-1)$ times). This result still holds if we allow $f$ to be a map from $V^{\Sigma}$ to $V$ with each $c_{\sigma}$ a linear map from $V$ to $V$. Moreover, the conclusion holds regardless of whether or not $\mathfrak{C}_{\mathcal{F}}$ has any other network multipliers (perhaps even with $n_l >> 1$), as an expression of $\tr(\gamma_f)$ in multiples of the traces of the multipliers is unique. However, using the same technique as with $\gamma_f$ one easily verifies that
\begin{equation}\label{trac3q84765}
\tr(\Gamma_f) =  \Lambda^1(c) + (\#\Sigma-1)\Lambda^2(c)\, ,
\end{equation}
so that $\Lambda^1$ and $\Lambda^2$ are in fact the only network multipliers of $\mathfrak{C}_{\mathcal{F}}$. This implies that the eigenvalues of any admissible map for a constructible network are given by those of $\Lambda^1$ and possibly $\Lambda^2$. If $\Lambda^2$ is not involved then every arrow in the network is a self-loop, meaning that the network is the disjoint union of all single cell networks. \\
As an example, it follows that the eigenvalues of the block matrix
\begin{equation}\label{gammafex142we2}
\arraycolsep=1.4pt\def\arraystretch{1}
{\gamma}_f = \begin{bmatrix}
A + \ora{E} & \roo{B} + \pa{F} & \gr{C} + \geel{G} &  \bl{D} \\
\ora{E} & A+ \pa{F}& \roo{B}+\geel{G}& \gr{C} + \bl{D}   \\
\ora{E} & \pa{F} & A + \geel{G}  &\roo{B}+ \gr{C} + \bl{D}   \\
\ora{E} & \pa{F} & \geel{G} &  A + \roo{B} + \gr{C} +\bl{D} 
\end{bmatrix} \, ,
\end{equation}
corresponding to $\tau = [2,3,4,4]$, are given by those of the matrices $A + \roo{B} + \gr{C} + \bl{D} + \ora{E} + \pa{F} + \geel{G}$ and $A$.
$\hfill \triangle$
\end{ex}

\begin{ex}\label{exam11}
Just as in the previous example, suppose $\Sigma$ is generated by the elements $\theta_p$ for $p \in N$ and another element $\tau$ satisfying $\tau^M = \tau^{M+1}$ for some $M \in \N$. We  furthermore assume $\tau$ to be non-invertible. It follows that $\tau^M$ is the projection onto a (strict) subset of nodes   $\{q_1, \dots q_s\} \subset N$. We have already covered the case $s=1$ in Example \ref{exam10}, so assume $1 < s < \#N$. It follows that there is a partition of the nodes $N = Q_1 \sqcup \dots \sqcup Q_s$ with $q_i \in Q_i$ and with $\tau^M(p) = q_i$ for all $p \in Q_i$, for all $i \in \{1, \dots, s\}$. We may assume without loss of generality that $Q_1$ contains more than one element. As a result, $Q_1$ contains at least one element $q$ that is not in the image of $\tau$. We claim that the partition of the set of nodes into three classes given by 
\begin{equation}
N = (\{q\}) \bigsqcup (Q_1 \setminus \{q\}) \bigsqcup (Q_2 \sqcup \dots \sqcup Q_s) := N_1 \sqcup N_2 \sqcup N_3
\end{equation}
is balanced. Indeed, one readily verifies that $\tau$ respects this partition, from which it follows that all powers of $\tau$ do. Moreover, any input map of the form $\theta_p$ for $p \in N$ in fact respects any synchrony space. In the corresponding quotient network, the identity map acts as the identity on the three nodes, all (positive) powers of $\tau$ act as the map $(N_2, N_2, N_3)$ and every map $\theta_p$ becomes a map $\theta_{N_i}$ for some $i \in \{1,2,3\}$. In other words, the quotient network is the same as the network discussed in Example \ref{exam9} (possibly after identifying certain input maps). Proceeding the same way as in Example \ref{exam9}, we obtain the three network multipliers 
\begin{align}
\Lambda^1(c) = \sum_{\sigma \in \Sigma} c_{\sigma}, \, \Lambda^2(c) = \sum_{i = 0}^M c_{\tau^i} \text{ and } \Lambda^3(c) = c_{\Id} \,,
\end{align}
where $M \in \N$ is assumed minimal subject to $\tau^M = \tau^{M+1}$. It remains to look at the trace of an admissible map $\gamma_f$ for the original network, in the hope that it may be expressed using only the three network multipliers we have found. The equation $\theta_p(r) = r$ for $r \in N$ has exactly one solution for every node $p \in N$. For any positive power of $\tau$, we find precisely the nodes $q_1$ to $q_s$, and the identity map of course fixes all nodes. We conclude that
\begin{align}\label{trac3q8q4765}
\tr(\gamma_f) &= \sum_{\sigma \in \Sigma} c_{\sigma} + (s-1)\sum_{i=0}^M c_{\tau^i}+ (\#N-s)c_{\Id} \\ \nonumber
&=  \Lambda^1(c) + (s-1)\Lambda^2(c) + (\#N-s)\Lambda^3(c)\, .
\end{align}
Hence, the eigenvalues of $\gamma_f$ are given by those of $\Lambda^1(c)$ (one time), those of $\Lambda^2(c)$ ($s-1$ times) and  those of $\Lambda^3(c)$ ($\#N-s$ times).\\
As an example, we have that the eigenvalues of the block matrix
\begin{equation}\label{gammafex18942we2}
\arraycolsep=1.4pt\def\arraystretch{1}
{\gamma}_f = \begin{bmatrix}
A + \gr{C} &  \roo{B}+ \bl{D} &  \ora{E} & \pa{F} & \geel{G} & \gre{H} \\
\gr{C} & A +\roo{B}+ \bl{D} &  \ora{E} & \pa{F} & \geel{G} & \gre{H} \\
\gr{C} &  \roo{B} + \bl{D} & A + \ora{E} & \pa{F} & \geel{G} & \gre{H} \\
\gr{C} &  \bl{D} &  \ora{E} & A +\pa{F} & \roo{B} + \geel{G} & \gre{H} \\
\gr{C} &  \bl{D} &  \ora{E} & \pa{F} & A + \roo{B} + \geel{G} & \gre{H} \\
\gr{C} &  \bl{D} &  \ora{E} & \pa{F} & \geel{G} &A +\roo{B}+ \gre{H} \\
\end{bmatrix} \, ,
\end{equation}
corresponding to $\tau = [2,2,2,5,5,6]$, are given by those of the matrices $A + \roo{B} + \gr{C} + \bl{D} + \ora{E} + \pa{F} + \geel{G} + \gre{H}$ (one time), $A + \roo{B}$ (two times) and $A$ (three times). $\hfill \triangle$
\end{ex}
\noindent Up to this point, we have only encountered classes of constructible networks with $n_l = 1$ for all $l \in \{1, \dots, k\}$. That is, we have only seen network multipliers of size one. The following example, studied in Section 3 of \cite{Schwenker},  shows that network multipliers can come as matrices of size greater than one as well.
\begin{ex}\label{exam12}
We study the fundamental network $\mathcal{F} = (\Sigma, \Sigma)$ with linear admissible maps given by 
\begin{align}\label{gammafex18a942we2}
&{\Gamma}_f = \\ \nonumber
&\begin{pmatrix}
A  &  \roo{B} & \gr{C} & \bl{D} &  \ora{E} & \pa{F} & \geel{G} & \gre{H} \\
0  & \hspace{-0.2cm}A+ \ora{E} \hspace{-0.2cm}& 0 & \gr{C} & 0 & \roo{B} + \pa{F} &  \geel{G} & \bl{D} + \gre{H} \\
0  & 0  & \hspace{-0.2cm}A + \bl{D}\hspace{-0.2cm} & 0 & \roo{B} & \pa{F} & \gr{C} + \geel{G}&   \ora{E} + \gre{H} \\
0  & \roo{B} & 0  & \hspace{-0.2cm}A + \bl{D}\hspace{-0.2cm} & 0  & \pa{F} & \gr{C} + \geel{G}&   \ora{E} + \gre{H} \\
0 & 0 & \gr{C}  & 0 & \hspace{-0.2cm}A+ \ora{E}\hspace{-0.2cm} &  \roo{B} + \pa{F} & \geel{G} & \bl{D} + \gre{H} \\
0 & 0 & 0 & 0 & 0 &  \hspace{-0.2cm}A+  \roo{B} + \ora{E}  + \pa{F}\hspace{-0.2cm} & \geel{G} & \gr{C} + \bl{D} + \gre{H} \\
0 & 0 & 0 & 0 & 0 &   \pa{F} & \hspace{-0.3cm}A + \gr{C} + \bl{D}+ \geel{G}\hspace{-0.3cm} &    \roo{B} + \ora{E} + \gre{H} \\
0 & 0 & 0 & 0 & 0 &   \roo{B} +\pa{F} &  \gr{C} + \geel{G} &  \hspace{-0.3cm}  A + \bl{D} + \ora{E} + \gre{H}  \\
\end{pmatrix} \, . 
\end{align}
It can be shown that the network multipliers of $\mathfrak{C}_{\Sigma}$ are given by
\begin{align}
\Lambda^1(A, \dots, \gre{H}) &= A + \roo{B} + \dots + \gre{H} \\ \nonumber
\Lambda^2(A, \dots, \gre{H}) &= A \\ \nonumber
\Lambda^3(A, \dots, \gre{H}) &= 
\begin{pmatrix}
A  + \bl{D} & \roo{B}  \\
\gr{C} & A  +  \ora{E} \\
\end{pmatrix} \, . 
\end{align}
This result may be obtained by going through the analysis of the following sections (using the representation theory of $\Sigma$). Alternatively, we will describe more techniques for obtaining network multipliers in a follow-up article. These will allow us to simply read off the network multipliers of the matrix in this example. By looking at the trace of \eqref{gammafex18a942we2}, we conclude that the eigenvalues of $\Gamma_f$ are given by those of the block combinations $\Lambda^1$ and $\Lambda^2$ (both one time), together with those of the two by two block matrix $\Lambda^3$ (three times). Moreover, for any constructible network $\mathcal{N} = (N, \Sigma) \in \mathfrak{C}_{\Sigma}$, the eigenvalues of an admissible map are given by those of $\Lambda^1$ ($m^{\mathcal{N}}_1$ times), $\Lambda^2$ ($m^{\mathcal{N}}_2$ times), and $\Lambda^3$ ($m^{\mathcal{N}}_3$ times). Here, $m^{\mathcal{N}}_1$ through $m^{\mathcal{N}}_3$ may be described by
\begin{align}
m^{\mathcal{N}}_1 &= \#\{\text{self-loops in $\mathcal{N}$ of the arrow type corresponding to $\gre{H}$}\}\\ \nonumber
m^{\mathcal{N}}_3 &= \#\{\text{self-loops in $\mathcal{N}$ of the arrow type corresponding to $\bl{D}$}\} - m^{\mathcal{N}}_1\\ \nonumber
m^{\mathcal{N}}_2 &= \#N - m^{\mathcal{N}}_1 - 2m^{\mathcal{N}}_3\, ,
\end{align} 
among other ways.  $\hfill \triangle$ 
\end{ex}

\begin{remk}
The examples considered in this section demonstrate the use of theorems \ref{main1} and  \ref{main2} in studying (linear) control theory of network systems. More precisely, it follows that a linear admissible map $\gamma_f^{\mathcal{N}}$ for a network $\mathcal{N}$ is stable, if and only if the (often significantly smaller) network multipliers involved in describing the trace of $\gamma_f^{\mathcal{N}}$ are stable. Moreover, we see that the eigenvalues of $\gamma_f^{\mathcal{N}}$ depend on the coefficients of $f$ only through their dependence on the coefficients of the network multipliers. As these latter coefficients often span a strict subspace of the space of all coefficients, we may interpret Theorem \ref{main1} as a result on dimensional reduction as well. For example, even though the total space of linear response functions for the network of Example \ref{exam12} has dimension $8\dim(V)^2$, the coefficients of the network multipliers only span a space of dimension $6\dim(V)^2$. It also follows that the coefficients of the network multipliers are a more natural choice of variables than the coefficients of $f$ when considering stability. For instance, in Example \ref{exam6} we found the network multipliers $\Lambda^2(A, \dots, \ora{E}) = A + \roo{B}$  and $\Lambda^3(A, \dots, \ora{E}) = A - \roo{B}$. This means that a change in the block $\roo{B}$ might very well make one multiplier stable, while making the other unstable. A more natural choice (if available) might therefore be to change the blocks $A$ and $\roo{B}$ simultaneously. Likewise, if $\Lambda^1(A, \dots, \ora{E}) = A + \dots + \ora{E}$ is stable, even though the whole network map is not, then a change in $A$ or $\roo{B}$ is necessary to make the entire system stable. In other words, changing only the blocks $\gr{C}$, $\bl{D}$ and $\ora{E}$ cannot make the system stable in this case. $\hfill \triangle$ 
\end{remk}

\noindent In the next sections we present the proof of theorems \ref{main1} and \ref{main2}. It will turn out that network multipliers are a consequence of the symmetries of the fundamental network, see Lemma \ref{symmetryyeah}. Recall that these symmetries do not necessarily correspond to groups, but rather to monoids. For this reason, we will first describe the basics of monoid representation theory.

%%%%%%%%%%%%%%%%%%%%%%Representation Theory of Monoids%%%%%%%%%%%%%%%%%%%%%%%%%%%%%%%%%

\section{Representation Theory of Monoids} \label{Representation Theory of Monoids}
In this section we briefly introduce the representation theory of monoids. Most results will be presented without proof, as these can be found in for example \cite{RinkSanders3}. One minor difference is that the results in \cite{RinkSanders3} are presented over the real numbers, instead of the complex ones. However, one can copy most results almost verbatim. The only difference is that indecomposable representations over the real numbers can be classified into three types (real, complex or quaternionic), whereas only one type exists for indecomposable representations over the complex numbers (arguably best called complex type). We begin with the basic definitions.

\begin{defi}
A \textit{representation} of a monoid $\Sigma$ on a vector space $X$ is a set of linear maps $(A_{\sigma})_{\sigma \in \Sigma}$ from $X$ to itself such that 
\begin{itemize}
\item $A_e = \Id_X$ for $e$ the unit of $\Sigma$,
\item $A_{\sigma} \circ A_{\tau} = A_{\sigma \circ \tau}$ for all $\sigma, \tau \in \Sigma$.
\end{itemize}
In this article, we will furthermore only consider representations for which $X$ is a finite dimensional vector space over $\C$. We denote a representation of $\Sigma$ by $(X, (A_{\sigma})_{\sigma \in \Sigma})$, or just $X$ when the maps $(A_{\sigma})_{\sigma \in \Sigma}$ are clear from context. \\

\noindent Given two representations of $\Sigma$, $(X, (A_{\sigma})_{\sigma \in \Sigma})$ and $(Y, (A'_{\sigma})_{\sigma \in \Sigma})$, a \textit{homomorphism} from $X$ to $Y$ is a linear map $B: X\rightarrow Y$ such that
\begin{equation}
B \circ A_{\sigma} = A'_{\sigma} \circ B \text{ for all } \sigma \in \Sigma \, . 
\end{equation}
We denote the space of all homomorphisms between $X$ and $Y$ by $\Hom(X,Y)$, and write $\End(X)$ for $\Hom(X,X)$. The representations $X$ and $Y$ are said to be \textit{isomorphic} if there exists an invertible $F \in \Hom(X, Y)$. We then write $X \cong Y$. It can be shown that in that case, $F^{-1} \in \Hom(Y,X)$. In other words, the inverse of a homomorphism, if it exists, is a homomorphism as well. Likewise, for $\Sigma$-representations $X$, $Y$ and $Z$ and $G \in \Hom(X,Y)$, $H \in \Hom(Y,Z)$ we have $H \circ G \in \Hom(X,Z)$.\\

\noindent Given a representation $(X, (A_{\sigma})_{\sigma \in \Sigma})$, a linear subspace $U \subset X$ is called $\Sigma$-\textit{invariant} if we have $A_{\sigma}(U) \subset U$ for all $\sigma \in \Sigma$. In that case, $U$ naturally defines a representation of $\Sigma$ given by $(U, (A_{\sigma}|_{U})_{\sigma \in \Sigma})$, and we say that $U$ is a \textit{sub-representation} of $X$. The inclusion $i: U \rightarrow X$ is then readily seen to be a homomorphism. As opposed to compact group representations, an invariant space $U \subset X$ may not always have an invariant complementary space. That is, there may not exist an invariant space $W \subset X$ such that 
\begin{equation}\label{decompositiw12}
X = U \oplus W \, .
\end{equation}
If such a $W$ does exist then we will call $U$ a \textit{complementable} invariant subspace of $X$. In can be seen that the projection from $X$ onto $U$ with respect to the decomposition \eqref{decompositiw12} is then a homomorphism. \\

\noindent Examples of invariant subspaces are given by the image and kernel of a homomorphism and by the span of the eigenvectors of some eigenvalues of $F \in \End(X)$. None of these are necessarily complementable though. Examples of complementable subspaces are given by the generalised kernel of $F \in \End(X)$, and more generally by the span of the generalised eigenvectors of some eigenvalues of $F$. In other words, by the kernels of $F^n$ and $(F-\la_1\Id_X)^n\cdot (F-\la_2\Id_X)^n \dots (F-\la_k\Id_X)^n$ for $n$ the dimension of $X$ and $\la_1, \dots, \la_k \in \C$. Or equivalently, by the spaces corresponding to the Jordan blocks of $F$ for these eigenvalues. Invariant complements are then given by the span of the other generalised eigenvectors. The observation that not all linear subspaces are invariant or complementable may be seen as an explanation for the fact that symmetry often forces unusual Jordan normal forms, as the (generalised) eigenspaces are restricted by the symmetry.  $\hfill \triangle$
\end{defi}
\noindent Equation \eqref{decompositiw12} may be seen as not just an equality of vector spaces, but one of representations as well. More precisely, starting from two representations $(U, (B_{\sigma})_{\sigma \in \Sigma})$ and $(W, (C_{\sigma})_{\sigma \in \Sigma})$ of $\Sigma$, we may construct the representation $U \oplus W: = (U \oplus W, ({B}_{\sigma} + C_{\sigma})_{\sigma \in \Sigma})$. Here, ${B}_{\sigma} + C_{\sigma}$ is defined simply by $({B}_{\sigma} + C_{\sigma})(u,  w) = (B_{\sigma}u, C_{\sigma}w)$ for all $u \in U$, $w \in W$ and $\sigma \in \Sigma$.  In particular, if $U$ and $W$ are two complementary invariant subspaces of a representation $X$, then we may construct the representation $U \oplus W$ out of $U$ and $W$ with the representation structure induced by $X$. In that case equation \eqref{decompositiw12} holds as an isomorphism between representations of $\Sigma$. As such, representations may be understood as the direct sum of two or more invariant complementary subspaces. A special role is played by those representations of $\Sigma$ that cannot be decomposed in this fashion any further.
\begin{defi}
A (non-zero) representation $(X, (A_{\sigma})_{\sigma \in \Sigma})$ is said to be \textit{indecomposable} if it cannot be written in a nontrivial way as the direct sum of two  sub-representations. In other words, if we have 
\begin{equation}\label{decompositiwasd}
X = U \oplus W \, ,
\end{equation}
for $U, W \subset X$ invariant spaces then $U = X$ and $W = \{0\}$, or the other way around. $\hfill \triangle$
\end{defi}
\noindent Indecomposable representations may be seen as the `atoms' of a representation, as the following important theorem shows.

\begin{thr}[The Krull-Schmidt Theorem]\label{krullschmidt}
Any (finite dimensional) $\Sigma$-representation $X$ may be written as the direct sum of indecomposable representations:
\begin{equation}\label{decomposiitinv}
X \cong W_1 \oplus W_2 \oplus \dots \oplus W_k \, ,
\end{equation}
for $W_1$ through $W_k$ indecomposable. This decomposition is furthermore unique. In other words, if we also have
\begin{equation}\label{decomposiitinv2}
X \cong U_1 \oplus U_2 \oplus \dots \oplus U_l \, 
\end{equation}
with $U_1$ through $U_l$ indecomposable, then $k = l$ and we have $W_1 \cong U_1, W_2 \cong U_2, \dots, W_k \cong U_k$, after renumbering.
\end{thr}
\noindent A proof can be found in for example \cite{RinkSanders3}. The main reason for identifying indecomposable representations is that their algebra of endomorphisms is very well understood, as the following result shows.

\begin{lem}\label{endoindecomp}
Let $W$ be an indecomposable representation over the complex numbers. Every endomorphism $F \in \End(W)$ can be uniquely written as  $F = \la\Id_W + N$, for some $\la \in \C$ and nilpotent endomorphism $N \in  \End(W)$ (i.e. satisfying $N^n = 0$ for some $n \in \N$). In particular, any element of $\End(W)$ is either nilpotent or invertible. Moreover, the space $\nil(W) \subset \End(W)$ of nilpotent endomorphisms is a two-sided ideal of $\End(W)$. This means that for any $N,M \in \nil(W)$, $A \in \End(W)$ and $\la, \mu \in \C$, we have that $AN$, $NA$, $N + M$ and more generally $\la N + \mu M$ are all elements of $\nil(W)$.
\end{lem}

\begin{proof}
We will first show that any endomorphism of $W$ is either nilpotent or invertible. To this end, suppose $F \in \End(W)$ is non-invertible. For large enough $n \in \N$ we have that 
\begin{equation}\label{decomposiitinva3}
W = \ker(F^n) \oplus \im(F^n) \, .
\end{equation}
This can for instance be seen by looking at the Jordan normal form of $F$. As $W$ is indecomposable, and as both $\ker(F^n)$ and $\im(F^n)$ are $\Sigma$-invariant, either  $\ker(F^n)$ or $\im(F^n)$ equals $W$, with the other vanishing. Since $F$ was assumed non-invertible, it follows that  $\ker(F^n)$ is non-trivial and so equals the whole space $W$. This shows that $F$ is in fact nilpotent. We conclude that any endomorphism of an indecomposable representation is indeed either invertible or nilpotent. \\
Next, let $F$ be any endomorphism of $W$. If $\la \in \C$ is an eigenvalue of $F$, then $F - \la\Id_W$ is a non-invertible endomorphism of $W$. Hence $F - \la\Id_W$ must be nilpotent. Setting $N:= F - \la\Id_W$, we get the required expression $F = \la\Id_W + N$. \\
For uniqueness, suppose we have $F = \la\Id_W + N  = \mu\Id_W + N'$ with $\la, \mu \in \C$ and $N, N' \in \nil(W)$. It follows that 
\begin{equation}
\la = \tr(F)/\dim(W) = \mu \, ,
\end{equation}
from which we also see that $N = N'$.\\
Finally, suppose we are given $N \in \nil(W)$. Since any endomorphism is either invertible or nilpotent, it follows immediately that $AN, NA , \la N \in \nil(W)$ for all $A \in \End(W)$ and $\la \in \C$, since these cannot be invertible. It remains to show that $N+M \in \nil(W)$ whenever we have $N,M \in \nil(W)$. To this end, write $N+M = \la\Id_W + N'$ for some $N' \in \nil(W)$ and $\la \in \C$. It follows that
\begin{equation}
\la = \tr(N+M)/\dim(W) = (\tr(N) + \tr(M))/\dim(W) = 0 \, .
\end{equation}
Hence, we see that $N+M = N' \in \nil(W)$. This finishes the proof.
\end{proof}

\begin{remk}\label{remkiso}
Suppose we have $F, G \in \End(W)$ for $W$ indecomposable. We write $F = \la\Id + N$ and $G = \mu \Id + M$ for $\la, \mu \in \C$ and $N, M \in \nil(W)$. It follows from Lemma \ref{endoindecomp} that the unique expression as the sum of a multiple of the identity and a nilpotent homomorphism of $F+G$ is given by $(\la + \mu)\Id + (N+M)$. Likewise, $FG$ may be expressed in this way as $FG = (\la\Id + N)(\mu\Id + M) = (\la\mu)\Id + (\la M + \mu N + NM)$. It is also clear that $F = \la\Id + N$ is nilpotent if and only if we have $\la = 0$. Moreover, for any $\la \in \C$ we have $\la\Id \in \End(W)$. These results may be summarised by saying that there exists an isomorphism of complex algebras between $\End(W)/\nil(W)$ and $\C$. Explicitly, this isomorphism is given by $[\la\Id + N] = [\la\Id] \in \End(W)/\nil(W) \mapsto \la \in \C\,$. $\hfill \triangle$
\end{remk}

\begin{remk}\label{burrealcse}
Lemma \ref{endoindecomp} only holds because we consider representations over the complex numbers. If we consider indecomposables over the real numbers then we get one of three cases. The indecomposable $W$ is then either of real, complex or quaternionic type, depending on the structure of $\End(W)/\nil(W)$. It does remain true that any endomorphism of an indecomposable $W$ is either invertible or nilpotent though, with $\nil(W)$ an ideal of $\End(W)$. We will not pursue this any further here, but see for example \cite{Schwenker} or \cite{Gensym}. $\hfill \triangle$
\end{remk}
\noindent Lastly, we will use the following result.
\begin{lem}[The Fitting Lemma]\label{fitting}
Let $W$ and $\tilde{W}$ be two indecomposable representations of $\Sigma$, and let $F \in \Hom(W, \tilde{W})$ and $G \in \Hom(\tilde{W}, W)$ be two homomorphisms. If $G \circ F \in \End(W)$ is invertible, then $W$ and $\tilde{W}$ are isomorphic (with $F$ and $G$ both isomorphisms). In particular, if $W$ and $\tilde{W}$ are not isomorphic, then $G \circ F$ and $F \circ G$ are necessarily nilpotent.
\end{lem}
\noindent Note that the last part of Lemma \ref{fitting} follows from Lemma \ref{endoindecomp} or Remark \ref{burrealcse}, as $G \circ F$ and $F \circ G$ are nilpotent whenever they are not invertible. A proof of Lemma \ref{fitting} can be found in for example \cite{RinkSanders3}.

%%%%%%%%%%%%%%%%%Definition of Network Multipliers%%%%%%%%%%%%%%%%%%%%%%

%%%%%%%%%%%%%%%%%%%%%%%%%%%%%%appendixalgebra%%%%%%%%%Kroepoek

%Appendix; Algebraic Machinery

\section{Multipliers and Robust Subspaces} \label{Network Multipliers and Eigenvalues} 
In this section we define the so-called multipliers of a representation and relate them to the eigenvalues of an endomorphism. In the next chapter we will then define the network multipliers to be the multipliers for the particular representation $\C^{\Sigma}$ corresponding to the fundamental network. Throughout this section,  $(X,(A_{\sigma})_{\sigma \in \Sigma})$ denotes a complex, finite dimensional representation of a monoid $\Sigma$. We also assume that $S \subset X$ is a (complex) linear subspace with the property that $F(S)\subset S$ for all $F \in \End(X)$. Generalising the analogous property for synchrony spaces, we say that $S$ is a \textit{robust} subspace of the representation $X$. Note that $S$ need not be $\Sigma$-invariant. \\
We will not assume anything else about $X$ or $S$. In particular, the results proven in this section will hold in the special case of $X = V^{\Sigma}$, with $S$ given by $S_p$ or more generally $S_{\bowtie}$ for $\bowtie$ a balanced relation on $\Sigma$.\\
We start with two important lemmas. These are similar to results in \cite{cen} and \cite{proj}, but included here for completeness.

\begin{lem}\label{knip}
Let $W$ be a complementable sub-representation of $X$, and suppose we have the decomposition
\begin{equation}
W = W_1 \oplus W_2 \oplus \dots \oplus W_k\, ,
\end{equation}
for some sub-representations $W_1, \dots, W_k$. Then we also have
\begin{equation}\label{inddecom341}
W \cap S = (W_1 \cap S) \oplus (W_2 \cap S) \oplus \dots \oplus (W_k \cap S)\, .
\end{equation}
\end{lem}

\begin{proof}
Let $U \subset X$ be some invariant complement to $W$ and denote by $P_i$ the projection onto $W_i$ for $i \in \{1, \dots, k\}$, with respect to the decomposition
\begin{equation}
X = U \oplus W_1 \oplus W_2 \oplus \dots \oplus W_k\, .
\end{equation}
As every $P_i$ is an endomorphism, it follows that these maps send the space $S$ to itself. In particular, given an element $w \in W \cap S$ we may write
\begin{equation}
w = \sum_{i=1}^k P_i(w) \,
\end{equation}
with $P_i(w) \in W_i \cap S$ for every $i \in \{1, \dots, k\}$. Hence, every element of $W \cap S$ can be written as a sum of elements in the different $W_i \cap S$. Such an expression is furthermore unique, as any element of $W$ may be written uniquely as a sum of elements in the $W_i$. This shows that we indeed have the decomposition \eqref{inddecom341}.
\end{proof}
\noindent As a special case of Lemma \ref{knip}, suppose we are given a decomposition of $X$ into so-called \textit{isotypic} components. That is, we have
\begin{equation}\label{decomp290234}
X \cong W_1^{n_1} \oplus W_2^{n_2} \oplus \dots \oplus W_k^{n_k} \, ,
\end{equation}
with the isotypic components $W_i^{n_i}$ given by
\begin{equation}\label{components238736732}
W_i^{n_i} := W_i \oplus \dots \oplus W_i \, (n_i \text{ times})\, ,
\end{equation}
and with the $W_i$ mutually non-isomorphic indecomposable representations. Identifying $S$ with its image under the isomorphism of equation \eqref{decomp290234}, it follows that we may write 
\begin{equation}\label{decomp290222}
S = (W_1^{n_1} \cap S) \oplus (W_2^{n_2}  \cap S) \oplus \dots \oplus (W_k^{n_k}  \cap S) \, ,
\end{equation}
as well as
\begin{equation}\label{components2387367sa32}
W_i^{n_i} \cap S := (W_i \cap S) \oplus \dots \oplus (W_i \cap S) \, .
\end{equation}
Of course, equation \eqref{components2387367sa32} seems somewhat misleading, as it implies some kind of similarity between the intersections of $S$ with the different copies of $W_i$. The following result shows that such a similarity indeed exists.

\begin{lem}\label{plak}
Suppose $W, W' \subset X$ are two complementable, isomorphic sub-representations (not necessarily appearing in a same decomposition of $X$). We have
\begin{equation}
\dim(W \cap S) = \dim(W' \cap S) \, .
\end{equation}
\end{lem}

\begin{proof}
As $W$ and $W'$ are isomorphic, there exists an isomorphism $\phi$ from $W$ to $W'$. We may extend $\phi$ to an endomorphism $\tilde{\phi}$ from $X$ to itself by letting it vanish on some invariant complement of $W$. It follows that $\tilde{\phi}$ sends $W \cap S$ into $W' \cap S$. Note that $\tilde{\phi}$ is injective on $W$, as we have $\tilde{\phi}|_{W} = \phi$. It follows that
\begin{equation}
\dim(W \cap S) \leq \dim(W' \cap S) \, .
\end{equation}
By reversing the roles of $W$ and $W'$, we see that indeed
\begin{equation}
\dim(W \cap S) = \dim(W' \cap S) \, .
\end{equation}
This proves the Lemma.
\end{proof}

\noindent Contrary to compact group representations, there may be non-zero endomorphisms between non-isomorphic indecomposable monoid representations. In particular, suppose we have a map $F \in \End(X)$ and a given decomposition of $X$ into indecomposable representations. We may then write $F$ in matrix form, with the entries given by endomorphisms between the indecomposable components. It follows that such a matrix need not necessarily have any vanishing entries. However, it is shown in \cite{Gensym} that some entries may be set to zero, without changing the spectrum of $F$. To this end, we have the following definition.
\begin{defi}
Given a decomposition of $X$ into indecomposable representations, we denote by $\mathcal{J} \subset \End(X)$ the set of all endomorphisms such that the entries between isomorphic indecomposable components are all nilpotent. In other words, let us write a decomposition of $X$ as
\begin{equation}\label{decos6}
X \cong W_1\oplus W_2 \oplus \dots \oplus W_s \, ,
\end{equation}
where the $W_i$ are (not necessarily distinct) indecomposable representations. We furthermore denote by $P_i: X \rightarrow W_i \subset X$ the projection onto the $i$th component of expression \eqref{decos6}. An endomorphism $F \in \End(X)$ then belongs to $\mathcal{J}$ if and only if $P_i \circ F \circ P_j$ is nilpotent (seen as an endomorphism of $W_i$) for all $i,j \in \{1, \dots,  s \}$ such that $W_i = W_j$. \\
We likewise denote by $\mathcal{D}$ the space of all maps $F \in \End(X)$ such that $P_i \circ F \circ P_j$ is a scalar multiple of the identity on $W_i$ if $W_i = W_j$, and such that $P_i \circ F \circ P_j = 0$ whenever $W_i \not= W_j$. \\
Recall from Lemma \ref{endoindecomp} that when $W_i$ is indecomposable, there is a unique expression of the elements in $\End(W_i)$ as the sum of a scalar multiple of the identity and a nilpotent endomorphism. From this we see that 
\begin{equation}
\End(X) = \mathcal{J} \oplus \mathcal{D} \, ,
\end{equation}
as complex vector spaces. We will use $P_{\mathcal{J}}$ and $P_{\mathcal{D}}$ to denote the projections from $\End(X)$ onto $\mathcal{J}$ and $\mathcal{D}$ respectively, and write $F_{\mathcal{J}} := P_{\mathcal{J}}(F)$ and $F_{\mathcal{D}} := P_{\mathcal{D}}(F)$ for $F \in \End(X)$.
\end{defi}

\begin{remk}\label{almostthre}
Instead of equation \eqref{decos6}, let us denote a decomposition of $X$ as
\begin{equation}\label{decomp290q234}
X \cong W_1^{n_1} \oplus W_2^{n_2} \oplus \dots \oplus W_k^{n_k} \, 
\end{equation}
with 
\begin{equation}\label{components238736q732}
W_i^{n_i} := W_i \oplus \dots \oplus W_i \, (n_i \text{ times})\, ,
\end{equation}
and with the $W_i$ mutually non-isomorphic indecomposable representations. In other words, equations \eqref{decomp290q234} and \eqref{components238736q732} together give the decomposition of \eqref{decos6}, but with copies of the same indecomposable representation grouped together into isotypic components.  In particular, we have $n_1 + \dots + n_k = s$. A map $F \in \End(X)$ may then be written as
\begin{equation}
F  = \begin{pmatrix*}[c]\label{fs1234}
 \ddots & & & & {\color{blue}F_{\bullet,\bullet} }\\
 & F^{l}_{1,1}& \dots & F^{l}_{1,n_l} &  \\
& \vdots & & \vdots &   \\
& F^{l}_{n_l,1} & \dots & F^{l}_{n_l,n_l}&  \\
 {\color{blue}F_{\bullet,\bullet} }& & & & \ddots\\
\end{pmatrix*}  \, .
\end{equation}
Here, $F^{l}_{i,j}$ denotes a homomorphism from the $j$th copy of $W_l$ in $W_l^{n_l}$ to the $i$th copy of $W_l$ in $W_l^{n_l}$, for $l \in \{1, \dots, k\}$ and $i,j \in \{1, \dots, n_l\}$. We also write $P^l_i$ for the projection from $X$ onto the $i$th copy of $W_l$ in $W_l^{n_l}$, so that we have $F^{l}_{i,j} = P^l_i\circ F \circ P^l_j$. The maps ${\color{blue}F_{\bullet,\bullet} }$ denote homomorphisms between non-isomorphic indecomposable representations, which will not be of further interest to us here.  \\
Let us now write $F^{l}_{i,j} = \la_{i,j}^l\Id_{W_l} +  N_{i,j}^l$ for some complex number $\la_{i,j}^l$ and a nilpotent map $N_{i,j}^l \in \End(W_l)$. We then have
\begin{equation}
F_{\mathcal{J}}  = \begin{pmatrix*}[c]
 \ddots & & & & {\color{blue}F_{\bullet,\bullet} }\\
 & N^{l}_{1,1}& \dots & N^{l}_{1,n_l} &  \\
& \vdots & & \vdots &   \\
& N^{l}_{n_l,1} & \dots & N^{l}_{n_l,n_l}&  \\
 {\color{blue}F_{\bullet,\bullet} }& & & & \ddots\\
\end{pmatrix*}  \, ,
\end{equation} 
and
\begin{equation}\label{fs123}
F_{\mathcal{D}}  = \begin{pmatrix*}[c]
 \ddots & & & & {\color{blue} 0 }\\
 & \la^{l}_{1,1}\Id_{W_l}& \dots & \la^{l}_{1,n_l}\Id_{W_l} &  \\
& \vdots & & \vdots &   \\
& \la^{l}_{n_l,1}\Id_{W_l} & \dots & \la^{l}_{n_l,n_l}\Id_{W_l}&  \\
 {\color{blue} 0 }& & & & \ddots\\
\end{pmatrix*}  \, .
\end{equation} 
Note that the numbers $\la_{i,j}^l$ are related to $F$ by
\begin{equation}
\la_{i,j}^l = \tr(P^l_i\circ F \circ P^l_j)/\dim(W_l)\, .
\end{equation}
These numbers will be used to define the network multipliers of a class of constructible networks. $\hfill \triangle$
\end{remk}

\begin{defi}[Multipliers of a Representation]
Given a decomposition of $X$ into indecomposable representations as in Remark \ref{almostthre}, we define linear maps $\Lambda_{i,j}^l : \End(X) \rightarrow \C$ for $l \in \{1, \dots, k\}$ and $i,j \in \{1, \dots, n_l\}$ by
\begin{equation}
\Lambda_{i,j}^l(F) := \tr(P^l_i\circ F \circ P^l_j)/\dim(W_l)\, ,
\end{equation}
for all $F \in \End(X)$.  In other words, $\Lambda_{i,j}^l(F)$ is equal to the number $\la_{i,j}^l$ as in equation \eqref{fs123}. Next, we define matrix-valued functions $\Lambda^l : \End(X) \rightarrow \C^{n_l \times n_l}$ for $l \in \{1, \dots, k\}$ by simply setting $(\Lambda^l(F))_{i,j} := \Lambda^l_{i,j}(F)$. We call these maps $\Lambda_l$ the \textit{multipliers} of the representation $X$ (corresponding to the given decomposition of $X$). Later when we return to the special case of $X = \C^{\Sigma}$, the multipliers will be called the \textit{network multipliers} of $\Sigma$.
\end{defi}
\noindent From the definition, we may already conclude a first basic fact about multipliers.

\begin{lem}\label{indjoepie}
The full set of maps $\Lambda^l_{i.j}: \End(X) \rightarrow \C$, for $l \in \{1, \dots, k\}$ and $i,j \in \{1, \dots, n_l\}$ is linearly independent. In particular, given any $n_1^2 + \dots + n_k^2$ complex numbers $(\la^l_{i,j})_{1 \leq i,j \leq n_l}^{1 \leq l \leq k}$, there exists an endomorphism $F \in \End(X)$ such that $\Lambda^l_{i.j}(F) = \la^l_{i,j}$ for all $l \in \{1, \dots, k\}$ and $i,j \in \{1, \dots, n_l\}$.
\end{lem}

\begin{proof}
The statement that the maps $\Lambda^l_{i,j}$ are linearly independent is equivalent to the statement that any $n_1^2+ \dots + n_k^2$ complex numbers $(\la^l_{i,j})_{1 \leq i,j \leq n_l}^{1 \leq l \leq k}$ may be obtained as the images of the maps $\Lambda^l_{i,j}$ simultaneously. To prove the latter statement, simply construct an endomorphisms $F$ such that in the notation of equation \eqref{fs1234} the maps $F_{i,j}^l$ have the required trace. For example, set $F^l_{i,j} = \lambda^l_{i,j}\Id_{W_l}$ for all $l \in \{1, \dots, k\}$ and $i,j \in \{1, \dots, n_l \}$, with all other entries of $F$ vanishing.  By definition, it follows that $\Lambda^l_{i,j}(F) = \la^l_{i,j}$ for all $l \in \{1, \dots, k\}$ and $i,j \in \{1, \dots, n_l\}$. This proves the lemma.
\end{proof}

\begin{remk}\label{almostthre2}
As any element of $\mathcal{J}$ has only nilpotent entries on the diagonal, we conclude that necessarily $\tr(F_{\mathcal{J}}) = 0$. From this we see that
\begin{equation}
\tr(F) = \tr(F_{\mathcal{J}}+F_{\mathcal{D}}) = \tr(F_{\mathcal{D}}) = \sum_{l=1}^k \dim(W_l)\tr(\Lambda^l(F))\, .
\end{equation}
More generally, it follows from Lemma \ref{knip} that a decomposition of $S$ (as vector spaces) is given by
\begin{equation}\label{Sdecomp290q234}
S \cong  \bigoplus_{l=1}^k (W_l^{n_l} \cap S)\, 
\end{equation}
with 
\begin{equation}\label{Scomponents238736q732}
W_i^{n_i} \cap S = (W_i \cap S) \oplus \dots \oplus (W_i \cap S) \, (n_i \text{ times})\, .
\end{equation}
Here we have identified $S$ with its image under the isomorphism of the decomposition in \eqref{decomp290q234}. It follows from the proof of Lemma \ref{knip} that projections for the decomposition given by \eqref{Sdecomp290q234} and \eqref{Scomponents238736q732} are given by $ P^l_j|_S: S \rightarrow W_l \cap S \subset W_l^{n_l} \cap S$ ($j$th copy) for $l \in \{1, \dots, k\}$ and $j \in \{1, \dots, n_l\}$. Moreover, as every endomorphism respects $S$, we may write
\begin{equation}
P^l_i|_S\circ F|_S \circ P^l_j|_S = (P^l_i\circ F \circ P^l_j)|_S = F^{l}_{i,j}|_S
\end{equation}
for all $l \in \{1, \dots, k\}$ and $i,j \in \{1, \dots, n_l\}$. In conclusion, we have
\begin{equation}
F|_S  = \begin{pmatrix*}[c]
 \ddots & & & & {\color{blue}F_{\bullet,\bullet}|_S }\\
 & F^{l}_{1,1}|_S& \dots & F^{l}_{1,n_l}|_S &  \\
& \vdots & & \vdots &   \\
& F^{l}_{n_l,1}|_S & \dots & F^{l}_{n_l,n_l}|_S&  \\
 {\color{blue}F_{\bullet,\bullet}|_S }& & & & \ddots\\
\end{pmatrix*}  \, .
\end{equation}
More importantly, it follows that
\begin{equation}
F_{\mathcal{J}} |_S  = \begin{pmatrix*}[c]
 \ddots & & & & {\color{blue}F_{\bullet,\bullet}|_S }\\
 & N^{l}_{1,1}|_S& \dots & N^{l}_{1,n_l}|_S &  \\
& \vdots & & \vdots &   \\
& N^{l}_{n_l,1}|_S & \dots & N^{l}_{n_l,n_l}|_S&  \\
 {\color{blue}F_{\bullet,\bullet}|_S }& & & & \ddots\\
\end{pmatrix*}  
\end{equation} 
and
\begin{equation}
F_{\mathcal{D}} |_S  = \begin{pmatrix*}[c]
 \ddots & & & & {\color{blue} 0 }\\
 & \la^{l}_{1,1}\Id_{W_l \cap S}& \dots & \la^{l}_{1,n_l}\Id_{W_l \cap S} &  \\
& \vdots & & \vdots &   \\
& \la^{l}_{n_l,1}\Id_{W_l \cap S} & \dots & \la^{l}_{n_l,n_l}\Id_{W_l \cap S}&  \\
 {\color{blue} 0 }& & & & \ddots\\
\end{pmatrix*}  \, .
\end{equation} 
We conclude that again $\tr(F_{\mathcal{J}} |_S) = 0$. Hence, we see that 
 \begin{equation}
\tr(F|_S) = \tr(F_{\mathcal{J}}|_S+F_{\mathcal{D}}|_S) = \tr(F_{\mathcal{D}}|_S) = \sum_{l=1}^k \dim(W_l \cap S)\tr(\Lambda^l(F))\, .
\end{equation} 
We collect these results in the proposition below. $\hfill \triangle$
\end{remk}

\begin{prop}\label{steponee33}
For any robust space $S \subset X$ there exist numbers $m^S_1, \dots m^S_k \in \N_{\geq 0}$ such that for all $F \in \End(W)$ we have
 \begin{equation}
\tr(F|_S) = \sum_{l=1}^k m^S_l \tr(\Lambda^l(F))\, .
\end{equation} 
These numbers moreover satisfy 
 \begin{equation}
\sum_{l=1}^k m^S_ln_l = \dim(S)\, ,
\end{equation} 
and in the case of $S=X$ we have $m^X_l > 0$ for all $l \in \{1, \dots, k\}$.
\end{prop}

\begin{proof}
This follows directly from Remark \ref{almostthre2} by setting $m^S_l := \dim(W_l \cap S)$.
\end{proof}
\noindent Next, we prove multiplicity of the multipliers.
\begin{prop}\label{multiplicative}
Given $l \in \{1, \dots, k\}$ we have $\Lambda^l(F\circ G) = \Lambda^l(F)\Lambda^l(G)$ for all $F, G \in \End(X)$. 
\end{prop}
\noindent To prove Proposition \ref{multiplicative} we first need the following important lemma. 
\begin{lem}\label{jacobideal}
The space $\mathcal{J}$ is a two-sided ideal of $\End(X)$. That is, for any $N \in \mathcal{J}$ and $F \in  \End(X)$ we have $F\circ N, N \circ F \in \mathcal{J}$.\end{lem}
\noindent The proof of Lemma \ref{jacobideal} is essentially the same as that of the analogous statement for representations over the real numbers, provided in \cite{Gensym}. Nevertheless, we give it here for completeness.
\begin{proof}
To show that $\mathcal{J}$ is a two-sided ideal, let us again denote the decomposition of $X$ into indecomposable representations as  
\begin{equation}\label{decosf}
X \cong W_1\oplus W_2 \oplus \dots \oplus W_s \, ,
\end{equation}
where we may now have $W_i = W_j$ for some $i \not= j$. Given $N \in \mathcal{J}$ and $F \in  \End(X)$, we need to show that $(F \circ N)_{i,j}$ is a nilpotent map from $W_j$ to $W_i$ for all $i,j$ such that $W_i = W_j$. To this end, let us simply write $FN:= F \circ N$ and compute:
\begin{equation}
(FN)_{i,j} = \sum_{p = 1}^s F_{i,p}N_{p,j}\, .
\end{equation}
Now, if $W_p = W_i = W_j$ then $N_{p,j}$ is nilpotent (as we have $N \in \mathcal{J}$). Because $\nil(W_i)$ is a two-sided ideal of $\End(W_i)$ (see Lemma \ref{endoindecomp}), we conclude that $F_{i,p}N_{p,j} \in \nil(W_i)$. If on the other hand we have $W_p \not= W_i$, then it follows from Lemma \ref{fitting} that likewise $F_{i,p}N_{p,j} \in \nil(W_i)$. Using again that $\nil(W_i)$ is an ideal of $\End(W_i)$, we conclude that $(FN)_{i,j} \in \nil(W_i)$ whenever $W_i = W_j$. This shows that indeed $FN \in \mathcal{J}$ whenever we have $N \in \mathcal{J}$. The proof for $NF$ goes exactly the same, which concludes the proof.
\end{proof}

\begin{proof}[Proof of Proposition \ref{multiplicative}]
As the different $\Lambda^l(F)$ simply represent the blocks of $F_{\mathcal{D}}$, it suffices to show that $F_{\mathcal{D}}G_{\mathcal{D}} = (FG)_{\mathcal{D}}$ for all $F, G \in \End(X)$. We first note that $F_{\mathcal{D}}G_{\mathcal{D}}$ is again an element of $\mathcal{D}$, as this product is again a block diagonal endomorphism with only multiples of the identity as its entries. Next, we have 
\begin{equation}
FG = (F_{\mathcal{D}} + F_{\mathcal{J}})(G_{\mathcal{D}} + G_{\mathcal{J}}) = F_{\mathcal{D}} G_{\mathcal{D}} + F_{\mathcal{D}} G_{\mathcal{J}} + F_{\mathcal{J}} G_{\mathcal{D}}  + F_{\mathcal{J} }G_{\mathcal{J} } \, .
\end{equation}
Finally, we note that $F_{\mathcal{D}} G_{\mathcal{J}} + F_{\mathcal{J}} G_{\mathcal{D}}  + F_{\mathcal{J} }G_{\mathcal{J} } \in \mathcal{J}$, as $\mathcal{J}$ is an ideal by Lemma \ref{jacobideal}. This means that $FG$ decomposes as
\begin{equation}
FG = [FG]_{\mathcal{D}} + [FG]_{\mathcal{J}} = [F_{\mathcal{D}} G_{\mathcal{D}}] + [F_{\mathcal{D}} G_{\mathcal{J}} + F_{\mathcal{J}} G_{\mathcal{D}}  + F_{\mathcal{J} }G_{\mathcal{J} }] \, .
\end{equation}
In particular, we see that indeed $F_{\mathcal{D}}G_{\mathcal{D}} = (FG)_{\mathcal{D}}$. This proves the proposition.
\end{proof}
\noindent As a warm-up to the proof of Theorem $\ref{main1}$, we now show how the results of propositions \ref{steponee33} and \ref{multiplicative} allow us to describe the eigenvalues of an endomorphism $F \in \End(X)$ and its restriction $F|_S$.
\begin{prop}\label{jacobsonforeigenvalues}
Let $m^S_1$ up to $m^S_k$ be as in Proposition \ref{steponee33} (so that we have $m^S_l:= \dim(W_l \cap S)$ for all $l \in \{1, \dots, k\}$). Counted with algebraic multiplicity, the eigenvalues of $F|_S$ are given by those of $\Lambda^1(F)$ ($m^S_1$ times), together with those of $\Lambda^2(F)$ ($m^S_2$ times), up to those of $\Lambda^k(F)$ ($m^S_k$ times).
\end{prop}
\noindent The main ingredient of the proof is the following important lemma.
\begin{lem}\label{sameeig}
Suppose we are given a finite dimensional (real or complex) vector space $V$ and two linear maps $A,B \in \lin(V,V)$. If we have $\tr(A^n) = \tr(B^n)$ for all $n \in \N$, then the eigenvalues of $A$ and $B$, counted with algebraic multiplicity, coincide.
\end{lem}
\noindent Lemma \ref{sameeig} is a well-known result that is proven in for example \cite{Gensym}.

\begin{proof}[Proof of Proposition \ref{jacobsonforeigenvalues}]
We define a matrix $C$ of size $\dim(S)$ as follows: $C$ has a block-diagonal consisting of $m^S_1$ times the matrix $\Lambda_1(F)$, $m^S_2$ times the matrix $\Lambda_2(F)$, up to $m^S_k$ times the matrix $\Lambda_k(F)$. All the other entries vanish. What we need to prove is that $F|_S$ and $C$ have the same eigenvalues, counted with algebraic multiplicity. In view of Lemma \ref{sameeig}, it suffices to show that $\tr([F|_S]^n) = \tr([C]^n)$ for all $n \in \N$. We first note that the matrix $C^n$ has a block-diagonal consisting of $m^S_1$ times the matrix $(\Lambda_1(F))^n$, up to $m^S_k$ times the matrix $(\Lambda_k(F))^n$. By Proposition \ref{multiplicative} we furthermore have $(\Lambda_l(F))^n = (\Lambda_1(F^n))$, so that the block-diagonal of $C^n$ is given by $m^S_1$ times the matrix $\Lambda_1(F^n)$, up to $m^S_k$ times the matrix $\Lambda_k(F^n)$. In particular, we conclude that 
 \begin{equation}
\tr(C^n) = \sum_{l=1}^k m^S_l \tr(\Lambda^l(F^n))\, ,
\end{equation} 
for all $n \in \N$. By Proposition \ref{steponee33} we also have 
 \begin{equation}
\tr([F|_S]^n) = \tr(F^n|_S)= \sum_{l=1}^k m^S_l \tr(\Lambda^l(F^n))\, .
\end{equation} 
Hence, we conclude that indeed $\tr([F|_S]^n) = \tr([C]^n)$ for all $n \in \N$, so that $C$ and $F|_S$ have the same eigenvalues. This completes the proof.
\end{proof}

%%%%%%%%%%%%%%%%%%%%%%%%%%% The role of V%%%%%%%%%%%%%%%%%%%

\section{Network Multipliers and Constructible Networks} \label{Appendix;  The role of V}
Using the algebraic machinery of the previous sections, we will now start gathering results towards proving the main theorems \ref{main1} and \ref{main2}. Our strategy will be to generalise propositions \ref{steponee33} and \ref{multiplicative} from the previous section to all constructible networks, and to any phase space $V$. Recall that for any complete fundamental network $\mathcal{F} = (\Sigma, \Sigma)$, we have a representation $X := (\C^{\Sigma}, (A_{\sigma})_{\sigma \in \Sigma})$ of $\Sigma$. It follows from Lemma  \ref{symmetryyeah} that $\End(X)$ is exactly equal to the space of linear admissible maps $\Gamma_f$ for the network $\mathcal{F}$. In particular, we may choose a decomposition of $X$ into indecomposable representations, which gives us the representation multipliers $\Lambda^l(\Gamma_f) := (\Lambda^l_{i,j}(\Gamma_f))$ for $l \in \{1, \dots, k\}$ and $i,j \in \{1, \dots, n_l\}$. We will use these to define the network multipliers.

\begin{remk}\label{shifttocoeffff}
As every endomorphism $\Gamma_f: \C^{\Sigma} \rightarrow \C^{\Sigma}$ is determined by the response function $f$, we may view the representation multipliers of $(\C^{\Sigma}, (A_{\sigma})_{\sigma \in \Sigma})$ as linear maps from the space of linear response functions to the space of complex matrices. In other words, we may consider the maps $f \mapsto \Lambda^l(\Gamma_f)$, as well as its coefficients $f \mapsto \Lambda_{i,j}^l(\Gamma_f)$, for all $l \in \{1, \dots, k\}$ and $i,j \in \{1, \dots, n_l\}$. Every response function $f: \C^{\Sigma} \rightarrow \C$ is in turn completely determined by its coefficients $c =(c_{\sigma})_{\sigma \in \Sigma} \in \C^{\Sigma}$, defined by 
\begin{equation}\label{relatedcoeffff1}
f(x) = \sum_{\sigma \in \Sigma} c_{\sigma}x_{\sigma}\,
\end{equation} 
for all $x \in \C^{\Sigma}$. Hence, we get linear maps from the coefficient-space $\C^{\Sigma}$ to the space of complex matrices. We will simply denote these maps by $\Lambda_{i,j}^l: \C^{\Sigma} \rightarrow \C$ and $\Lambda^l: \C^{\Sigma} \rightarrow \C^{n_l \times n_l}$ as well. In other words, we have $\Lambda_{i,j}^l(c) := \Lambda_{i,j}^l(\Gamma_f)$ and $\Lambda^l(c) := \Lambda^l(\Gamma_f)$, where $f$ and $c \in \C^{\Sigma}$ are related by equation \eqref{relatedcoeffff1}. $\hfill \triangle$
\end{remk}

\begin{defi}[Network Multipliers]\label{netwoorkmulti}
For any finite dimensional complex vector space $V$, we extend the maps $\Lambda_{i,j}^l: \C^{\Sigma} \rightarrow \C$ and $\Lambda^l: \C^{\Sigma} \rightarrow \C^{n_l \times n_l}$ of Remark \ref{shifttocoeffff} to maps from $\lin(V,V)^{\Sigma}$ to $\lin(V,V)$ and  $\lin(V,V)^{n_l \times n_l}$, respectively. Here, $\lin(V,V)^{n_l \times n_l}$ denotes the space of $n_l \times n_l$ block matrices with blocks in $\lin(V,V)$. Note that $\lin(V,V)^{\Sigma}$ denotes the coefficient-space of a linear response function $f: V^{\Sigma} \rightarrow V$. More precisely, if $\Lambda_{i,j}^l$ is given by
 \begin{equation}
\Lambda^l_{i,j}(c) = \sum_{\sigma \in \Sigma} a_{i,j}^{l,\sigma} \cdot c_{\sigma}\, ,
\end{equation}
for some $a_{i,j}^{l,\sigma} \in \C$ and with $c \in \C^{\Sigma}$, then we simply define
 \begin{equation}
\Lambda^l_{i,j}(C) = \sum_{\sigma \in \Sigma} a_{i,j}^{l,\sigma} \cdot C_{\sigma}\, ,
\end{equation}
for all $C \in \lin(V,V)^{\Sigma}$. We then set $\Lambda^l(C) := (\Lambda_{i,j}^l(C))$. The formal maps $\Lambda^l$ obtained in this way are called the \textit{network multipliers} of $\Sigma$. They are the maps featured in Theorem \ref{main1}. Note that we only use a decomposition of $\C^{\Sigma}$ (and that we do not look at  $V^{\Sigma}$ for general $V$) to define the maps $\Lambda^l: \lin(V,V)^{\Sigma} \rightarrow \lin(V,V)^{n_l \times n_l}$ for all $V$.
\end{defi}

\subsection{The Trace of an Admissible Map} 
We start by generalising Proposition \ref{steponee33} to admissible maps for all constructible networks. More precisely, the result we want to prove is
\begin{thr}\label{first-half}
Let $\mathcal{F} = (\Sigma, \Sigma)$ be a complete fundamental network, and denote by $\Lambda^l$ for $l \in \{1, \dots, k\}$ the corresponding $n_l \times n_l$ network multipliers found by decomposing $\C^{\Sigma}$. For any constructible network $\mathcal{N} = (N, \Sigma) \in \mathfrak{C}_{\mathcal{F}}$, there exist non-negative integers $m^{\mathcal{N}}_l$, $l \in \{1, \dots, k\}$, such that
\begin{equation}
\tr(\gamma^{\mathcal{N}}_f) = \sum_{l =1}^k m^{\mathcal{N}}_l \tr(\Lambda^l(C)) \, ,
\end{equation}
for any linear response function $f:V^{\Sigma} \rightarrow V$ with coefficients $C \in \lin(V,V)^{\Sigma}$ (and for any finite dimensional complex vector space $V$). The integers $(m^{\mathcal{N}}_l)_{l=1}^k$ furthermore satisfy 
\begin{equation}
\sum_{l =1}^k m^{\mathcal{N}}_l n_l = \# N\, .
\end{equation}
\end{thr}
\noindent To prove Theorem \ref{first-half} we will first need the following results.

\begin{lem}\label{inputmakeswhole}
Let $\mathcal{N} = (N, \mathcal{T})$ be a homogeneous network with asymmetric input, and let $U \subset N$ be a set of nodes such that 
\begin{equation}
\bigcup_{p \in U} N_p = N\, .
\end{equation}
In other words, the input networks for the nodes of $U$ cover the whole network. Let $f: V^{\mathcal{T}} \rightarrow V$ be a linear response function for $\mathcal{N}$, and denote by $\gamma_f$ and $\gamma_f^p$ the corresponding admissible maps for $\mathcal{N}$ and the input network $\mathcal{N}_p$ of a node $p \in N$, respectively. Then any eigenvalue of $\gamma_f$ is also an eigenvalue of at least one of the maps  $\gamma^p_f$ for $p \in U$.
\end{lem}

\begin{proof}
Recall from Section \ref{Homogeneous Networks and the Fundamental Network} that we have surjective linear maps $\psi_p: V^N \rightarrow V^{N_p}$ for all nodes $p \in N$ satisfying 
\begin{equation}\label{comm34}
\gamma^p_f \circ \psi_p = \psi_p \circ \gamma_f \, .
\end{equation}
Moreover, it holds that 
\begin{equation} \label{descriptionn1}
\ker(\psi_p) = \{ v \in V^N \, \mid \, v_q = 0 \,\, \forall\, q \in N_p \subset N \} \, .
\end{equation}
Now suppose $\la \in \C$ is an eigenvalue of $\gamma_f$ and choose a corresponding eigenvector $0 \not=v \in V^N$. It follows from equation \eqref{comm34} that 
\begin{equation}
\gamma^p_f(\psi_p(v)) = \psi_p(\gamma_f (v)) = \psi_p(\la v) = \la \psi_p(v)\, ,
\end{equation}
for all $p \in N$. Hence, either $\psi_p(v)$ vanishes, or it is an eigenvector for the eigenvalue $\la$ of $\gamma^p_f$. Suppose that $\psi_p(v)$ vanishes for all $p \in U$. From expression \eqref{descriptionn1} it follows that $v_q = 0$ for all $q \in N_p$, for all $p \in U$. However, as the input networks of the nodes in $U$ cover all of $N$, we see that necessarily $v = 0$. This contradicts the fact that $v$ is an eigenvector of $\gamma_f$, and we conclude that $\la$ is an eigenvalue of $\gamma^p_f$ for at least one $p \in U$. This completes the proof.
\end{proof}

\begin{remk}\label{same--eigenvalues}
It follows from Lemma \ref{inputmakeswhole} that any eigenvalue of an admissible map $\gamma_f^{\mathcal{N}}$ for a constructible network $\mathcal{N} \in \mathfrak{C}_{\mathcal{F}}$ is also an eigenvalue of $\Gamma_f$ (the corresponding admissible map for the fundamental network $\mathcal{F}$). More precisely, Lemma \ref{inputmakeswhole} tells us that an eigenvalue of $\gamma_f^{\mathcal{N}}$ is also an eigenvalue of $\gamma_f^{{\mathcal{N}}_p}$ (the corresponding admissible map for the input network $\mathcal{N}_p \subset \mathcal{N}$) for some node $p \in N$. This holds because any network is covered by the input networks of all of its nodes. By definition of a constructible network, these input networks are all quotient networks of $\mathcal{F}$. This means that there exists a robust synchrony space $S_p \subset V^{\Sigma}$ such that $\Gamma_f|_{S_p}$ is conjugate to $\gamma_f^{{\mathcal{N}}_p}$. We conclude that any eigenvalue of a map $\gamma_f^{{\mathcal{N}}_p}$, and hence any eigenvalue of $\gamma_f^{\mathcal{N}}$, is in the spectrum of $\Gamma_f$. \\
We can already say more in the case of $V = \C$. As the network multipliers are defined as the representation multipliers of a decomposition of $\C^{\Sigma}$ (and as $\End(\C^{\Sigma})$ consists of exactly all the maps $\Gamma_f$), we conclude from Proposition \ref{jacobsonforeigenvalues} that the eigenvalues of $\Gamma_f$ are given by those of the network multipliers. Hence, we conclude that for $V = \C$, every eigenvalue of a map $\gamma_f^{\mathcal{N}}$ is an eigenvalue of a network multiplier.  $\hfill \triangle$
\end{remk}

%%%%%%%%%%%%%%%%%%%%%%%%%%

\noindent Next, we need the following result.
\begin{prop}\label{intergerwholenetwork}
Given a network $\mathcal{N} = (N, \Sigma)$, there exist integers $(s_p)_{p \in N}$ such that for any linear response function $f: \C^{\Sigma} \rightarrow \C$ we have
\begin{equation}
\tr(\gamma_f) = \sum_{p \in N} s_p \tr(\gamma_f^p)\, .
\end{equation}
\end{prop}
\noindent In order to prove Proposition \ref{intergerwholenetwork}, we first need to know more about how the different input networks make up a whole network. This is captured by the following algebraic objects.

\begin{defi}
Given a network $ \mathcal{N} = (N, \Sigma)$, we define the $\Z$-module
\begin{equation}
\Z N:= \bigoplus_{q \in N} \Z q \, .
\end{equation}
Elements of this module are simply given by formal sums 
\begin{equation}
\sum_{q \in N}n_q \cdot q \, ,
\end{equation}
with $n_q \in \Z$. We also distinguish special elements in $\Z N$, given by 
\begin{equation}
T := \sum_{q \in N}1 \cdot q \quad \text{ and } \quad T_p := \sum_{q \in N_p}1 \cdot q \, ,
\end{equation}
for $p \in N$ and with $N_p$ the nodes of the input network of $p$.
\end{defi}

\begin{lem}\label{zmodultresult}
The element $T$ is contained in the $\Z$-span of the elements $T_p$. In other words, we have 
\begin{equation}
T = \sum_{p \in N}s_p T_p \, ,
\end{equation}
for some $s_p \in \Z$.
\end{lem}

\begin{proof}
We argue by induction on the size of the network $\mathcal{N}$. If there is only one node $p$ then the whole network is clearly equal to the input network of $p$. Hence,  we find $T = T_p$. Similarly, if there are only two nodes then the whole network is either equal to the disjoint union of (the input networks of) the two nodes, or equal to the input network of one of the nodes. Let us therefore assume the statement is true for all networks with less than $k$ nodes, for $k>2$. Given a network $\mathcal{N} = (N, \Sigma)$ with $k$ nodes, let $S \subset N$ be a minimal set of nodes whose input networks cover the whole network. In other words, the input networks of the nodes of any strict subset of $S$ do not cover $\mathcal{N}$. We moreover pick a single node $u \in S$. If $S = \{u\}$ then we clearly have $T = T_u$, so assume $S$ contains more elements. By assumption, the network $\mathcal{M}$ that consists of the input networks of all nodes in $S \setminus \{u\}$ is not the whole network. However, $\mathcal{M}$ is a subnetwork of $\mathcal{N}$, meaning that it has only outgoing arrows, and so contains the input networks of all of its nodes. It follows from the induction hypotheses that
\begin{equation}\label{compform}
T_{\mathcal{M}} := \sum_{p \in M} 1 \cdot p = \sum_{p \in M} t_{p} T_p\, ,
\end{equation}
where $M$ denotes the nodes of $\mathcal{M}$, and for some $t_p \in \Z$. Likewise, it can be seen that $\mathcal{P}:= \mathcal{M} \cap \mathcal{N}_u$ is a strict subnetwork of $\mathcal{N}$. Hence, we may write
\begin{equation}\label{compfort}
T_{\mathcal{P}} := \sum_{p \in P} 1 \cdot p = \sum_{p \in P} r_{p} T_p\, ,
\end{equation}
where $P$ denotes the nodes of $\mathcal{P}$, and where we have $r_p \in \Z$. If $\mathcal{P}$ is empty then we simply write $T_{\mathcal{P}} = 0$. Combining expressions \eqref{compform} and  \eqref{compfort} we finally arrive at
\begin{equation}
T = T_{\mathcal{M}} + T_u - T_{\mathcal{P}} =  \sum_{p \in M} t_{p} T_p + T_u - \sum_{p \in P} r_{p} T_p = \sum_{p \in N} s_{p} T_p \, ,
\end{equation}
for some appropriate $s_p \in \Z$. This proves the lemma.
\end{proof}

\begin{proof}[Proof of Proposition \ref{intergerwholenetwork}]
We again use the fact that for every node $p \in N$ there is a surjective map $\psi_p: \C^N \rightarrow \C^{N_p}$ conjugating $\gamma_f$ and $\gamma_f^p$. I.e. we have 
\begin{equation}\label{comm35}
\gamma^p_f \circ \psi_p = \psi_p \circ \gamma_f \, .
\end{equation}
The kernel of this map is given by
\begin{equation}
\ker(\psi_p) = \{ (x_q)_{q \in N} \in \C^N \, \mid \, x_q = 0 \, \forall \, q \in N_p\} =: K_p\, .
\end{equation}
We furthermore define the space 
\begin{equation}
L_p := \{ (x_q)_{q \in N} \in \C^N \, \mid \, x_q = 0  \, \forall \,  q \notin N_p\}\, ,
\end{equation}
so that we have 
\begin{equation}\label{decoom1r}
\C^{N} = K_p \oplus L_p\, ,
\end{equation}
for every $p \in N$. Moreover, we denote by $Q_p: \C^N \rightarrow L_p$ and $I_p: L_p \rightarrow \C^{N}$ the projection and inclusion corresponding to the decomposition \eqref{decoom1r}. As $\psi_p$ is surjective, it follows that the map $ \psi_p\circ I_p: L_p \rightarrow \C^{N_p}$ is a bijection. Moreover, since $K_p$ equals the kernel of $\psi_p$, we have $\psi_p = \psi_p \circ I_p \circ Q_p$. We may therefore rewrite equation \eqref{comm35} to 
\begin{equation}
\gamma^p_f \circ \psi_p = \psi_p \circ I_p \circ Q_p \circ \gamma_f \, .
\end{equation}
From this it follows that 
\begin{equation}
\gamma^p_f  \circ [\psi_p \circ I_p]= [\psi_p \circ I_p]\circ ( Q_p \circ \gamma_f   \circ I_p) \, .
\end{equation}
In other words, the bijection  $\psi_p\circ I_p$ conjugates $\gamma^p_f $ and $Q_p \circ \gamma_f   \circ I_p$. The latter of these maps is exactly the diagonal block of $\gamma_f$ corresponding to the $L_p$ component of the decomposition \eqref{decoom1r}. In particular, we conclude that this block of $\gamma_f$ has the same trace as $\gamma_f^p$. \\
Next, for any node $q \in N$ we define the element $\delta_q \in \C^N$ by $(\delta_q)_r = \delta_{q,r}$ for all $r \in N$. It follows that $(\delta_q)_{q \in N}$ forms a basis of $\C^{N}$, and that $(\delta_q)_{q \in N_p}$ forms a basis of $L_p$. We may therefore write
\begin{equation}\label{trace44p11}
 \tr(\gamma_f^p) = \sum_{q \in N_p}(\gamma_f(\delta_q))_q\, .
\end{equation}
As the trace of $\gamma_f$ is given by the full sum
\begin{equation}\label{trace4411}
 \tr(\gamma_f) = \sum_{q \in N}(\gamma_f(\delta_q))_q\, ,
\end{equation}
it remains to show that the set of nodes $N$ may be expressed as a combination of the sets of nodes of the various input networks $N_p$. However, this is exactly the result of Lemma \ref{zmodultresult}. More precisely, if we have
 \begin{equation}
T = \sum_{p \in N}s_p T_p \, ,
\end{equation}
for some $s_p \in \Z$, then equations \eqref{trace44p11} and \eqref{trace4411}  tell us that we also have 
\begin{equation}
\tr(\gamma_f) = \sum_{p \in N} s_p \tr(\gamma_f^p)\, .
\end{equation}
This proves the proposition.
\end{proof}

\noindent As a next step towards proving Theorem \ref{first-half}, we now show that the result holds in the special case of $V = \C$.
\begin{prop}\label{pp001}
For any constructible network $\mathcal{N} = (N, \Sigma) \in \mathfrak{C}_{\mathcal{F}}$, there exist non-negative integers $m^{\mathcal{N}}_l$, $l \in \{1, \dots, k\}$, such that
\begin{equation}
\tr(\gamma^{\mathcal{N}}_f) = \sum_{l =1}^k m^{\mathcal{N}}_l \tr(\Lambda^l(c)) \, ,
\end{equation}
for any linear response function $f:\C^{\Sigma} \rightarrow \C$ with coefficients $c \in \C^{\Sigma}$. The integers $(m^{\mathcal{N}}_l)_{l=1}^k$ furthermore satisfy 
\begin{equation}
\sum_{l =1}^k m^{\mathcal{N}}_l n_l = \# N\, .
\end{equation}
\end{prop}

\begin{proof}
By Proposition \ref{intergerwholenetwork} there exist integers $(s_p)_{p \in N}$ such that 
\begin{equation}\label{combine11}
\tr(\gamma^{\mathcal{N}}_f) = \sum_{p \in N} s_p \tr(\gamma_f^{\mathcal{N}_p})\, ,
\end{equation}
for any linear response function $f:\C^{\Sigma} \rightarrow \C$. Here $\gamma_f^{\mathcal{N}_p}$ denotes an admissible map for the input network $\mathcal{N}_p$ of a node $p \in N$. By definition of a constructible network, every input network $\mathcal{N}_p$ may be realised as a quotient of $\mathcal{F}$. It follows from Lemma \ref{steponee33} that there exist integers $(m^p_l)_{l=1}^k$ such that
\begin{equation}\label{combine12}
\tr(\gamma_f^{\mathcal{N}_p}) = \sum_{l =1}^k m^p_l \tr(\Lambda^l(c)) \, .
\end{equation}
Combining equations \eqref{combine11} and \eqref{combine12}, we get
\begin{align}
\tr(\gamma^{\mathcal{N}}_f) &= \sum_{p \in N} s_p \tr(\gamma_f^{\mathcal{N}_p}) =  \sum_{p \in N} s_p  \sum_{l =1}^k m^p_l \tr(\Lambda^l(c)) \\ \nonumber
&= \sum_{l =1}^k  \left( \sum_{p \in N}  s_p  m^p_l \right)\tr(\Lambda^l(c)) =  \sum_{l =1}^k  m^{\mathcal{N}}_l \tr(\Lambda^l(c))\, ,
\end{align}
where we have set
\begin{equation}
m^{\mathcal{N}}_l :=   \sum_{p \in N}  s_p  m^p_l \, .
\end{equation}
It remains to show that all $m^{\mathcal{N}}_l$ are non-negative, and that $\sum_{l =1}^k m^{\mathcal{N}}_l  n_l = \#N$. For the first statement, we fix an index $s \in \{1, \dots, k\}$. It follows from Lemma \ref{indjoepie} that there exists a response function $f_s$ such that $\Lambda^s(\Gamma_{f_s}) = \Id_{n_s}$ and  $\Lambda^l(\Gamma_{f_s}) = 0$ whenever $l \not=s$. Therefore, we see from Remark \ref{same--eigenvalues} that $\gamma^{\mathcal{N}}_{f_s}$ only has eigenvalues $0$ and $1$, and we conclude that $\tr(\gamma^{\mathcal{N}}_{f_s}) \geq 0$. From this we see that
\begin{equation}
0 \leq \tr(\gamma^{\mathcal{N}}_{f_s}) =  \sum_{l =1}^k  m^{\mathcal{N}}_l \tr(\Lambda^l(\Gamma_{f_s})) = m^{\mathcal{N}}_s n_s\, ,
\end{equation}
and it follows that $m^{\mathcal{N}}_s \geq 0$. To show that $\sum_{l =1}^k m^{\mathcal{N}}_l n_l = \#N$, we choose a linear response function $\alpha: \C^{\Sigma} \rightarrow \C$ such that $\Lambda^l(\Gamma_{\alpha}) = \Id_{n_l}$ for all $l \in \{1, \dots, k\}$ (for example by choosing  $\alpha(x) = x_e$, with $e \in \Sigma$ the unit). It follows from Remark \ref{same--eigenvalues} that $\gamma^{\mathcal{N}}_{\alpha}$ only has the eigenvalue $1$. Therefore, we see that
\begin{equation}
\#N = \tr(\gamma^{\mathcal{N}}_{\alpha}) =  \sum_{l =1}^k  m^{\mathcal{N}}_l \tr(\Lambda^l(\Gamma_{\alpha})) = \sum_{l=1}^k m^{\mathcal{N}}_l n_l\, .
\end{equation}
This proves the proposition.
\end{proof}

\begin{proof}[Proof of Theorem \ref{first-half}]
We choose the numbers $m_l^{\mathcal{N}}$ as in Proposition \ref{pp001}, so that it holds that 
\begin{equation}\label{expressionforjj123}
\tr(\gamma^{\mathcal{N}}_f) = \sum_{l =1}^k m^{\mathcal{N}}_l \tr(\Lambda^l(c)) \, ,
\end{equation}
for all linear response functions $f:\C^{\Sigma} \rightarrow \C$ with coefficients $c \in \lin(\C,\C)^{\Sigma}$. Note that both sides of equation \eqref{expressionforjj123} are linear expressions in the coefficients $c = (c_{\sigma})_{\sigma \in \Sigma} \in \C^{\Sigma}$. For a general phase space $V$, we may get the same expression (but now in the coefficients $C \in \lin(V,V)^{\Sigma}$) by replacing the trace by the so-called \ \textit{block-trace}, $\btr(\bullet)$. This is simply the sum of the diagonal $(V \times V)$-blocks of a linear map from a space $V^n$ to itself. In other words, we may conclude that
\begin{equation}\label{toughpointtomake2}
\btr(\gamma^{\mathcal{N}}_g) =   \sum_{l =1}^k m^{\mathcal{N}}_l  \btr(\Lambda^l(C)) \, .
\end{equation}
for all linear response functions $g:V^{\Sigma} \rightarrow V$ with coefficients $C \in \lin(V,V)^{\Sigma}$. Next, we note that for any linear map $A: V^n \rightarrow V^n$ we have the identity $\tr(\btr(A)) = \tr(A)$. That is, it makes no difference if one takes the block-trace and then the trace, or the trace straight away. Taking the (usual) trace of both sides of equation \eqref{toughpointtomake2} therefore gives
\begin{equation}\label{toughpointtomake23}
\tr(\gamma^{\mathcal{N}}_g) = \tr(\btr(\gamma^{\mathcal{N}}_g)) =  \sum_{l =1}^k m^{\mathcal{N}}_l  \tr(\btr(\Lambda^l(C))) = \sum_{l =1}^k m^{\mathcal{N}}_l  \tr(\Lambda^l(C)) \, ,
\end{equation}
which is what we want to show. Finally,  it follows from Proposition \ref{pp001} that we have $m_l^{\mathcal{N}} \geq 0$ for all $l \in \{1, \dots, k\}$, as well as
\begin{equation}
\sum_{l =1}^k m^{\mathcal{N}}_l n_l = \# N\, .
\end{equation}
This finishes the proof.
\end{proof}

\subsection{Multiplicity of the Network Multipliers}

Our next step is to generalise Proposition \ref{multiplicative}. To this end, we first want to understand the relation between the coefficients of linear response functions $f, g, h: V^{\Sigma} \rightarrow V$ if we have $\Gamma_f \circ \Gamma_g = \Gamma_h$.

\begin{lem}\label{nodiff}
Define a product $\circ_{\Sigma}$ on $\lin(V,V)^{\Sigma}$ by 
\begin{equation}\label{productte23q21}
(C \circ_{\Sigma} D)_{\sigma} = \sum_{\substack{\tau, \kappa \in \Sigma \\ \kappa \circ \tau = \sigma}} C_{\tau}D_{\kappa} \, .
\end{equation}
If the coefficients of the response functions $f$, $g$ and $h$ (all on $V^{\Sigma}$) are given by $C$, $D$ and $C \circ_{\Sigma} D$ respectively, then we have
\begin{equation}\label{productte23q2}
\Gamma_f \circ \Gamma_g = \Gamma_{h}\, .
\end{equation}
\end{lem}
\begin{proof}
It is shown in \cite{CCN} that equation \eqref{productte23q2} holds if we have
\begin{equation}\label{prodto34}
h = f((g \circ A_{\sigma_1}), \dots, (g \circ A_{\sigma_n}))\, ,
\end{equation}
where the we have set $\Sigma = \{\sigma_1, \dots, \sigma_n\}$ for notational convenience. To show that equation \eqref{prodto34} is indeed satisfied, we pick a vector $v \in V$ and an element $\sigma \in \Sigma$ to construct $\delta_{\sigma}v \in V^{\Sigma}$ by $(\delta_{\sigma}v)_{\tau} = \delta_{\sigma, \tau}v$ for all $\tau \in \Sigma$. By linearity, it suffices to show that the two sides of equation \eqref{prodto34} agree on all elements of this form. By definition, we have 

\begin{equation}
h(\delta_{\sigma}v) = \sum_{\substack{\tau, \kappa \in \Sigma \\ \kappa \circ \tau = \sigma}} C_{\tau}D_{\kappa} v\, .
\end{equation}
On the other hand, we find
\begin{align}
f(\dots, (g \circ A_{\tau}), \dots) (\delta_{\sigma}v) &= \sum_{\tau \in \Sigma}C_{\tau}g(A_{\tau}\delta_{\sigma}v) = \sum_{\tau \in \Sigma}C_{\tau}\sum_{\kappa \in \Sigma}D_{\kappa}(A_{\tau}\delta_{\sigma}v)_{\kappa} \\ \nonumber
&= \sum_{\tau \in \Sigma}C_{\tau}\sum_{\kappa \in \Sigma}D_{\kappa}(\delta_{\sigma}v)_{\kappa \circ \tau} = \sum_{\tau \in \Sigma}C_{\tau}\sum_{\kappa \in \Sigma}D_{\kappa}(\delta_{\sigma, \kappa \circ \tau}v) \\ \nonumber
&= \sum_{\substack{\tau, \kappa \in \Sigma \\ \kappa \circ \tau = \sigma}} C_{\tau}D_{\kappa} v \, .
\end{align}
This shows that  equation \eqref{prodto34} indeed holds, thus proving lemma.
\end{proof}
\noindent As we may retrieve $f$ from $\Gamma_f$ by using that $(\Gamma_f)_e = f$, it follows that both equations \eqref{productte23q21} and \eqref{productte23q2} completely describe the product $\circ_{\Sigma}$ on $\lin(V,V)^{\Sigma}$. As a result, multiplicity of representation multipliers (Proposition \ref{multiplicative}) may be restated in the case of network multipliers as
\begin{equation}\label{productte23q22}
\Lambda^l(c)\Lambda^l(d) = \Lambda^l(c \circ_{\Sigma} d)\, ,
\end{equation}
for all $c,d \in \C^{\Sigma}$ and $l \in \{1, \dots, k\}$. As the network multipliers are defined in terms of the multipliers of the representation $(\C^{\Sigma}, (A_{\sigma})_{\sigma \in \Sigma})$, we may not immediately conclude equation \eqref{productte23q22} holds for $c,d \in \lin(V,V)$, for general $V$. However, using an argument similar to the one used in the proof of Theorem \ref{first-half} we can show that equation \eqref{productte23q22} does in fact hold true for general $V$:

\begin{prop}\label{second-half}
For all $C,D \in \lin(V,V)$ and $l \in \{1, \dots, k\}$ we have 
\begin{equation}\label{productte2322}
\Lambda^l(C)\Lambda^l(D) = \Lambda^l(C \circ_{\Sigma} D)\, .
\end{equation}
\end{prop}

\begin{proof}
Equation \eqref{productte2322} is equivalent to
\begin{equation}\label{productte233}
\sum_{r=1}^{n_l} \Lambda^l_{i,r}(C)\Lambda^l_{r,j} (D)= \Lambda^l_{i,j}(C \circ_{\Sigma} D)\, ,
\end{equation}
for all $i,j \in \{1, \dots, n_l\}$. We will therefore fix $l \in \{1, \dots, k\}$ and a pair $(i,j)$ with $i,j \in \{1, \dots, n_l\}$. As $\Lambda_l$ was defined by generalising an expression for $V=\C$ to general $V$, we conclude that the left side of equation \eqref{productte233} may be written as
\begin{equation}\label{formallala}
\sum_{r=1}^{n_l} \Lambda^l_{i,r}(C)\Lambda^l_{r,j} (D) = \sum_{\sigma, \tau \in \Sigma}a_{\sigma, \tau} C_{\sigma}D_{\tau}\,
\end{equation}
for certain unique numbers $a_{\sigma, \tau} \in \C$, independent of $V$ and for all $C,D \in \lin(V,V)$. As the product $\circ_{\Sigma}$ likewise has the same formal expression for all $V$, we may similarly write
\begin{equation}\label{formallala3}
\Lambda^l_{i,j}(C \circ_{\Sigma} D) = \sum_{\sigma, \tau \in \Sigma}b_{\sigma, \tau} C_{\sigma}D_{\tau}\,
\end{equation}
for certain unique $b_{\sigma, \tau} \in \C$. It remains to show that $a_{\sigma, \tau} = b_{\sigma, \tau}$ for all $\sigma, \tau \in \Sigma$. To this end, we define $\Delta^{\kappa} \in \lin(V,V)$ by $(\Delta^{\kappa})_{\iota} = \delta_{\kappa, \iota}\Id_V$ for all $\kappa, \iota \in \Sigma$. It follows that
\begin{align}\label{formallala2}
\sum_{r=1}^{n_l} \Lambda^l_{i,r}(\Delta^{\sigma})\Lambda^l_{r,j} (\Delta^{\tau}) &= a_{\sigma, \tau}\Id_V\, ,\\ \nonumber
\Lambda^l_{i,j}(\Delta^{\sigma} \circ_{\Sigma} \Delta^{\tau}) &= b_{\sigma, \tau}\Id_V
\end{align}
for all $\sigma, \tau \in \Sigma$. We therefore conclude that equation \eqref{productte233} holds if it holds for all $C$ and $D$ whose entries are scalar multiples of the identity. However, this follows directly from equation \eqref{productte23q22}. More precisely, let us define $c\Id_V \in \lin(V,V)$ for $c \in \C^{\Sigma}$ by $(c\Id_V)_{\sigma} = c_{\sigma}\Id_V$ for all $\sigma \in \Sigma$. It follows that for all $c,d \in \C^{\Sigma}$ we have
 \begin{align}\label{productte234}
\sum_{r=1}^{n_l} \Lambda^l_{i,r}(c\Id_V)\Lambda^l_{r,j} (d\Id_V) &= \sum_{r=1}^{n_l} \Lambda^l_{i,r}(c)\Lambda^l_{r,j} (d)\Id_V \\ \nonumber
&= \Lambda^l_{i,j}(c \circ_{\Sigma} d)\Id_V = \Lambda^l_{i,j}([c\Id_V] \circ_{\Sigma} [d\Id_V])\, ,
\end{align}
where in the second step we have used equation \eqref{productte23q22}. This proves the proposition.
\end{proof}

\noindent Our final step will be to generalise the result of Lemma \ref{nodiff} to any constructible network.

\begin{prop}\label{third-half}
Let $\circ_{\Sigma}$ be the product on linear response functions defined by 
\begin{equation}\label{productte23q2116}
\Gamma_f \circ \Gamma_g = \Gamma_{f \circ_{\Sigma} g}\, ,
\end{equation}
for $f,g: V^{\Sigma} \rightarrow V$, or equivalently by Lemma \ref{nodiff} on the coefficients of $f$, $g$ and $f \circ_{\Sigma} g$. Then for any constructible network $\mathcal{N} \in \mathfrak{C}_{\Sigma}$ we have
\begin{equation}\label{productte23q2116}
\gamma^{\mathcal{N}}_f \circ \gamma^{\mathcal{N}}_g = \gamma^{\mathcal{N}}_{f \circ_{\Sigma} g}\, .
\end{equation}
\end{prop}
\noindent To prove Proposition \ref{second-half} we need the following easy lemma.

\begin{lem}\label{backtoinptt1}
Let $\mathcal{N}  = (N, \mathcal{T})$ be a homogeneous coupled cell network with asymmetric input. Denote by $\gamma_f^{\mathcal{N}}$ an admissible map for $\mathcal{N}$ and by $\gamma_f^{\mathcal{N}_p}$ an admissible map for the input network $\mathcal{N}_p$ of a node $p \in N$. For any linear response function $f: V^{\Sigma} \rightarrow V$, we have $\gamma_f^{\mathcal{N}} = 0$, if and if only if $\gamma_f^{\mathcal{N}_p} = 0$ for all $p \in N$.
\end{lem}

\begin{proof}
Recall that for any node $p \in N$ there is a surjective linear map $\psi_p: V^{\mathcal{N}} \rightarrow V^{\mathcal{N}_p}$ such that $\gamma_f^{\mathcal{N}_p} \circ \psi_p =  \psi_p \circ \gamma_f^{\mathcal{N}}$. It follows that $\gamma_f^{\mathcal{N}_p} \circ \psi_p = 0$ whenever $\gamma_f^{\mathcal{N}} = 0$. As  $\psi_p$ is surjective, we conclude that $\gamma_f^{\mathcal{N}_p} = 0$ for all $p \in N$ whenever $\gamma_f^{\mathcal{N}} = 0$. Conversely, suppose $\gamma_f^{\mathcal{N}_p} = 0$ for all $p \in N$, and let $v = (v_q)_{q \in N} \in V^{\mathcal{N}}$ be given. We furthermore pick a node $p \in N$ and a node $s$ in the input network of $p$. It follows that 
\begin{equation}
0 = [\gamma_f^{\mathcal{N}_p} \circ \psi_p(v)]_s =  [\psi_p \circ \gamma_f^{\mathcal{N}}(v)]_s = [\gamma_f^{\mathcal{N}}(v)]_s\, .
\end{equation}
As the nodes $p$ and $s$ may be chosen freely, we conclude that $[\gamma_f^{\mathcal{N}}(v)]_s = 0$ for all $s \in N$ and $v \in V^{\mathcal{N}}$. It follows that $\gamma_f^{\mathcal{N}} = 0$, which concludes the proof.
\end{proof}

\begin{proof}[Proof of Proposition \ref{third-half}]
Recall that any input network of $\mathcal{N} \in \mathfrak{C}_{\mathcal{F}}$ may be realised as a quotient network of $\mathcal{F}$. This means that for every node $p$ of $\mathcal{N}$ there exists an injective linear map $\theta_p: V^{\mathcal{N}_p} \rightarrow V^{\Sigma}$ such that $\Gamma_f \circ \theta_p = \theta_p \circ \gamma_f^{\mathcal{N}_p}$ for all response functions $f: V^{\Sigma} \rightarrow V$. It follows that 
\begin{align}
\theta_p \circ (\gamma_f^{\mathcal{N}_p} \circ \gamma_g^{\mathcal{N}_p} - \gamma_{f \circ_{\Sigma} g}^{\mathcal{N}_p}) &= \Gamma_f \circ \theta_p \circ \gamma_g^{\mathcal{N}_p} - \Gamma_{f \circ_{\Sigma} g} \circ  \theta_p \\ \nonumber
&=  \Gamma_f  \circ \Gamma_g \circ \theta_p- \Gamma_{f \circ_{\Sigma} g} \circ  \theta_p \\ \nonumber
&=   (\Gamma_f  \circ \Gamma_g  - \Gamma_{f \circ_{\Sigma} g}) \circ  \theta_p = 0 \circ \theta_p = 0\, .
\end{align}
As $\theta_p$ is injective, we conclude that $\gamma_f^{\mathcal{N}_p} \circ \gamma_g^{\mathcal{N}_p} = \gamma_{f \circ_{\Sigma} g}^{\mathcal{N}_p}$ for all response functions $f$ and $g$. \\
Next, for every node $p$ of $\mathcal{N}$ we have a surjective linear map $\psi_p: V^{\mathcal{N}} \rightarrow V^{\mathcal{N}_p}$ such that $\gamma_f^{\mathcal{N}_p} \circ \psi_p =  \psi_p \circ \gamma_f^{\mathcal{N}}$. It follows that
\begin{align}\label{trickblah1}
\psi_p \circ (\gamma_f^{\mathcal{N}} \circ \gamma_g^{\mathcal{N}} - \gamma_{f \circ_{\Sigma} g}^{\mathcal{N}}) &= \gamma^{\mathcal{N}_p}_f \circ \psi_p \circ \gamma_g^{\mathcal{N}} - \gamma_{f \circ_{\Sigma} g}^{\mathcal{N}_p} \circ  \psi_p \\ \nonumber
&=  \gamma^{\mathcal{N}_p}_f  \circ \gamma^{\mathcal{N}_p}_g \circ \psi_p- \gamma_{f \circ_{\Sigma} g}^{\mathcal{N}_p} \circ  \psi_p \\ \nonumber
&=   (\gamma^{\mathcal{N}_p}_f  \circ \gamma^{\mathcal{N}_p}_g  - \gamma^{\mathcal{N}_p}_{f \circ_{\Sigma} g}) \circ  \psi_p = 0 \circ \psi_p = 0\, .
\end{align}
From Corollary \ref{identityplusmulti} we know that $\gamma_f^{\mathcal{N}} \circ \gamma_g^{\mathcal{N}}$, and hence $\gamma_f^{\mathcal{N}} \circ \gamma_g^{\mathcal{N}} - \gamma_{f \circ_{\Sigma} g}^{\mathcal{N}}$, is again an admissible map for $\mathcal{N}$. We may therefore write
\begin{equation}
\gamma_f^{\mathcal{N}} \circ \gamma_g^{\mathcal{N}} - \gamma_{f \circ_{\Sigma} g}^{\mathcal{N}} = \gamma_{h}^{\mathcal{N}} \, ,
\end{equation}
for some linear response function $h: V^{\Sigma} \rightarrow V$. Equation \eqref{trickblah1} now tells us that $\psi_p \circ  \gamma_h^{\mathcal{N}} = 0$, and so that $\gamma_h^{\mathcal{N}_p} \circ \psi_p= 0$. Next, we conclude by surjectivity of $\psi_p$ that $\gamma_h^{\mathcal{N}_p} = 0$ for all nodes $p$ of $\mathcal{N}$. Finally, we conclude by Lemma \ref{backtoinptt1} that $\gamma_{h}^{\mathcal{N}} = 0$, so that indeed $\gamma_f^{\mathcal{N}} \circ \gamma_g^{\mathcal{N}} = \gamma_{f \circ_{\Sigma} g}^{\mathcal{N}}$. This proves the proposition.
\end{proof}

\begin{remk}
Proposition \ref{third-half} and Lemma \ref{backtoinptt1} still hold if one drops the condition that $f$ and $g$ are linear. The product $f \circ_{\Sigma} g$ may then again be defined by the equation $\Gamma_f \circ \Gamma_g = \Gamma_{f \circ_{\Sigma} g}$, after which the proofs can be copied almost verbatim. This leads us to suspect that many dynamical properties of an admissible (non-linear) vector field $\Gamma_f$ for $\mathcal{F}$ may be transferred to the corresponding admissible vector field $\gamma_f^{\mathcal{N}}$ for any constructible network ${\mathcal{N}} \in \mathfrak{C}_{\mathcal{F}}$. $\hfill \triangle$
\end{remk}

\section{Proof of Theorems \ref{main1} and \ref{main2}}

\noindent We will now prove the core result of this article.
\begin{thr}\label{mainn}
Let $\mathcal{N} = (N, \Sigma)$ be a constructible network for the complete fundamental network $\mathcal{F}$, and suppose we are given the network multipliers $\Lambda^l_{i,j}: \C^{\Sigma} \rightarrow \C$ for $l \in \{1, \dots, k\}$ and $i,j \in \{1, \dots, n_l\}$ from a decomposition of $\C^{\Sigma}$ into indecomposable representations. There exist non-negative integers $(m^{\mathcal{N}}_l)_{l=1}^k$ such that for any finite dimensional complex vector space $V$ and any linear response function $f:V^{\Sigma} \rightarrow V$ with coefficients $C = (C_{\sigma})_{\sigma \in \Sigma}$, the eigenvalues of $\gamma_f^{\mathcal{N}}$ are given by those of $\Lambda^1(C)$ ($m^{\mathcal{N}}_1$ times), together with those of $\Lambda^2(C)$ ($m^{\mathcal{N}}_2$ times) up to those of $\Lambda^k(C)$ ($m^{\mathcal{N}}_k$ times). 
\end{thr}

\begin{proof}
Our proof will follow along the same lines as that of Proposition \ref{jacobsonforeigenvalues}. Given a network $\mathcal{N} \in \mathfrak{C}_{\mathcal{F}}$, let the integers $(m^{\mathcal{N}}_l)_{l=1}^k$ be as in Theorem \ref{first-half}. Out of these numbers, we construct a block diagonal matrix $\Omega^{\mathcal{N}}_f$. Specifically, the blocks on the diagonal of $\Omega^{\mathcal{N}}_f$ are given by $\Lambda^1(C)$ ($m_1^{\mathcal{N}}$ times), $\Lambda^2(C)$ ($m_2^{\mathcal{N}}$ times) up to $\Lambda^k(C)$ ($m_k^{\mathcal{N}}$ times). Note that this makes sense, as the integers $m_l^{\mathcal{N}}$ are all non-negative. If we have $m_l^{\mathcal{N}} = 0$ for some $l \in \{1, \dots, k\}$, then we simply leave out $\Lambda^l(C)$. By construction of $\Omega^{\mathcal{N}}_f$, the proof is done if we can show that $\Omega^{\mathcal{N}}_f$ and $\gamma_f^{\mathcal{N}}$ have the same eigenvalues. We will show that this is indeed the case, by using Lemma \ref{sameeig}.\\
First of all, $\Omega^{\mathcal{N}}_f$ is a square matrix of size
\begin{equation}
\sum_{l=1}^k m_l^{\mathcal{N}} n_l \dim(V) \,.
\end{equation}
By Theorem \ref{first-half}, this number equals $\#N \dim(V)$, which is equal to the dimension of $V^{\mathcal{N}}$. Hence, $\Omega^{\mathcal{N}}_f$ and $\gamma_f^{\mathcal{N}}$ are matrices of the same size. Next, we note that the trace of $\Omega^{\mathcal{N}}_f$ is given by
\begin{equation}
\tr(\Omega^{\mathcal{N}}_f) = \sum_{l =1}^k m^{\mathcal{N}}_l \tr(\Lambda^l(C)) \, ,
\end{equation}
which by Theorem \ref{first-half} is equal to $\tr(\gamma^{\mathcal{N}}_f)$. It remains to show that $\tr([\Omega^{\mathcal{N}}_f]^n) = \tr([\gamma^{\mathcal{N}}_f]^n)$ for all $n \in \N$. To this end, we write  $\circ^n_{\Sigma}f$ for $f \circ_{\Sigma} f \circ_{\Sigma} \dots \circ_{\Sigma} f$ ($n$ times). This expression is well defined, as multiplication of elements of the form $\Gamma_f$ is associative. From this it follows that $\circ_{\Sigma}$ is associative on response functions as well. Likewise, we may write  $\circ^n_{\Sigma}C$ for $C \circ_{\Sigma} C \circ_{\Sigma} \dots \circ_{\Sigma} C$ ($n$ times). It follows from Proposition \ref{second-half} that $(\Lambda^l(f))^n = (\Lambda^l(C))^n = \Lambda^l(\circ^n_{\Sigma}C) = \Lambda^l(\circ^n_{\Sigma}f)$ for all $l \in \{1, \dots, k\}$ and $n \in \N$. From this we conclude that $(\Omega^{\mathcal{N}}_f)^n = \Omega^{\mathcal{N}}_{\circ^n_{\Sigma}f}$. Likewise, Proposition \ref{third-half} tells us that $(\gamma_f^{\mathcal{N}})^n = \gamma_{\circ^n_{\Sigma}f}^{\mathcal{N}}$ for all $n \in \N$. Putting this together, we get
\begin{equation}
\tr([\Omega^{\mathcal{N}}_f]^n) = \tr(\Omega^{\mathcal{N}}_{\circ^n_{\Sigma}f}) = \tr(\gamma_{\circ^n_{\Sigma}f}^{\mathcal{N}}) =  \tr([\gamma^{\mathcal{N}}_f]^n) \,,
\end{equation}
for all $n \in \N$. Here we have used in the second step that $\tr(\Omega^{\mathcal{N}}_{g}) = \tr(\gamma_{g}^{\mathcal{N}})$ for all response functions $g$. Hence, we may conclude by Lemma \ref{sameeig} that $\gamma^{\mathcal{N}}_f$ and $\Omega^{\mathcal{N}}_f$ have the same eigenvalues, counted with algebraic multiplicity. This concludes the proof.
\end{proof}

\noindent We may now prove our main theorems.
\begin{proof}[Proof of Theorems \ref{main1} and \ref{main2}]
Let $\mathcal{F} = (\Sigma, \Sigma)$ denote the fundamental network of the complete homogeneous network with asymmetric input $\mathcal{N}$. The class of networks $\mathfrak{C}_{\mathcal{N}}$ of Theorem \ref{main1} is of course just the class of constructible networks for $\mathcal{F}$ as given by Definition \ref{ConstructibleNetworks}. Likewise, the formal maps $\Lambda^l_{i,j}$ of Theorem \ref{main1} are the coefficients of the network multipliers as given by Definition \ref{netwoorkmulti}, using a decomposition of $\C^{\Sigma}$. The different parts of Theorems \ref{main1} and \ref{main2} follow from our results in the following way.
\begin{enumerate}
\item This follows directly from Proposition \ref{cloosed} and Remark \ref{remkconstrntwrks}.
\item This is the content of Lemma \ref{indjoepie}.
\item[4.] This follows from Theorem \ref{mainn}, using the numbers $(m^{\mathcal{N}}_l)_{l=1}^k$ as given there.
\item By identifying all the nodes in $\mathcal{N}$, we obtain a network $\mathcal{N}_0$ with a single node. As $\mathcal{N}_0$ is a quotient network of $\mathcal{N} \in \mathfrak{C}_{\mathcal{N}}$, it is constructible as well. Therefore, the trace of an admissible map for $\mathcal{N}_0$ may be expressed as the sum of the traces of some of the network multipliers. From this it follows that $\Lambda^1(C) = \Lambda_{1,1}^1(C) = \sum_{\sigma \in \Sigma}C_{\sigma}$ is necessarily a network multiplier.
\item[5.] The first part follows from Theorem \ref{first-half}. If $\mathcal{P}$ is a quotient network of $\mathcal{M} \in  \mathfrak{C}_{\mathcal{N}}$, then for any fixed value of $s \in \{1, \dots, k\}$ we may choose a response function $f_s: \C^{\Sigma} \rightarrow \C$ such that $\Lambda^l(f_s)= \delta_{s,l}\Id_{n_l}$ for all $l \in \{1, \dots, k\}$. It follows that $\gamma_{f_s}^{\mathcal{P}}$ and $\gamma_{f_s}^{\mathcal{M}}$ have the eigenvalue $1$ exactly $m^{\mathcal{P}}_sn_s$ and $m^{\mathcal{M}}_sn_s$ times, respectively.  As $\gamma_{f_s}^{\mathcal{P}}$ is the restriction of $\gamma_{f_s}^{\mathcal{M}}$ to some robust synchrony space, we conclude that $m^{\mathcal{P}}_sn_s \leq m^{\mathcal{M}}_sn_s$. From this we see that indeed $m^{\mathcal{P}}_s \leq m^{\mathcal{M}}_s$. As every constructible network $\mathcal{M}$ contains $\mathcal{N}_0$ as a quotient, it follows that in particular $m^{\mathcal{M}}_1 \geq 0$.
\item[6.] It follows from Proposition \ref{cloosed} that $\mathcal{F} \in \mathfrak{C}_{\mathcal{N}} = \mathfrak{C}_{\mathcal{F}}$. From the fact that $m^{\mathcal{F}}_l = \dim(W_l)$ for some indecomposable representation $W_l$ of $\C^{\Sigma}$, it follows that necessarily $m^{\mathcal{F}}_l > 0$ for all  $l \in \{1, \dots, k\}$.  See also Proposition \ref{steponee33}.
\item[7.] This follows from propositions \ref{second-half} and \ref{third-half}.
\end{enumerate}
\end{proof}

\section{Acknowledgements}

\begin{minipage}{0.8\textwidth}
This work is part of the research programme \textit{Designing Network Dynamical Systems through Algebra}, which is financed by the Dutch Research Council (NWO). \\
\end{minipage}
\begin{minipage}{0.2\textwidth}% adapt widths of minipages to your needs
\hfill \includegraphics[width=1.2cm]{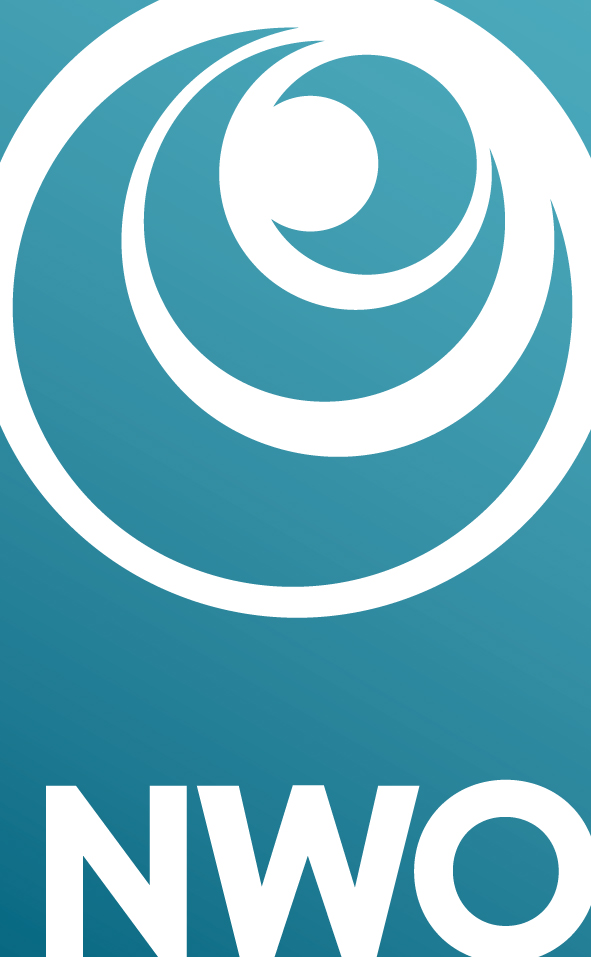}
\end{minipage}

\bibliography{networks}
\bibliographystyle{unsrt}

\end{document}